\documentclass{amsart}

\usepackage{comment}

\numberwithin{equation}{section}
\usepackage[initials,nobysame]{amsrefs}
\usepackage{amsmath,mathtools}
\usepackage{thmtools}
\usepackage{hyperref,xcolor,graphicx}
\usepackage[capitalize]{cleveref}
\usepackage{enumitem}
\setlist[enumerate,1]{label=\upshape{(\roman*)},ref=\roman*}

\usepackage{autonum}

\newtheorem{theorem}{Theorem}[section]
\newtheorem{proposition}[theorem]{Proposition}
\newtheorem{corollary}[theorem]{Corollary}
\newtheorem{lemma}[theorem]{Lemma}

\theoremstyle{definition}
\newtheorem{definition}[theorem]{Definition}

\newtheorem*{acknowledgements}{Acknowledgements}
\theoremstyle{remark}
\newtheorem{remark}[theorem]{Remark}
\newtheorem{example}[theorem]{Example}

\newcommand{\R}{\mathbf{R}}
\newcommand{\N}{\mathbf{N}}

\newcommand{\C}{\mathbf{C}}

\newcommand{\D}{\mathrm{D}}

\renewcommand{\S}{\mathbf{S}}
\renewcommand{\d}{\mathop{}\!\mathrm{d}}
\newcommand{\E}{\mathcal{E}}
\newcommand{\T}{\mathbf{T}}

\DeclareMathOperator{\Div}{\mathrm{div}}

\DeclareMathOperator{\dist}{\mathrm{dist}}

\DeclareMathOperator{\sign}{\mathrm{sgn}}

\newcommand{\mres}{\mathbin{\vrule height 1.6ex depth 0pt width
0.13ex\vrule height 0.13ex depth 0pt width 1.3ex}}

\title[Calibration energy and mean curvature flow]{Calibration energy and mean curvature flow}

\author[T.~Miura]{Tatsuya Miura}
\address[T.~Miura]{Department of Mathematics, Graduate School of Science, Kyoto University, Kitashirakawa Oiwake-cho, Sakyo-ku, Kyoto 606-8502, Japan}
\email{tatsuya.miura@math.kyoto-u.ac.jp}

\author[F.~Rupp]{Fabian Rupp}
\address[F.~Rupp]{Faculty of Mathematics, University of Vienna, Oskar-Morgenstern-Platz 1, 1090 Vienna, Austria.}
\email{fabian.rupp@univie.ac.at}

\date{\today}

\begin{document}

\begin{abstract}
We introduce the calibration energy for oriented immersions into Euclidean space, quantifying the deviation from calibrated geometry. A key property is that this energy may remain finite for infinite-volume immersions, while a null-Lagrangian structure ensures that it has the same first variation as the volume functional. We establish an exact dissipation identity for the calibration energy along proper oriented mean curvature flows in arbitrary dimensions and codimensions, under a mild local-volume bound. In fact, our result covers a class of singular calibrations and singular initial data. Even in the smooth setting, this provides a new finite variational framework for mean curvature flow beyond the finite-volume regime. As a main application, we establish a general dynamical rigidity theorem for calibrated cones in arbitrary codimension: no singular calibrated cone can be desingularized by a proper oriented mean curvature flow. Our framework further yields novel rigidity theorems for solitons and convergence for two-dimensional immortal flows.
\end{abstract}

\maketitle
\setcounter{tocdepth}{2}
\tableofcontents

\section{Introduction}

Mean curvature flow is the $L^2$-gradient flow of volume. 
For compact immersions, this immediately yields a natural Lyapunov functional, but in the noncompact setting this most basic variational interpretation degenerates, as the volume is typically infinite.
The primary goal of this paper is to replace volume by a finite functional without changing the underlying gradient-flow structure.

To that end, we introduce the \emph{calibration energy} for oriented immersions in $\R^n$, which measures the deviation of the oriented tangent plane from calibrated geometry in an averaged sense.
Calibrated geometry was introduced in the celebrated work of Harvey and Lawson \cite{MR666108} as a framework for constructing minimal submanifolds via differential forms; see also \cite{MR1045637,MR2292510}. 

The key feature of our calibration energy is a null-Lagrangian structure, resulting in the same first variation as the volume.
The present work investigates the consequences of this observation on the dynamical side.
Our main analytic result, \Cref{thm:main_calibration}, establishes an exact dissipation identity for the calibration energy along proper oriented mean curvature flows. Geometrically, this may be interpreted as follows: if the initial submanifold is at finite distance from being calibrated, then the mean curvature flow monotonically decreases the calibration energy and thus drives the submanifold towards calibrated geometry. 

This energy identity has several important consequences, the most striking one being a dynamical rigidity theorem for calibrated cones (\Cref{thm:cone_rigidity_intro}):
\emph{a singular calibrated cone can never be desingularized by mean curvature flow}.
%This gives a variational mechanism for rigidity that remains available in arbitrary codimension, where the barrier and comparison methods of hypersurface theory are generally absent.

In contrast to the volume, our calibration energy can be finite for a large class of noncompact submanifolds, thus providing a novel, finite variational framework for infinite-volume mean curvature flow. In fact, even the class of zero-energy states is rich and corresponds to calibrated submanifolds. Moreover, unlike the entropy \cite{MR2528703,MR2993752}, the calibration energy is not scale-invariant and is thus especially suited to analyzing rigidity and the asymptotic behavior of immortal solutions.

Since our approach is variational rather than comparison-based, we allow for arbitrary codimension. Higher-codimensional mean curvature flows include important classes such as Lagrangian mean curvature flow, and their general behavior is much less understood than in the hypersurface case; see, e.g., \cite{MR2483374,MR3289845}.

% old version
%In contrast to the volume, the calibration energy can be finite for a large class of noncompact submanifolds, thus providing a novel, finite variational framework for infinite-volume mean curvature flow. 
% This viewpoint is complementary to existing monotone quantities such as the entropy \cite{MR2528703,MR2993752}. 
% Whereas the entropy is scale invariant and is particularly effective for the study of finite-time singularities, the calibration energy does not share this invariance and is especially suited to analyzing rigidity and the asymptotic behavior of immortal solutions.
% Moreover, in the constant-coefficient case it inherits the translation and scaling symmetries of volume, which will be crucial for the rigidity results discussed later.

A further important aspect is that the calibration energy is not merely an ad hoc quantity: its null-Lagrangian structure indicates potential applicability beyond mean curvature flow, including to other geometric gradient flows, possibly of higher order. Indeed, in the one-dimensional case, this has been successfully implemented in our previous work \cite{MR5008349} for several higher order flows of curves, where comparison-based methods are typically unavailable, cf.\ \cite{MR4861585,miura2025embeddednessgraphicalityelasticflow}.

% old version
% A further important aspect is that the calibration energy is not merely an ad hoc quantity: its null-Lagrangian structure indicates potential applicability beyond mean curvature flow, including to other geometric gradient flows, possibly of higher order. 
% In fact, the present paper can be viewed as a substantial extension of our previous work \cite{MR5008349}, where curve flows with ``unidirectional ends'' were studied, to arbitrary dimensions and, possibly, ``multidirectional ends''.
% The curve-flow setting in \cite{MR5008349} already includes applications to fourth-order flows, demonstrating the broader scope of the method.
% For such higher-order flows, comparison-based methods are typically unavailable, which makes the development of energy-based approaches particularly important, cf.\ \cite{MR4861585,miura2025embeddednessgraphicalityelasticflow}.
% In the present work, however, we focus on mean curvature flow while allowing for arbitrary dimensions and codimensions.

\subsection{Calibration energy}

Throughout the paper, let $M^m$ be a possibly noncompact, $m$-dimensional smooth manifold (without boundary). We also fix an arbitrary $n\in\N$ with $1\leq m<n$ unless otherwise specified. 
All maps that appear will be smooth unless otherwise noted.

Following \cite{MR666108}, a \emph{calibration} $\phi\in\Omega^m(\R^n)$ is a closed $m$-form of comass $1$ on $\R^n$. Given an oriented immersion $F\colon M^m\to\R^n$, we may thus define $\phi(\tau)$ as the function $p\mapsto \phi_{F(p)}(\tau(p))$, where $\tau\colon M\to \mathbf{Gr}_m^+(\R^n)$ denotes the tangential Gauss map, see \eqref{eq:tangent_field}. Here $\mathbf{Gr}_m^+(\R^n)$ is the oriented Grassmannian, identified with the set of simple unit $m$-vectors in $\Lambda^m\R^n$.
We then define the \emph{$\phi$-calibration energy} by
\begin{equation}
    \E_\phi[F]\vcentcolon =\int_M (1-\phi(\tau))\d\mathrm{vol},
\end{equation}
where $\mathrm{vol}$ is the volume measure on $M$ induced by $F$, cf.\ \eqref{eq:dvol}.

Since $\phi$ has comass $1$, we have $\phi(\tau)\le 1$ and hence $\E_\phi[F]\geq0$, where equality holds if and only if $F$ is \emph{calibrated by $\phi$,} i.e., $\phi(\tau)=1$ on $M$. The functional $\E_\phi[F]$ measures the deviation of $F$ from being calibrated by $\phi$ and, roughly speaking, the finiteness of $\E_\phi[F]$ means that the ends of $F$ are ``asymptotically calibrated by $\phi$''. 

In the simplest case, the calibration $\phi_E\in\Omega^m(\R^n)$ is induced by a fixed oriented plane $E\in\mathbf{Gr}_m^+(\R^n)$ through $\phi_E(\xi)=\langle \xi,E\rangle$.
In this case we say that $\phi_E$ is \emph{unidirectional}.
Then the calibration energy takes the form
\[
\E_{\phi_E}[F]=\int_M (1-\langle\tau,E\rangle) \d\mathrm{vol}
=\frac12\int_M |\tau-E|^2 \d\mathrm{vol}.
\]
In the one-dimensional case $m=1$, this reduces to the \emph{direction energy} for curves, where $E$ is simply a fixed unit vector in $\R^n$.
The direction energy method for curve flows was introduced in \cite{MR5008349} (inspired by previous works in stationary problems \cite{Miura20,miura2024uniqueness}); we also note the recent work of Otto--Schubert--Westdickenberg \cite{MR4822909}, where an equivalent quantity appears under the name \emph{excess surface area} in the study of the Mullins--Sekerka flow. 
The finite direction energy regime thus requires the ends to be ``asymptotically unidirectional''. 

For curves, it can be shown that being asymptotically calibrated implies asymptotic unidirectionality, see \Cref{prop:curve_to_unidirectional} below.
However, in the case $m\geq 2$, the class of calibrations is much richer and connects to various interesting geometric concepts, including the \emph{K\"ahler} and \emph{special Lagrangian} geometries \cite{MR666108}, see \Cref{sec:finite_energy} for details. 

Any $\phi$-calibrated immersion is locally volume-minimizing, in particular minimal, by an application of Stokes' theorem. For $M$ compact (thus closed), there are no minimal immersions into $\R^n$, and, if $M$ is also oriented, the $\phi$-calibration energy reduces to the volume functional: since $\R^n$ is contractible, there exists $\Phi\in \Omega^{m-1}(\R^n)$ with $\d\Phi=\phi$. The difference to the volume satisfies (see \eqref{eq:pullback})
\begin{align}
    \int_M\phi(\tau)\d\mathrm{vol} = \int_M F^*\phi = \int_M F^*(\d\Phi) = \int_M \d (F^*\Phi) = 0,
\end{align}
by Stokes' theorem since $\partial M=\emptyset$.

In the language of the calculus of variations, $\phi(\tau)\d\mathrm{vol}$ is a \emph{null Lagrangian}.
One may thus expect that the calibration energy has the same first variation as the standard volume functional.
In particular, mean curvature flow may be viewed, formally, as the $L^2$-gradient flow of $\E_\phi$ instead of volume. 
The point of the present paper is to make this interpretation rigorous in the noncompact case.

\subsection{Energy identity}
We say that a one-parameter family $F\colon M^m\times[0,T)\to\R^n$ of immersions $F_t\vcentcolon=F(\cdot,t)$ is  a \emph{mean curvature flow} if it solves the geometric evolution equation
\begin{equation}\label{eq:MCF}
    (\partial_tF)^\perp=H,
\end{equation}
where $H$ denotes the mean curvature vector of $F_t$ and $\perp$ is the normal projection along $F_t$.
We say that a flow $F\colon M\times I\to\R^n$ of immersions on a time interval $I$ is proper if $F|_{M\times[t_1,t_2]}$ is a proper map for every $[t_1,t_2]\subset I$.

For a proper immersion $F\colon M^m\to\R^n$, let $\mu=\mu_{F}$ be (the weight measure of) the associated rectifiable $m$-varifold, i.e., the push-forward Radon measure $\mu\vcentcolon =F_\#\mathrm{vol}$ on $\R^n$.
We say that $F$ has \emph{uniformly finite local volume} if
\begin{align}\label{eq:intro_M_finite}
    \sup_{x_0\in\R^n} \mu(B_1(x_0)) < \infty,
\end{align}
where $B_r(x_0)\subset \R^n$ denotes the open ball of radius $r$ centered at $x_0$.
By properness, the uniformly finite local volume condition only controls the behavior at infinity, ruling out arbitrarily high local sheet multiplicity caused by the ends.
In particular, \eqref{eq:intro_M_finite} follows whenever $F$ has finite entropy in the sense of \cite{MR2993752}. 

Our main analytic theorem is the following exact dissipation identity.

\begin{theorem}\label{thm:main_calibration}
    Let $F\colon M^m\times[0,T)\to\R^n$ be a proper oriented mean curvature flow.
    Let $\phi\in\Omega^m(\R^n)$ be a calibration.
    Suppose that $F_0$ has uniformly finite local volume and finite $\phi$-calibration energy.
    Then, for all $0\leq t_1\leq t_2<T$,
    \begin{equation}\tag{E}\label{eq:energy_identity}
         \E_\phi[F_{t_2}]+\int_{t_1}^{t_2}\int_{M}|H|^2 \d\mathrm{vol} \d t = \E_\phi[F_{t_1}].
    \end{equation}
\end{theorem}
Crucially, no further control of the behavior at infinity other than uniformly finite local volume and finite $\phi$-calibration energy is assumed here.
In fact, we will prove more general and relative versions; see \Cref{cor:main_calibration_detailed} below, which allows for a class of singular initial data and non-smooth calibrations, and also \Cref{thm:main_calibration_detailed}, where the assumption of uniformly finite local volume is also slightly weakened.

The energy identity \eqref{eq:energy_identity} is the core result of the paper. %For interpretations in the K\"ahler and Lagrangian cases, see \Cref{rem:interpretation_kaehler,rem:interpretation_lagrangian}. 
In \Cref{subsec:cone_intro}, we present a key application on the desingularization of calibrated cones. Moreover, in \Cref{subsec:intro_application} we discuss further geometric consequences.

The mean curvature flow is often called the \emph{geometric heat equation} and is therefore expected not to outperform the classical heat equation. 
\Cref{thm:main_calibration}, however, shows an improvement at the level of the underlying gradient flow structure.
Indeed, solutions to the heat equation on the real line may start with zero Dirichlet energy and immediately have infinite energy for every $t>0$, whereas \Cref{thm:main_calibration} rules out this behavior for mean curvature flow.
Other examples of unexpectedly better behavior include pseudolocality, see \cite{MR1117150,MR2379801,MR3455158,MR3909904}, and uniqueness \cite{MR2379801,MR4554469}.

A notion of ``gradient flow calibrations'' also appears in the theory of (multiphase) mean curvature flow of compact hypersurfaces by Fischer--Hensel--Laux--Simon \cite{FHLS}. While structurally related to our calibration energy, their concept of relative entropy is specifically tailored to obtain weak-strong uniqueness results. 

% We also mention that a notion of ``gradient flow calibrations'' appears in the theory of weak solutions of (multiphase) mean curvature flow of hypersurfaces by Fischer--Hensel--Laux--Simon \cite{FHLS}. Indeed, they introduce a notion of relative entropy that is structurally related to our calibration energy and measures the deviation of a weak solution from a given moving classical solution. 
% By establishing a dissipation inequality for the relative entropy, they prove powerful weak-strong uniqueness results in the finite-volume regime. Our calibration energy, on the other hand, measures deviation from a fixed calibrated geometry and is tailored to capture the gradient flow structure of noncompact smooth solutions through an exact dissipation identity, in arbitrary codimension.

In the unidirectional and curve case ($m=1$), \Cref{thm:main_calibration} was proven under additional bounded geometry assumptions in \cite[Section 4.2]{MR5008349}. 
%These assumptions are replaced here by the weaker condition \eqref{eq:intro_M_finite} that only needs to be imposed on the initial datum and is in fact redundant for $m=1$; see \Cref{rem:m=1_volume_bound}.
Here, we allow not only for any dimension $m$, but, more importantly, for an arbitrary calibration $\phi$, thereby admitting a much broader class of ends in high (co)dimensions. 

The proof of \Cref{thm:main_calibration} differs substantially from the one-dimensional argument in \cite{MR5008349}.
Besides the essential work to handle general dimensions and calibrations, we develop a purely ambient approach and rely on a new covering-argument step.
In this way, \Cref{thm:main_calibration} even improves the one-dimensional result \cite{MR5008349} by removing bounded geometry restrictions and conditions on the tangential speed; see also \Cref{rem:m=1_volume_bound}.

\subsection{Rigidity of calibrated cones}\label{subsec:cone_intro}

Our main geometric application is a rigidity theorem for calibrated cones.
In stark contrast to the volume functional, whose zero set is trivial, the calibration energy has a rich zero set consisting of calibrated geometries.
Moreover, thanks to the admissibility of non-smooth data and calibrations, this zero-energy regime includes a broad class of singular calibrated cones.
The energy identity then yields a strong rigidity consequence: a mean curvature flow starting from such a zero-energy configuration cannot dissipate the energy and hence cannot move.
In particular, we obtain the following

\begin{theorem}[\Cref{thm:calibrated_cone_rigidity,cor:calibrated_cone_rigidity_asymptotic_expander}]\label{thm:cone_rigidity_intro}
    A locally integral calibrated conical $m$-cycle in $\R^n$ cannot be desingularized by any smooth proper oriented mean curvature flow.
    In particular, if a smooth proper oriented self-expander has a calibrated cone as a varifold-current blow-down, then the self-expander must be a union of planes.
\end{theorem}

The resolution of singular initial data is a central topic in mean curvature flow. 
For instance, Ilmanen's ``resolution of point singularities'' conjecture is recently resolved for hypersurfaces with isolated conical singularities by Chodosh--Daniels-Holgate--Schulze in dimensions $2\leq m\leq6$ \cite[Theorems~1.2--1.4]{MR4824733}.
In particular, there even exist nontrivial smooth mean curvature flows starting, in an appropriate sense, from stationary singular cones, thereby providing a dynamical desingularization of them.
%In particular, it is worth emphasizing that even if the cone itself is stationary, a nontrivial smooth mean curvature flow may emanate from it.

In codimension one, the existing theory exhibits a striking minimizing/non-minimizing dichotomy for minimal hypercones.
On the non-minimizing side, extending Ding's result \cite[Theorem~1.1]{MR4055164}, Wang proved that if a viscosity mean-convex hypercone is not area-minimizing, then it admits a smooth properly embedded strictly mean-convex self-expander asymptotic to the cone \cite[Theorem~1.2]{MR4706448}, thus allowing for a self-similar desingularization.
In fact, some cones can even be desingularized in a non self-similar way \cite{MR4530317,chen2022existence,bernstein2024existence}.
On the minimizing side, Wang also recently extended in \cite[Theorem~1.1]{MR4706448} the construction of Hardt--Simon's foliation \cite[Theorem~2.1]{MR809969} from area-minimizing hypercones with isolated singularities to arbitrary singularities.
Together with weak-set comparison (namely, the fact that a level-set flow cannot collide with the support of a stationary varifold \cite[Theorem~7.8]{MR2545035}), this implies that the level-set flow of an area-minimizing hypercone is static.
In particular, any area-minimizing singular hypercone cannot be a measure-theoretic initial datum of a smooth flow.
%verified via Brakke flows
%(See also \cite[Lemma~6.6]{Ding} for a special case of self-expanders.)

In higher codimensions, analogous resolution phenomena play an important role in Lagrangian mean curvature flow.
Joyce--Lee--Tsui constructed zero-Maslov self-expanders asymptotic to transverse pairs of Lagrangian planes in the non-minimizing regime \cite[Theorems~C--D]{MR2629511}; uniqueness of such expanders is obtained by Lotay--Neves for $m=2$ \cite{MR3190297} and Imagi--Joyce--Oliveira dos Santos for $m\geq3$ \cite{MR3482334}.
Begley--Moore \cite{MR3623231} and Ghosh \cite{ghosh2025short} used these expanders to construct smooth compact Lagrangian mean curvature flows from isolated conical singularities.

By contrast, the minimizing side in higher codimensions is much less developed, due to the lack of the one-sided barrier and comparison mechanisms.
In the special setting of Lagrangian self-expanders asymptotic to two transverse planes, Neves observed rigidity for an area-minimizing plane pair \cite[Remark~4.3(4)]{MR2299559}, while related nonexistence phenomena also follow from the classification in \cite[Theorem 4.2]{MR3482334}.

Our result establishes a general dynamical rigidity principle for calibrated conical cycles in arbitrary codimension.
The positive-time evolution may be any smooth proper oriented mean curvature flow having the cone merely as a measure-current initial trace (see \Cref{def:initial_trace} below); in particular, the flow is not assumed to be self-similar, Lagrangian, nor to satisfy any other special structure.
The argument requires neither a foliation nor a barrier construction, but is purely variational.
Thus, once the oriented initial cone is calibrated, our result applies in a substantially broader geometric and dynamical framework than the existing higher-codimensional rigidity theory for Lagrangian self-expanders asymptotic to pairs of planes.

The class of admissible calibrated cones is quite broad. 
It contains, for instance, the Simons cone and all homogeneous area-minimizing hypercones \cite[Theorem~1.9]{MR3568077}, and, more importantly, a large family of higher-codimensional oriented cones satisfying Lawlor's curvature criterion \cite[Theorem~1.1]{MR5057730}, as well as pairs of two oriented planes satisfying the angle criterion \cite{MR912845,MR974911}.

It is worth emphasizing that, even in higher codimensions, calibrated cones may nevertheless admit smooth \emph{static} desingularizations: for example, certain pairs of special Lagrangian planes admit Lawlor necks \cite{MR974911,MR3482334}.
Thus the obstruction here is purely dynamical.

\subsection{Solitons and immortal solutions}\label{subsec:intro_application}

We also provide some further geometric consequences of the energy identity \eqref{eq:energy_identity}.
The results may be summarized as follows, all under the assumption of finite calibration energy:
\begin{itemize}
    \item Uniqueness of self-expanders (\Cref{thm:self_expander}), and applications to rescaled limits of graphical flows (\Cref{thm:rescaled_limit}).
    \item Rigidity of translators (\Cref{thm:translator_constant_coeff}), and uniqueness under additional assumptions (\Cref{cor:translator_plane}).
    \item No breather theorems in the nonshrinking case (\Cref{thm:no_breather_1,thm:no_breather_2}).
    \item Asymptotic planarity of two-dimensional immortal mean curvature flows and full convergence in the unidirectional case (\Cref{thm:2d-conv}).
    \item A Pohozaev-type identity for self-shrinkers (\Cref{prop:self-shrinker_Pohozaev}) and a gap theorem in dimension two (\Cref{thm:shrinker_2D}).
\end{itemize}

These results also appear to be new, to the best of the authors’ knowledge. They will be presented in detail and compared with previous works in \Cref{sec:application}.

\subsection{Organization}
This paper is organized as follows. After reviewing some notation, \Cref{sec:energy_dissipation} is devoted to proving the main energy identity, \Cref{thm:main_calibration}, and its more general version, \Cref{thm:main_calibration_detailed}. In \Cref{sec:application}, we discuss the various applications highlighted in \Cref{subsec:cone_intro,subsec:intro_application}; in particular, in \Cref{sec:cone} we prove \Cref{thm:cone_rigidity_intro}. In \Cref{sec:finite_energy}, for the reader's convenience, we discuss the class of immersions with finite calibration energy in detail and reinterpret the energy in several important settings, including the unidirectional, K\"ahler, and special Lagrangian cases. 

\begin{acknowledgements}
    Part of this work was conducted during the thematic program \emph{Free Boundary Problems} in 2025 at the \emph{Erwin Schr\"odinger International Institute for Mathematics and Physics}, whose hospitality is gratefully acknowledged.
    TM is supported by JSPS KAKENHI Grant Numbers JP23H00085 and JP24K00532.
    FR is funded, in whole or in part, by the Austrian Science Fund (FWF), grant number \href{https://doi.org/10.55776/ESP557}{10.55776/ESP557}.
    The authors would like to thank Giovanni Alberti, Jacob Bernstein, Ulrich Menne, and Felix Schulze for helpful comments.
\end{acknowledgements}

\section{Energy dissipation identity}\label{sec:energy_dissipation}

\subsection{Preliminaries and notation}\label{subsec:notation}
For the convenience of the reader, we briefly review some basic geometric and measure-theoretic concepts and notation.

For an open set $U\subset\R^n$ and $k\in\N_0$, we write $\Omega^k(U)$ for the set of (smooth) \emph{$k$-forms on $U$} and $\Omega^k_c(U)$ for the $k$-forms with compact support. We use $\d$ for the \emph{exterior derivative}, $\wedge$ for the \emph{wedge product}, and $\iota_X$ for the \emph{interior product} with a vector field $X$. We also use $\langle\cdot,\cdot\rangle$ for the \emph{inner product of $k$-vectors}.
The \emph{comass} of $\phi\in\Omega^m(U)$ is defined by
\[
    \|\phi\|_U^*\vcentcolon=\sup\{\phi_x(\xi)\mid \xi\in\mathbf{Gr}_m^+(\R^n),\ x\in U\}.
\]
Recall that $\phi\in\Omega^m(U)$ is a calibration if $\d\phi=0$ and $\|\phi\|_U^*=1$.
For $x\in U$ we write
\[
G(\phi_x)\vcentcolon =\{\xi\in\mathbf{Gr}_m^+(\R^n)\mid \phi_x(\xi)=1\},
\]
and if $\phi$ has constant coefficients, we simply write
\[
G(\phi)\vcentcolon =\{\xi\in\mathbf{Gr}_m^+(\R^n)\mid \phi(\xi)=1\}.
\]

We will also use differential forms with locally integrable coefficients. If $\alpha\in L^1_{\mathrm{loc}}(U;\Lambda^k\R^n)$, then its \emph{distributional exterior derivative} $\d\alpha$ is defined by
\begin{equation}\label{eq:distributional_exterior_derivative}
    \langle \d\alpha,\eta\rangle
    \vcentcolon=(-1)^{k+1}\int_U\alpha\wedge\d\eta,
    \qquad
    \eta\in C_c^\infty(U;\Lambda^{n-k-1}\R^n).
\end{equation}
Thus $\d\alpha=0$ distributionally if and only if $\int_U\alpha\wedge\d\eta=0$ for every such test form $\eta$. This agrees with the usual exterior derivative whenever $\alpha$ is smooth.

Following Harvey--Lawson \cite[Definition~II.4.7]{MR666108}, a \emph{coflat calibration} of degree $m$ on $U$ is a calibration $\phi$ on $U\setminus S_\phi$, where $S_\phi\subset U$ is relatively closed and $\mathcal{H}^m(S_\phi)=0$, where $\mathcal{H}^m$ denotes the $m$-dimensional Hausdorff measure. 
By setting $\phi=0$ on $S_\phi$, we regard $\phi$ as an $m$-form with locally bounded measurable coefficients on the whole set $U$. By \cite[Lemma~II.4.8]{MR666108}, this extension satisfies
\begin{equation}\label{eq:coflat_distributionally_closed}
    \d\phi=0\qquad\text{in the distributional sense}.
\end{equation}
The case $S_\phi=\emptyset$ is the usual smooth calibration.
The calibration energy is still well defined for a coflat calibration $\phi$ and a proper oriented immersion $F\colon M^m\to\R^n$, since then $\mathrm{vol}(F^{-1}(S_\phi))=\mu_F(S_\phi)=0$.

We now turn to immersed submanifolds. Given an immersion $F\colon M^m\to\R^n$ and writing $F^*$ for the pullback along $F$, we equip $M$ with the induced metric $g_F=F^*\langle\cdot,\cdot\rangle$ so that $F$ is isometric. We also write $\nabla$ for the Levi-Civita connection on $M$ and $\D$ for the usual derivative on $\R^n$. The symbol $|\cdot|$ is used both for the norm of $k$-vectors in $\R^n$ as well as for the norm on tensors on $M$ induced by the metric $g$. Further, for the sake of brevity, we refer to $F\colon M\to\R^n$ as complete/connected/oriented, provided $(M,g_F)$ is complete/connected/oriented.

The induced curvature quantities are the \emph{(vectorial) second fundamental form} $A_F$ and its trace $H_F$, the \emph{mean curvature vector}.
The induced \emph{volume measure} and (if $F$ is oriented) the \emph{tangential Gauss map} are given by
\begin{align}\label{eq:dvol}
    \d\mathrm{vol}_F=|\partial_1F\wedge\dots\wedge\partial_mF|\d x^1\cdots\d x^m
\end{align}
and
\begin{equation}\label{eq:tangent_field}
    \tau_F\vcentcolon=\frac{\partial_1F\wedge\cdots\wedge\partial_mF}{|\partial_1F\wedge\cdots\wedge\partial_mF|}
\end{equation}
in (positive) local coordinates $(x^1,\dots,x^m)$.
We will usually drop the $F$-dependence in the subscript when the respective immersion is clear from the context. Similarly, whenever there is no ambiguity, we will omit the $t$-dependence along families of immersions.

It follows from \eqref{eq:dvol}--\eqref{eq:tangent_field} that, provided that the integrals exist, we have
\begin{align}\label{eq:pullback}
    \int_M\phi(\tau)\d\mathrm{vol}=\int_MF^*\phi.
\end{align}
For convenience and by abuse of notation, if $B\subset\R^n$ and $\zeta\colon M\to\R$, we write
\begin{align}\label{eq:integral_notation}
    \int_B\zeta\d\mu\vcentcolon=\int_{F^{-1}(B)}\zeta\d\mathrm{vol},
\end{align}
whenever the right hand side exists.

For a proper oriented immersion $F\colon M^m\to\R^n$, let $\T_F$ denote the associated $m$-current in $\R^n$ defined as usual by
\[
    \T_F(\omega)\vcentcolon=\int_MF^*\omega,
    \qquad \omega\in\Omega_c^m(\R^n).
\]
Then $\T_F$ is a locally integral $m$-cycle, cf.\ \cite{MR257325}. In general, if $\T$ is a locally integral $m$-current in $\R^n$, then we can write $\T=\vec{T}\,\|\T\|$ by local integrality, where $\|\T\|$ is the mass of $\T$ and  $\vec{T}$ is a simple unit $m$-vector $\|\T\|$-a.e.
That way, the action of a current also extends to forms with non-smooth coefficients: If $U\subset\R^n$ is open and $\alpha$ is an $m$-form with compact support in $U$ whose coefficients belong to $L^1_{\mathrm{loc}}(\|\T\|\mres U)$, we may define (by extending $\alpha$ to be zero outside $U$)
\[
\T(\alpha) \vcentcolon= \int_{\R^n}\alpha(\vec{T})\d\|\T\|.
\]

Now we introduce a suitable weak notion of initial data in our setting. Along an oriented flow $F_t$, we often use the shorthand $\mu_t\vcentcolon=\mu_{F_t}$ and $\T_t\vcentcolon=\T_{F_t}$ for the associated measures and currents.

\begin{definition}[Oriented initial trace]\label{def:initial_trace}
    Let $F\colon M^m\times(0,T)\to\R^n$ be a proper oriented mean curvature flow. 
    We say that a pair $(\mu_0,\T_0)$ of a Radon measure $\mu_0$ on $\R^n$ and an $m$-current $\T_0$ in $\R^n$ is an \emph{oriented initial trace} of $F$ if, as $t\to0^+$,
    \begin{align}
        \mu_t\rightharpoonup\mu_0 \quad \text{as Radon measures},\quad 
        \T_t\rightharpoonup \T_0 \quad \text{as currents}.
        \label{eq:initial_trace}
    \end{align}
\end{definition}

Since $\|\T_t\|\leq\mu_t$ and since $\partial\T_t=0$, the convergence $\mu_t\rightharpoonup\mu_0$ yields uniform local mass bounds for $\T_t$. 
Hence, by the Federer--Fleming compactness theorem \cite[Theorem 27.3]{MR756417} and by $\T_t\rightharpoonup\T_0$, the limit $\T_0$ is a locally integral $m$-cycle.
Moreover, $\|\T_0\|\leq\mu_0$.
Clearly, any proper oriented flow $F\colon M^m\times[0,T)\to\R^n$ smooth up to $t=0$ has the canonical and unique oriented initial trace $(\mu_{F_0},\T_{F_0})$.

Let $U\subset\R^n$ be open, let $\phi$ be a coflat calibration on $U$ (extended by zero on $S_\phi$).
For an oriented initial trace $(\mu_0,\T_0)$, if $\mu_0$ is concentrated on $U$, we define
\begin{equation}\label{eq:current_calibration_energy}
    \E_\phi[\mu_0,\T_0]
    \vcentcolon=\int_{\R^n}\d\bigl(\mu_0-\phi(\vec{T}_0)\|\T_0\|\bigr).
\end{equation}
Here we note that $\mu_0-\phi(\vec{T}_0)\|\T_0\|=(\mu_0-\|\T_0\|)+(1-\phi(\vec{T}_0))\|\T_0\|$ defines a nonnegative Radon measure.
Since $\mu_0$ is concentrated on $U$, so is $\|\T_0\|$, and thus \eqref{eq:current_calibration_energy} is well-defined even though $\phi$ may not be defined outside of $U$.
In particular, $\E_\phi[\mu_0,\T_0]=0$ holds if and only if
\begin{equation}\label{eq:calibrated_initial_trace}
        \mu_0=\|\T_0\|,
        \qquad
        \phi(\vec{T}_0)=1\quad\|\T_0\|\text{-a.e.,}
\end{equation}
so that $\T_0$ is calibrated by $\phi$.
If the oriented initial trace is induced by a proper oriented immersion $F_0\colon M^m\to\R^n$, then, as Radon measures on $\R^n$, 
$$\mu_{F_0}-\phi(\vec{T}_{F_0})\|\T_{F_0}\| = (F_0)_\#\big( (1-\phi(\tau_{F_0}))\mathrm{vol}_{F_0}\big),$$ 
cf.\ \eqref{eq:pullback}, so that
\[
    \E_\phi[\mu_{F_0},\T_{F_0}]=\E_\phi[F_0],
\]
even when oppositely oriented coincident sheets cancel in the associated current. 

\subsection{A general version of the energy identity}\label{subsec:energy_identity}
We now present a more general version of \Cref{thm:main_calibration}.
We work in an open subset in order to admit the widest possible class of admissible ends.

\begin{theorem}\label{thm:main_calibration_detailed}
    Let $F\colon M^m\times(0,T)\to\R^n$ be a proper oriented mean curvature flow.
    Let $U\subset\R^n$ be an open set such that
    \begin{equation}\label{eq:inclusion}
        F_t(M)\subset U\qquad\text{for all }t\in(0,T),
    \end{equation}
    and let $\phi$ be a coflat calibration of degree $m$ on $U$.
    Suppose that
    \begin{equation}\label{eq:local_calibration_energy_bound}
        \sup_{0<t\leq T'}\sup_{x_0\in\R^n}\int_{B_1(x_0)}(1-\phi(\tau))\d\mu_t<\infty
        \qquad\text{for all }T'<T,
    \end{equation}
    and  $F$ has an oriented initial trace $(\mu_0,\T_0)$ with $\mu_0$ concentrated on $U$.
    If $\E_\phi[\mu_0,\T_0]<\infty$, then, for all $0\leq t_1\leq t_2<T$,
    \begin{equation}\label{eq:initial_trace_energy_identity}
        \E_\phi[F_{t_2}]+\int_{t_1}^{t_2}\int_M|H|^2\d\mathrm{vol}\d t
        =\E_\phi[F_{t_1}],
    \end{equation}
    where we interpret $\E_\phi[F_0]\vcentcolon=\E_\phi[\mu_0,\T_0]$.
    In particular,
    \begin{equation}\label{eq:initial_trace_energy_convergence}
        \lim_{t\to0^+}\E_\phi[F_t]=\E_\phi[\mu_0,\T_0].
    \end{equation}
\end{theorem}
It is in fact enough that the initial trace is attained in a sequential sense, see \Cref{prop:seq_initial_trace} below.

To prove \Cref{thm:main_calibration_detailed}, we first compute the general localized evolution of the calibration energy.
Hereafter we use the following notation:
for an $m$-form $\phi$, a map $W\colon M^m\to\R^n$, and an immersion $G\colon M^m\to U$, we define $\beta^\phi_{W,G}$ almost everywhere by
\begin{align}\label{eq:def_beta}
    (\beta^\phi_{W,G})_p(X_1,\dots,X_{m-1})\vcentcolon=\phi_{G(p)}(W(p),\d G_p(X_1),\dots,\d G_p(X_{m-1})).
\end{align}
The values on $G^{-1}(S_\phi)$ are irrelevant as $\mathrm{vol}_G(G^{-1}(S_\phi))=0$.

\begin{lemma}\label{lem:loc_calib_evolution}
Let $F\colon M\times(0,T)\to\R^n$ be a proper oriented flow with normal velocity $V\vcentcolon=(\partial_tF)^\perp$. 
Let $U\subset\R^n$ be an open set and let $\phi$ be a coflat calibration of degree $m$ on $U$, with $F_t(M)\subset U$ for all $t\in(0,T)$.
Let $\tilde\eta\in C_c^\infty(\R^n)$ and set $\eta\vcentcolon=\tilde\eta\circ F$. Then the map
\[
    (0,T)\ni t\mapsto\int_M(1-\phi(\tau))\eta\d\mathrm{vol}
\]
is locally absolutely continuous and, for almost every $t$, we have
\begin{align}
    &\partial_t\int_M(1-\phi(\tau))\eta\d\mathrm{vol}+\int_M\langle H,V\rangle\eta\d\mathrm{vol}\\
    &\quad=\int_M(1-\phi(\tau))\langle\D\tilde\eta\circ F_t,V\rangle\d\mathrm{vol}+\int_MF_t^*(\d\tilde\eta)\wedge\beta,
    \label{eq:localized_calibration_evolution}
\end{align}
where $\beta\vcentcolon=\beta^\phi_{V(\cdot,t),F_t}$ as defined in \eqref{eq:def_beta}.
\end{lemma}

\begin{proof}
    We first suppose that $\phi$ is a smooth closed form on $U$ (with any comass). 
    We denote the tangential speed by $\xi(p,t)\vcentcolon=(\partial_tF)^\top(p,t)=(F_t)_*X_\xi(p,t)$, where $X_\xi(\cdot,t)$ is the corresponding vector field on $M$.
    The variation of the volume form is classically given by
    \begin{align}\label{eq:time_derivative_volume_form}
        \partial_t\d\mathrm{vol}=(\Div X_\xi-\langle H,V\rangle)\d\mathrm{vol}.
    \end{align}
    Hence, after differentiating under the integral and using properness, we have
    \begin{align}\label{eq:09-04_03}
        \partial_t\int_M\eta\d\mathrm{vol}=\int_M\big(\Div X_\xi\eta-\langle H,V\rangle\eta+\langle\D\tilde\eta\circ F,\partial_tF\rangle\big)\d\mathrm{vol}.
    \end{align}
    Since $\partial_tF=\xi+V$ and
    $
        \langle\D\tilde\eta\circ F,\xi\rangle
        =g(X_\xi,\nabla\eta),
    $
    the tangential part of the last term cancels with the first term after integration by parts. Therefore,
    \begin{align}\label{eq:29-05_01}
        \partial_t\int_M\eta\d\mathrm{vol}=\int_M\big(-\langle H,V\rangle\eta+\langle\D\tilde\eta\circ F,V\rangle\big)\d\mathrm{vol}.
    \end{align}

    For the term involving $\phi$, we first observe that \eqref{eq:pullback} yields
    \begin{align}\label{eq:09-04_04}
        \int_M\phi(\tau)\eta\d\mathrm{vol}=\int_MF_t^*(\tilde\eta\wedge\phi).
    \end{align}
    Given $t\in(0,T)$, let $i_t\colon M\to M\times(0,T)$, $i_t(p)=(p,t)$. We note that $F_t=F\circ i_t$ and consequently
    \begin{align}\label{eq:0523_01}
        F_t^*(\tilde\eta\wedge\phi)=i_t^*(F^*(\tilde\eta\wedge\phi)).
    \end{align}
    Now, for a general $m$-form $\alpha$ on $M\times(0,T)$, by the flow definition of the Lie derivative, we obtain
    \begin{align}\label{eq:09-04_05}
        \partial_t(i_t^*\alpha)=i_t^*(\mathcal{L}_{\partial_t}\alpha)=i_t^*(\iota_{\partial_t}(\d\alpha))+\d(i_t^*(\iota_{\partial_t}\alpha)),
    \end{align}
    where we used Cartan's formula and the commutation of the exterior derivative with the pullback in the last step.
    We apply this to $\alpha\vcentcolon=F^*(\tilde\eta\wedge\phi)$. Since $\phi$ is closed, we compute
    \begin{align}
        \d\alpha=F^*(\d\tilde\eta\wedge\phi).
        \label{eq:09-04_01}
    \end{align}
    Take positive normal coordinates $p=(x^1,\dots, x^m)$ on $M$ and write $\tau_i\vcentcolon = \partial_i F(p,t)$ for $i=1,\dots,m$. Then we evaluate
    \begin{align}
        (i_t^* (\iota_{\partial_t} \d\alpha))_p(\partial_1,\dots,\partial_m) &= (\iota_{\partial_t}F^* (\d\tilde\eta\wedge\phi))_{i_t(p)}(\partial_1,\dots,\partial_m) \\
        &= (F^*(\d\tilde\eta\wedge\phi))_{(p,t)}(\partial_t, \partial_1,\dots,\partial_m) \\
        &= (\d\tilde\eta\wedge\phi)_{F(p,t)} (\partial_t F(p,t), \tau_1,\dots,\tau_m) \\
        & = (\d\tilde\eta\wedge\phi)_{F(p,t)} (V(p,t), \tau_1,\dots,\tau_m),\label{eq:10-04_03}
    \end{align}
    since $\xi(p,t),\tau_1,\dots,\tau_m$ are linearly dependent. Further, by definition of the wedge product,
    \begin{align}
       &(\d\tilde\eta\wedge\phi)_{F(p,t)} (V(p,t), \tau_1,\dots,\tau_m) \\
       &\qquad = (\d\tilde \eta)_{F(p,t)}(V)\,\phi_{F(p,t)}(\tau_1,\dots,\tau_m) \\
       &\qquad \qquad + \sum_{j=1}^m (-1)^{j} (\d\tilde \eta)_{F(p,t)}(\tau_j)\,\phi_{F(p,t)}(V(p,t), \tau_1,\dots, \widehat{\tau}_j,\dots,\tau_m) \\
       &\qquad = \langle \mathrm{D}\tilde\eta \circ F_t, V\rangle \phi(\tau)\vert_{(p,t)} - (F_t^*\d\tilde\eta) \wedge \beta\vert_{(p,t)} .\label{eq:10-04_04}
    \end{align}
    Again, as in \eqref{eq:09-04_03}, we may differentiate under the integral in \eqref{eq:09-04_04} and use 
    \eqref{eq:0523_01}, \eqref{eq:09-04_05}, \eqref{eq:10-04_03}, and \eqref{eq:10-04_04}
    to conclude
    \begin{align}\label{eq:09-04_02}
        \partial_t \int_M F_t^*(\tilde\eta\wedge\phi) = \int_M \langle \mathrm{D}\tilde\eta\circ F,V\rangle \phi(\tau)\d\mathrm{vol} - \int_M F_t^*(\d \tilde\eta) \wedge \beta.
    \end{align}
    Combining this with \eqref{eq:29-05_01} proves the assertion in the smooth case.

    We now consider a coflat calibration $\phi$. Fix any $0<t_1<t_2<T$.
    As $F$ is proper, the set $K:=(F|_{M\times[t_1,t_2]})^{-1}(\operatorname{supp}\tilde\eta)$ is compact.
    Hence $F(K)$ is also compact in $U$.
    Let $U'$ be a bounded open set such that $F(K)\subset U'$ and $\overline{U'}\subset U$.
    Let $\phi_\varepsilon:=\rho_\varepsilon*\phi$ be the coefficient-wise convolution of $\phi$ (extended by zero to $\R^n$) with the standard mollifier $\rho_\varepsilon$.
    By \eqref{eq:coflat_distributionally_closed} we have $\d\phi_\varepsilon=0$ on $U'$ for any small $\varepsilon>0$.
    Moreover, for any $x\in\R^n$ and $\xi\in\mathbf{Gr}_m^+(\R^n)$ we have
    \[
    (\phi_\varepsilon)_x(\xi)=\int_{\R^n} \rho_\varepsilon(y)\phi_{x-y}(\xi) \d y \leq \int_{\R^n} \rho_\varepsilon(y) \d y = 1,
    \]
    and $\phi_\varepsilon \to \phi$ as $\varepsilon\to0^+$ locally smoothly on $U\setminus S_\phi$.
    As the above computation in the smooth case only uses the form on a neighborhood of $F(K)$, we can apply the identity \eqref{eq:localized_calibration_evolution} to the smooth form $\phi_\varepsilon$ on $U'$ and integrate in time over $[t_1,t_2]$. 
    Since $\mu_t(S_\phi)=0$ for every $t$, Fubini's theorem implies that $F^{-1}(S_\phi)$ has zero spacetime measure. 
    All terms are uniformly bounded on the compact spacetime set $K$ (in particular, $|\phi_\varepsilon(\tau)|\leq1$ and $|\beta^{\phi_\varepsilon}_{V,F}|\leq|V|$), and thus the dominated convergence theorem allows us to take the limit $\varepsilon\to0^+$. 
    This gives the time-integrated form of \eqref{eq:localized_calibration_evolution} on any subinterval $[t_1,t_2]\subset(0,T)$, yielding local absolute continuity and the almost-everywhere differential identity.
\end{proof}

We then obtain a key geometric estimate for the last term of the localized evolution law.

\begin{lemma}\label{lem:loc_calib_remainder_estimate}
    Let $U\subset\R^n$ be an open set, let $\phi$ be a coflat calibration of degree $m$ on $U$, and let $F\colon M\to\R^n$ be an oriented immersion with $F(M)\subset U$. Suppose that $V\colon M\to\R^n$ is normal along $F$. Then, for any $\tilde\eta\in C^\infty(\R^n)$, we have the pointwise estimate
    \begin{align}
        \pm F^*(\d\tilde\eta)\wedge\beta
        \leq|\D\tilde\eta\circ F||V|\sqrt{1-\phi(\tau)^2}\d\mathrm{vol},
    \end{align}
    where $\beta\vcentcolon=\beta^\phi_{V,F}$ as defined in \eqref{eq:def_beta}.
\end{lemma}

\begin{proof}
    Note first that our zero-extension convention for $\phi$ makes the assertion trivial on $S_\phi$ (but the argument below also works there).
    
    Take positive normal coordinates $(x^1,\dots,x^m)$. If $V=0$ there is nothing to prove, so we assume $V\neq 0$ and set $\nu\vcentcolon = \frac{V}{|V|}$. Evaluating at $(\partial_1,\dots,\partial_m)$, we find
    \begin{align}
        &(F^*(\d\tilde\eta)\wedge \beta) (\partial_1,\dots,\partial_m)  =  |V| \sum_{j=1}^m(-1)^{j-1}\d\tilde\eta(\partial_j F) \phi(N_j)\label{eq:09-04_07}
    \end{align}
    where $N_j\vcentcolon = \nu\wedge \partial_1 F\wedge\dots\wedge \widehat{\partial_jF}\wedge \dots\wedge \partial_mF$. 
    Since, up to a sign, $N_j$ is given by replacing the unit vector $\partial_j F$ in $\tau=\partial_1 F\wedge\dots\wedge\partial_mF$ with the other unit vector $\nu$ such that $\langle \nu,\partial_jF\rangle=0$ for all $j$, we deduce that $N_1,\dots,N_m$ and $N_{m+1}\vcentcolon=\tau$ are mutually orthogonal unit $m$-vectors, i.e., $\langle N_i,N_j \rangle=\delta_{ij}$. We consider the $(m+1)$-dimensional subspace $P$ spanned by $\nu,\partial_1F,\dots,\partial_mF$. Since any $m$-vector in $\Lambda^m P$ is simple (by Hodge duality), %(since the Hodge star operator $\star\colon \Lambda^m(P)\to P$ is an isomorphism), 
    so is $\sum_{j=1}^{m+1}\phi(N_j) N_j$. Hence, we have the estimate
    \begin{align}
         \sum_{j=1}^{m+1} \phi(N_j)^2 = \phi\Big(\sum_{j=1}^{m+1} \phi(N_j) N_j\Big) \leq \|\phi\|_U^* \Big(\sum_{j=1}^{m+1} \phi(N_j)^2\Big)^{\frac{1}{2}}.
    \end{align}
    We conclude, using $\|\phi\|_U^*\leq1$, that $\sum_{j=1}^{m+1} \phi(N_j)^2\leq1$, i.e.,
    \begin{align}
        \sum_{j=1}^m \phi(N_j)^2 \leq 1-\phi(\tau)^2.\label{eq:09-04_08}
    \end{align}
    Similarly, we have $\sum_{j=1}^m \langle \mathrm{D}\tilde\eta\circ F, \partial_jF\rangle^2 \leq |\mathrm{D}\tilde\eta\circ F|^2.$
    Now, estimating \eqref{eq:09-04_07} by using $\d\tilde\eta(\partial_jF)=\langle \mathrm{D}\tilde\eta\circ F, \partial_jF\rangle$, the Cauchy--Schwarz inequality, and \eqref{eq:09-04_08}, we conclude
    \begin{align}
         |(F^*(\d\tilde\eta)\wedge \beta)(\partial_1,\dots,\partial_m)|^2 &\leq |V|^2 \Big(\sum_{j=1}^m \langle \mathrm{D}\tilde\eta\circ F, \partial_jF\rangle^2\Big) \Big(\sum_{j=1}^m \phi(N_j)^2\Big) \\
         &\leq |\mathrm{D}\tilde\eta\circ F|^2|V|^2 \big(1-\phi(\tau)^2\big).
    \end{align}
    The statement follows.
\end{proof}

Now we prove the energy identity in its general form.

\begin{proof}[Proof of \Cref{thm:main_calibration_detailed}]
   Let $T'<T$. For notational simplicity, for $t\in [0,T']$ we set
    \[
        % \nu_t\vcentcolon=(1-\phi(\tau))\mu_t \quad(t>0),
        % \qquad
        % \nu_0\vcentcolon=\mu_0-\phi(\vec{T}_0)\|\T_0\|.
        \nu_t\vcentcolon=\mu_t-\phi(\vec{T}_t)\|\T_t\|.
    \]
    Observe that \eqref{eq:local_calibration_energy_bound} and an elementary covering argument imply that
    \begin{align}
        \sup_{0<t\leq T'}\sup_{x_0\in\R^n}\nu_t(B_R(x_0))<\infty \qquad \text{for every $R>0$}.
        \label{eq:R-unif-bounded-local-volume}
    \end{align}
    
    We first show that $\nu_t\rightharpoonup\nu_0$ as Radon measures as $t\to0^+$.
    To this end it suffices to show that, for every $0\leq\zeta\in C_c^\infty(\R^n)$,
    \begin{equation}\label{eq:initial_trace_local_defect_convergence}
        \lim_{t\to0^+}\int_{\R^n}\zeta\d\nu_t = \int_{\R^n}\zeta\d\nu_0.
    \end{equation}
    As $\mu_0$ is concentrated on $U$, for any given $\varepsilon>0$, we can take a bounded open
    neighborhood $O$ of $(\R^n\setminus U)\cap\operatorname{supp}\zeta$ such that $\mu_0(O)<\varepsilon$ and $\mu_0(\partial O)=0$, and also take $\rho\in C_c^\infty(U;[0,1])$ such that $\rho=1$ on $\operatorname{supp}\zeta\setminus O$.
   By \Cref{def:initial_trace} and the discussion thereafter, we have uniform mass bounds for the family $\T_t$ as $t\to 0^+$. Hence, by \eqref{eq:initial_trace} and \cite[Theorem 31.2]{MR756417}, the locally flat distance between $\T_{t}$ and $\T_0$ goes to zero. Let $W\subset\subset U$ be a bounded open set with $\operatorname{supp}\rho \subset W$. It follows that there exist $m$-currents $\mathbf A_t$ and $(m+1)$-currents $\mathbf{B}_t$ such that $\T_t-\T_0 = \mathbf{A}_t+\partial \mathbf{B}_t$ and, as $t\to 0^+$,
    \begin{align}\label{eq:flat_distance}
        \| \mathbf A_t\|_{W} + \| \mathbf B_t\|_{W} \to 0,
    \end{align}
    where $\|\cdot\|_W$ denotes the mass of currents in a bounded open set $W$. Let $\phi_\delta$ be the coefficient-wise convolution of $\phi$ as in the proof of \Cref{lem:loc_calib_evolution}.
    It follows that, 
    \begin{align}\label{eq:31_08_26}
        \big|(\T_t -\T_0)(\rho\zeta\phi_\delta)\big| &= \big|\mathbf{A}_t(\rho\zeta\phi_\delta)+\mathbf{B}_t(\rho\d\zeta\wedge \phi_\delta) + \mathbf{B}_t(\zeta \d\rho\wedge \phi_\delta)\big| \\
        &\leq 3 \Vert \rho\Vert_{C^1}\Vert \zeta\Vert_{C^1}(\| \mathbf A_t\|_{W} + \| \mathbf B_t\|_{W}).
    \end{align}
    Since $\T_t$ is locally integral and $\mathcal{H}^m(S_\phi)=0$, we have $\|\T_t\|(S_\phi)=0$ for all $t\geq 0$. Together with the fact that $\phi_\delta\to\phi$ as $\delta\to 0^+$ locally smoothly on $U\setminus S_\phi$, the dominated convergence theorem implies $\T_t(\rho\zeta\phi_\delta)\to \T_t (\rho\zeta\phi)$ for all $t\geq 0$. Sending first $\delta\to 0^+$ in \eqref{eq:31_08_26} and then $t\to0^+$, \eqref{eq:flat_distance} implies $\T_t(\rho\zeta\phi)\to \T_0(\rho\zeta\phi)$ as $t\to 0^+$.
    Consequently, using \eqref{eq:initial_trace}, we find
    \begin{align}
        \lim_{t\to0^+}\int_{\R^n}\rho\zeta\d\nu_t &= \lim_{t\to0^+}\left(\int_{\R^n}\rho\zeta \d\mu_t - \T_t(\rho\zeta\phi)\right) \\
        &= \int_{\R^n}\rho\zeta \d\mu_0 - \T_0(\rho\zeta\phi) = \int_{\R^n}\rho\zeta\d\nu_0.
    \end{align}
    Since $\nu_t\leq 2\mu_t$ for $t\geq0$, the remaining terms are bounded by
    \begin{align}
        \limsup_{t\to0^+}\left| \int_{\R^n}(1-\rho)\zeta\d\nu_t - \int_{\R^n}(1-\rho)\zeta\d\nu_0\right| \leq \limsup_{t\to0^+}2\|\zeta\|_\infty\bigl(\mu_t(O)+\mu_0(O)\bigr),
    \end{align}
    which is at most $4\|\zeta\|_\infty\varepsilon$.
    Letting $\varepsilon\to0^+$ proves \eqref{eq:initial_trace_local_defect_convergence}.

    Now we turn to the proof of the main energy identity.
    Fix $x_0\in \R^n$ and $R>0$, and take $\tilde\gamma \in C_c^\infty(\R^n)$ such that $\chi_{B_R}(x_0)\leq \tilde
    \gamma\leq \chi_{B_{2R}(x_0)}$ and $\|\D\tilde\gamma\|_\infty\leq C/R$ for some $C>0$, and define $\gamma\vcentcolon =\tilde\gamma\circ F$.
    Hereafter $C=C(m,n)>0$ may change line by line. Moreover, we let $\tilde\eta = \tilde\gamma^2$ and $\eta =\tilde\eta\circ F = \gamma^2$. Integrating the time-derivative formula in \Cref{lem:loc_calib_evolution} with $V=H$ on $[s,t]\subset(0,T']$, we obtain
    \begin{align}
        &\int_M(1-\phi(\tau))\eta\d\mathrm{vol}\Big|_t
        +\int_s^t\int_M|H|^2\eta\d\mathrm{vol}\d r\\
        &=\int_M(1-\phi(\tau))\eta\d\mathrm{vol}\Big|_s
        +\int_s^t\int_M(1-\phi(\tau))
        \langle\D\tilde\eta\circ F_r,H\rangle\d\mathrm{vol}\d r\\
        &\quad+\int_s^t\int_MF_r^*(\d\tilde\eta)\wedge\beta\d r,
        \label{eq:localized_initial_trace_identity}
    \end{align}
    where $\beta=\beta^\phi_{H(\cdot,r),F_r}$. By
    \Cref{lem:loc_calib_remainder_estimate}, the estimate
    $|\D\tilde\eta|\leq C\tilde\gamma/R$, and
    \[
        1-\phi(\tau)+\sqrt{1-\phi(\tau)^2}
        \leq C\sqrt{1-\phi(\tau)},
    \]
    the modulus of the last two terms in
    \eqref{eq:localized_initial_trace_identity} is bounded by
    \begin{align}
        &\frac{C}{R}\int_s^t\int_M
        \sqrt{1-\phi(\tau)}|H|\gamma\d\mathrm{vol}\d r \label{eq:localized_initial_trace_error_0}\\
        &\qquad\leq\frac{C}{R^2}\int_s^t
        \nu_r(B_{2R}(x_0))\d r
        +\frac12\int_s^t\int_M|H|^2\gamma^2\d\mathrm{vol}\d r. \label{eq:localized_initial_trace_error}
    \end{align}
    Letting $s\to0^+$ in the resulting inequality, using $\nu_s\rightharpoonup\nu_0$ for the first term of the right hand side and monotone convergence for the spacetime terms, we obtain
    \begin{align}
        &\nu_t(B_R(x_0))
        +\frac12\int_0^t\int_{B_R(x_0)}|H|^2\d\mu_r\d r\\
        &\qquad\leq \E_\phi[\mu_0,\T_0]
        +\frac{C}{R^2}\int_0^t\nu_r(B_{2R}(x_0))\d r.
        \label{eq:local_initial_trace_energy_estimate}
    \end{align}

    For $t\in[0,T']$ and $R>0$, we define the functions
    \begin{align}
        \E(t,R)&\vcentcolon=
        \sup_{x_0\in\R^n}\nu_t(B_R(x_0)),\label{eq:def_local_calibration_energy}\\
        h(t,R)&\vcentcolon=
        \sup_{x_0\in\R^n}\int_0^t\int_{B_R(x_0)}|H|^2\d\mu_r\d r.
        \label{eq:def_local_curvature_dissipation}
    \end{align}
    By \eqref{eq:R-unif-bounded-local-volume} and \eqref{eq:local_initial_trace_energy_estimate}, together with $\E(0,R)\leq \E_\phi[\mu_0,\T_0]<\infty$ and $h(0,R)=0$, both $\E(\cdot,R)$ and $h(\cdot,R)$ are bounded on $[0,T']$ for any fixed $R>0$.
    Taking the supremum in \eqref{eq:local_initial_trace_energy_estimate} separately for the two terms on the left gives, for $t\in[0,T']$,
    \begin{align}
        \E(t,R)+\frac12h(t,R)
        \leq2\E_\phi[\mu_0,\T_0]+\frac{C}{R^2}\int_0^t\E(r,2R)\d r.
        \label{eq:covering_calibration_inequality}
    \end{align}
    Since any ball of radius $2R$ in $\R^n$ can be covered by $\Gamma=\Gamma(n)$ balls of radius $R$, we have
    \begin{align}\label{eq:covering}
        \E(r,2R)\leq\Gamma\E(r,R).
    \end{align}
    Dropping the nonnegative $h$-term in \eqref{eq:covering_calibration_inequality}, using \eqref{eq:covering}, and applying a version of Gr\"onwall's inequality (see, e.g., \cite[Appendix B.2.k]{MR2597943}), we find
    \begin{equation}
        \E(t,R)
        \leq2\E_\phi[\mu_0,\T_0]\exp\left(\frac{C\Gamma t}{R^2}\right).
        \label{eq:covering_calibration_Gronwall}
    \end{equation}
    Substituting this estimate into
    \eqref{eq:covering_calibration_inequality} also yields
    \begin{equation}
        h(t,R)\leq4\E_\phi[\mu_0,\T_0]\exp\left(\frac{C\Gamma t}{R^2}\right).
        \label{eq:covering_curvature_bound}
    \end{equation}
    Hence, thanks to the monotone convergence theorem, sending $R\to\infty$ gives
    \begin{align}
        \E_\phi[F_t]&\leq2\E_\phi[\mu_0,\T_0],\label{eq:global_calibration_energy_bound}\\
        \int_0^t\int_M|H|^2\d\mathrm{vol}\d r
        &\leq4\E_\phi[\mu_0,\T_0].
        \label{eq:global_curvature_dissipation_bound}
    \end{align}
    In particular, both terms are bounded on $[0,T']$.

    Now, letting $s\to0^+$ in \eqref{eq:localized_initial_trace_identity} with $\nu_s\rightharpoonup\nu_0$, revisiting \eqref{eq:localized_initial_trace_error_0}, and using the spacetime integrability of $|H|\sqrt{1-\phi(\tau)}$ by \eqref{eq:global_calibration_energy_bound} and \eqref{eq:global_curvature_dissipation_bound}, we conclude that, as $R\to\infty$, the second and third term on the right hand side of \eqref{eq:localized_initial_trace_identity} vanish.
    The remaining terms converge as desired thanks to the property $\eta\to 1$ as $R\to\infty$ and the dominated convergence theorem, yielding
    \begin{equation}
        \E_\phi[F_t]+\int_0^t\int_M|H|^2\d\mathrm{vol}\d r
        =\E_\phi[\mu_0,\T_0].
    \end{equation}
    Subtracting this at two positive times gives \eqref{eq:initial_trace_energy_identity} for arbitrary $0\leq t_1\leq t_2<T'$. 
    By integrability we have $\int_0^t\int_M|H|^2\d\mathrm{vol}\d r\to0$ as $t\to0^+$, proving \eqref{eq:initial_trace_energy_convergence}.
\end{proof}

In order to prove \Cref{thm:main_calibration}, we next observe that uniformly finite local volume propagates from a measure trace.
The argument is essentially the same as in the embedded codimension-one case in \cite[Proposition~4.9]{Ecker}; see also \cite[Theorem~3.7]{Brakke}.

\begin{lemma}\label{lem:uniform_bounded_local_volume_new}
    Let $F\colon M^m\times(0,T]\to\R^n$ be a proper mean curvature flow, and suppose that $\mu_{t_j}\rightharpoonup\mu_0$ as Radon measures as $j\to\infty$ for some sequence $t_j\to 0^+$.
    Then
    \begin{align}
        \sup_{0<t\leq T}\sup_{x_0\in \R^n}\mu_t(B_1(x_0))
        \leq C\sup_{x_0\in\R^n}\mu_0(B_1(x_0))
    \end{align}
    for a constant $C=C(m,n,T)\in(0,\infty)$.
\end{lemma}

\begin{proof}
    Arguing as in \eqref{eq:09-04_03} with $\tilde\eta$ replaced by $\zeta\in C_c^\infty(\R^n\times(0,T])$ and with $\eta(p,t)=\zeta(F(p,t),t)$, we obtain the Brakke formulation of mean curvature flow,
    \begin{align}
        \partial_t\int\zeta\d\mu_t
        =\int\big(\partial_t\zeta+\langle\D\zeta,H\rangle-|H|^2\zeta\big)\d\mu_t.
    \end{align}
    Integrating by parts, we find
    \begin{align}\label{eq:27-03_01}
        \partial_t\int\zeta\d\mu_t
        =\int\big(\partial_t\zeta-\operatorname{div}_\top(\D\zeta)-|H|^2\zeta\big)\d\mu_t,
    \end{align}
    where $\operatorname{div}_\top$ denotes the tangential divergence with respect to $\mu_t$.
    This remains valid for compactly supported test functions $\zeta$ with continuous partial derivatives of order one in $t$ and order two in $x$.

    For $0<s<t\leq T$ and $x_0\in\R^n$ we now consider
    \[
        \zeta(x,r)\vcentcolon=\Big(1-(|x-x_0|^2+2m(r-s))\Big)_+^3,
        \qquad s\leq r\leq t.
    \]
    As in \cite[(3.3) and (4.4)]{Ecker}, we have
    \begin{align}\label{eq:27-03_02}
        \partial_r\zeta-\operatorname{div}_\top(\D\zeta)\leq0.
    \end{align}
    If $|x-x_0|^2\leq\frac14$ and $2m(t-s)\leq\frac14$, then $\zeta(x,t)\geq\frac18$, whereas $\zeta(x,s)=0$ for $|x-x_0|>1$. 
    In combination with \eqref{eq:27-03_01} and \eqref{eq:27-03_02}, it follows that 
    \[
        \frac18\mu_t(B_{\frac12}(x_0))\leq\int\zeta(\cdot,t)\d\mu_t
        \leq\int\zeta(\cdot,s)\d\mu_s.
    \]
    For $0<t\leq\frac{1}{8m}$, we take $s=t_j\to0^+$ and use the convergence $\mu_{t_j}\rightharpoonup\mu_0$ to obtain
    \begin{align}\label{eq:01_09_26}
        \frac18\mu_t(B_{\frac12}(x_0))\leq\mu_0(B_1(x_0)).
    \end{align}
    Covering a unit ball by at most $\Gamma(n)$ balls of radius $\frac12$, we obtain
    \[
        \sup_{0<t\leq \min\{\frac{1}{8m},T\}}\sup_{x_0\in \R^n}\mu_t(B_1(x_0))\leq8\Gamma(n)\sup_{x_0\in\R^n}\mu_0(B_1(x_0)).
    \]
    Iterating the same estimate on successive time intervals of length $\frac1{8m}$ yields the assertion.
\end{proof}

This immediately yields the following corollary, which requires crucial assumptions only on the initial trace.

\begin{corollary}\label{cor:main_calibration_detailed}
    Let $F\colon M^m\times(0,T)\to\R^n$ be a proper oriented mean curvature flow.
    Let $U\subset\R^n$ be an open set such that $F_t(M)\subset U$ for all $t\in(0,T)$, and let $\phi$ be a coflat calibration of degree $m$ on $U$.
    Suppose that $F$ has an oriented initial trace $(\mu_0,\T_0)$ with $\mu_0$ concentrated on $U$ satisfying
    \begin{equation}\label{eq:initial_trace_local_volume}
        \sup_{x_0\in\R^n}\mu_0(B_1(x_0)) < \infty.
    \end{equation}
    If $\E_\phi[\mu_0,\T_0]<\infty$, then energy identity \eqref{eq:initial_trace_energy_identity} holds for all $0\leq t_1\leq t_2<T$ with $\E_\phi[F_0]\vcentcolon=\E_\phi[\mu_0,\T_0]$, and also convergence \eqref{eq:initial_trace_energy_convergence} holds.
\end{corollary}

\begin{proof}
    By \Cref{lem:uniform_bounded_local_volume_new}, condition \eqref{eq:initial_trace_local_volume} yields a uniform local-volume bound on every finite time interval. Since $1-\phi(\tau)\leq2$, condition \eqref{eq:local_calibration_energy_bound} follows. We may therefore apply \Cref{thm:main_calibration_detailed}.
\end{proof}

\Cref{thm:main_calibration} is now an immediate consequence.

\begin{proof}[Proof of \Cref{thm:main_calibration}]
    Properness and smoothness up to $t=0$ imply that \Cref{cor:main_calibration_detailed} is applicable with $(\mu_0,\T_0)=(\mu_{F_0},\T_{F_0})$ and $U=\R^n$.
\end{proof}

\begin{remark}\label{rem:energy_decay_phi=0}
    The arguments in this section also remain valid if $\phi$ has comass at most one rather than being a calibration, with the same interpretation in the coflat case. In particular, taking $\phi=0$ recovers the classical volume dissipation identity whenever the initial volume is finite.

    A priori this statement also allows noncompact proper flows, but in fact there do not exist noncompact, finite-volume, proper, smooth mean curvature flows. Suppose on the contrary that such a flow $F\colon M\times[0,T)\to\R^n$ exists. Fix any $t_0\in(0,T)$. Since $M$ is noncompact and $F_{t_0}$ is proper, $F_{t_0}(M)$ is unbounded; hence we may choose $x_j=F(p_j,t_0)$ with $|x_j|\to\infty$.
    For the backward heat kernel $\rho_{x_0,t_0}(x,t)\vcentcolon=(4\pi(t_0-t))^{-\frac m2}\exp(-\frac{|x-x_0|^2}{4(t_0-t)})$,
    Huisken's monotonicity formula \cite{MR1030675} gives $\int_{\R^n}\rho_{x_j,t_0}\d\mu_0\geq\lim_{t\to t_0^-}\int_{\R^n}\rho_{x_j,t_0}\d\mu_t\geq1$. The last inequality holds since $x_j\in F_{t_0}(M)$ using the same argument as in \cite[Remark 4.16]{Ecker} in the embedded codimension-one case. 
    On the other hand, we have $\mu_0(\R^n)<\infty$ by assumption, while $0\leq\rho_{x_j,t_0}(\cdot,0)\leq(4\pi t_0)^{-m/2}$ and $\rho_{x_j,t_0}(x,0)\to0$ as $j\to\infty$ for every fixed $x\in\R^n$. 
    Thus the dominated convergence theorem implies $\int_{\R^n}\rho_{x_j,t_0}\d\mu_0\to0$, a contradiction.
\end{remark}

\Cref{lem:uniform_bounded_local_volume_new} also yields that a mean curvature flow with an oriented initial trace in a sequential way actually attains it in the sense of \Cref{def:initial_trace}. This observation is independent of the calibration energy or the uniform local volume control.

\begin{proposition}\label{prop:seq_initial_trace}
     Let $F\colon M^m\times(0,T)\to\R^n$ be a proper oriented mean curvature flow.
     Suppose that there exist a Radon measure
    $\mu_0$, an $m$-current $\T_0$, and a sequence $t_j\to0^+$ such that, as $j\to\infty$, 
    \begin{align}\label{eq:seq_initial_trace}
        \mu_{t_j}\rightharpoonup\mu_0 \quad \text{as Radon measures},\quad 
        \T_{t_j}\rightharpoonup \T_0 \quad \text{as currents.}
    \end{align}
    Then $F$ has the oriented initial trace $(\mu_0,\T_0)$ in the sense of \Cref{def:initial_trace}.
\end{proposition}

\begin{proof}
    We first note that any compact set $K\subset\R^n$ can be covered by finitely many balls of radius $1/2$ so that, for any $\hat T <\min\{ \frac{1}{8m},T\}$, \eqref{eq:01_09_26} implies
    \begin{align}\label{eq:31_08_26_2}
        \sup_{0<t\leq \hat T}\mu_t(K)<\infty.
    \end{align}
    
    Let $0\leq \zeta \in C_c^\infty(\R^n)$ and $K_0\vcentcolon = \operatorname{supp}\zeta$.
    Fix $\tilde\gamma \in C_c^\infty(\R^n)$ with $\tilde\gamma=1$ on $K_0$, and let $K_1\vcentcolon=\operatorname{supp}\tilde\gamma$.
    Equation \eqref{eq:29-05_01} with $\tilde\eta=\tilde\gamma^2$ implies that for all $0<s\leq t<\hat T$
    \begin{align}\label{eq:31_08_26_3}
        \int\tilde\gamma^2\d\mu_t + \int_s^t \int|H|^2 \tilde\gamma^2 \d\mu_r\d r = \int\tilde\gamma^2\d\mu_s + 2\int_s^t \int \langle \D\tilde\gamma, H\rangle \tilde\gamma \d\mu_r \d r,
    \end{align}
    so that by Young's inequality and sending $s=t_j\to 0^+$, $t\to \hat T^-$, we obtain
    \begin{align}\label{eq:01_09_26_2}
        \frac12\int_0^{\hat T} \int_{K_0}|H|^2\d\mu_r\d r \leq \int \tilde\gamma^2\d\mu_0 + C \Vert \D\tilde\gamma\Vert_\infty^2 \int_0^{\hat T} \mu_r(K_1)\d r < \infty,
    \end{align}
    where finiteness of the last term follows by \eqref{eq:31_08_26_2}. 
    Now, \eqref{eq:29-05_01} with $\tilde \eta =\zeta$ yields
    \begin{align}
        \Big|\int \zeta\d\mu_t - \int \zeta\d\mu_s\Big|&\leq \Vert\zeta\Vert_\infty\int_s^t\int_{K_0}|H|^2\d \mu_r\d r+ \Vert \D \zeta\Vert_\infty \int_s^t \int_{K_0}|H|\d\mu_r dr\\
        &\leq C(1+\Vert\zeta\Vert_\infty) \int_s^t \int_{K_0}|H|^2\d\mu_r\d r + C\Vert \D\zeta\Vert_\infty^2 \int_s^t \mu_r (K_0) \d r.
    \end{align}
    By \eqref{eq:31_08_26_2} and \eqref{eq:01_09_26_2}, it follows that $t\mapsto \int \zeta\d\mu_t$ is Cauchy as $t\to 0^+$, hence $\mu_t \rightharpoonup \mu_0$ as Radon measures. 
    
    For the currents, we take $\omega\in \Omega^m_c(\R^n)$ and compute using \eqref{eq:0523_01} and \eqref{eq:09-04_05}
    \begin{align}
        \partial_t \int_MF_t^*\omega = \int_M i_t^*\left(\iota_{\partial_t}(F^*(\d\omega))\right).
    \end{align}
    Evaluating in positive normal coordinates as in \eqref{eq:10-04_03} (with $\tau_i=\partial_i F(p,t)$) yields
    \begin{align}
        \left(i_t^*\left(\iota_{\partial_t}(F^*(\d\omega))\right)\right)_p(\partial_1, \dots, \partial_m)= \d\omega_{F(p,t)}(H(p,t), \tau_1, \dots, \tau_m).
    \end{align}
     Now, by \eqref{eq:31_08_26_2} and \eqref{eq:01_09_26_2} with $K_0\vcentcolon=\operatorname{supp}\omega$ we find
    \begin{align}
          \big|\T_t(\omega) - \T_s(\omega)\big|&\leq \Vert\mathrm{d}\omega\Vert_\infty \int_s^t \int_{K_0} |H|\d\mu_r \d r \\
          &\leq C \int_s^t\int_{K_0} |H|^2 \d\mu_r\d r + C\Vert \d\omega\Vert_\infty^2 \int_s^t \mu_r(K_0)\d r.
    \end{align}
    Exactly as for the measures, we conclude $\T_t\rightharpoonup \T_0$ as currents as $t\to 0^+$. 
\end{proof}

\subsection{The one-dimensional case revisited}\label{subsec:1D}

In the rest of this section, we gain more insight into special properties in the one-dimensional case.

We first show that, in the whole space $U=\R^n$, finite calibration energy for curves does not enlarge the admissible class beyond the unidirectional case.

\begin{proposition}\label{prop:curve_to_unidirectional}
    Let $\gamma\colon\R\to\R^n$ be an immersed curve parametrized by arclength, and assume that $\E_\phi[\gamma]<\infty$ for some coflat calibration $\phi$ of degree $1$ on $\R^n$.
    Then $\gamma$ is proper, and there exists $E\in\mathbf{Gr}_1^+(\R^n)=\mathbf{S}^{n-1}$ such that the unidirectional calibration $\phi_E=\langle\cdot,E\rangle$ satisfies $\E_{\phi_E}[\gamma]\le \E_\phi[\gamma]<\infty$.
\end{proposition}

\begin{proof}
    Since $\phi$ has bounded coefficients and $\d\phi=0$ distributionally on the simply connected $\R^n$, there exists $u\in W^{1,\infty}_\textrm{loc}(\R^n)$ such that $\d u=\phi$ a.e. Moreover, since $\phi$ has comass $1$, the function $u$ has a $1$-Lipschitz representative which is smooth on $\R^n\setminus S_\phi$ and satisfies $\d u=\phi$ there. Then $u\circ\gamma$ is absolutely continuous and, since $\mathcal{L}^1(\gamma^{-1}(S_\phi))=0$, where $\mathcal{L}^1$ denotes the Lebesgue measure on $\R$, we have $(u\circ\gamma)'=\phi(\tau)$ a.e.

    %Since $\phi$ is closed, there exists $u\colon\R^n\to\R$ such that $\phi=\d u$.
    Writing $f(s)\vcentcolon = s-u(\gamma(s))$, for all $s_2>s_1$ we have
    \begin{align}
        s_2-s_1 - u(\gamma(s_2)) + u(\gamma(s_1)) = f(s_2)-f(s_1) \leq \int_{\R} f'(s)\d s=\E_\phi[\gamma]<\infty.
    \end{align}
    Since $u$ is $1$-Lipschitz, $\gamma$ is proper, and in particular for any $r>0$, we obtain
    \begin{align}
        2r-|\gamma(r)-\gamma(-r)| \leq \E_\phi[\gamma] < \infty.    
    \end{align}
    In addition, for $r$ sufficiently large, $E_r\vcentcolon=\frac{\gamma(r)-\gamma(-r)}{|\gamma(r)-\gamma(-r)|}\in\mathbf{S}^{n-1}$ is well-defined and
    \begin{align}
        \label{eq:09-05-26}\frac12\int_{-r}^r |\tau-E_r|^2\d s =\int_{-r}^r (1-\langle\tau,E_r\rangle)\d s = 2r-\langle \gamma(r)-\gamma(-r),E_r\rangle \le \E_\phi[\gamma].
    \end{align}
    Hence, for $r'\ge r$, we estimate
    \begin{align}
        2r|E_r-E_{r'}|^2&=\int_{-r}^r |E_r-E_{r'}|^2\d s \\
        &\le 2\int_{-r}^r |\tau-E_r|^2\d s + 2\int_{-r}^r |\tau-E_{r'}|^2\d s \le 8\E_\phi[\gamma].
    \end{align}
    Thus $\{E_r\}$ is a Cauchy sequence as $r\to\infty$, hence converging to some $E\in\mathbf{S}^{n-1}$.
    Taking $r\to\infty$ in \eqref{eq:09-05-26} and using Fatou's lemma, we obtain $\E_{\phi_E}[\gamma]\le \E_\phi[\gamma]$.
\end{proof}

In higher dimensions, there do exist nontrivial ends which have finite energy for some variable-coefficient smooth calibration, whereas the energy diverges for all constant-coefficient calibrations.
This occurs even in codimension one.
Here we record the important fact that if the calibration $\phi$ has constant coefficients, then it satisfies the same scaling property as volume:
\begin{align}\label{eq:scaling}
    \E_\phi[\lambda F] = \lambda^m\E_\phi[F] \qquad \text{for }\lambda>0.
\end{align}

\begin{example}
    There exists an entire minimal graph, i.e., $F\colon \R^8\to\R^9$, $F=(\cdot, f)$ for $f\colon \R^8\to\R$, and $H\equiv0$, whose tangent cone at infinity is the nonplanar cone $C_{3,3}\times\R$, where $C_{3,3}$ denotes the Simons cone
    \[
    C_{3,3}\vcentcolon =\{(u,v)\in\R^4\times\R^4 \mid |u|=|v|\}.
    \]
    Such examples were constructed by Bombieri--De Giorgi--Giusti \cite{MR250205}, and their asymptotic behavior was refined in \cite{MR2846486}.
    
    Writing points in $\R^9$ as $(x,y)\in \R^8\times \R$, we define
    \[
    X(x,y)\vcentcolon =\frac{(-\mathrm Df(x),1)}{\sqrt{1+|\mathrm D f(x)|^2}},
    \qquad
    \phi\vcentcolon =\iota_X(\d x_1\wedge\cdots\wedge \d x_8\wedge \d y).
    \]
    Then $\Div X=0$, hence $\d\phi=0$, and $|X|=1$, so $\phi$ is a smooth calibration on $\R^9$.
    Along $F$, the vector field $X$ coincides with the unit normal, and therefore
    \[
    \phi(\tau)=1, \qquad \E_\phi[F]=0.
    \]
    Moreover, in this case, $G(\phi_{(x,y)})=\{\tau_F(x)\}$ for all $(x,y)\in\R^9$, so the calibrated plane is a moving singleton.

    On the other hand, if $\psi$ is a constant-coefficient calibration, then, since we are in codimension one, $G(\psi)$ is a singleton (cf.\ \Cref{subsec:constant_coeff}), and if $\E_\psi[F]<\infty$, by scaling \eqref{eq:scaling}, we conclude that $\E_\psi[\lambda^{-1}F]\to 0$ as $\lambda\to\infty$. In particular, every tangent cone at infinity has to lie in $G(\psi)$,
    hence is given by a (unique) hyperplane. Since the tangent cone is $C_{3,3}\times\R$, this is impossible.
    Thus $\E_\psi[F]=\infty$.
\end{example}

With a similar argument, one can show that $C_{3,3}$ has infinite calibration energy for all constant-coefficient calibrations, whereas it is calibrated by some coflat calibration in $\R^8$ \cite{MR666108} (see also \cite{MR3568077}).

% Although the above translation-based construction of a suitable calibration works only in codimension one, a new construction of variable-coefficient calibrations from higher-codimensional minimal graphs is recently proposed in \cite{tsai2026calibrating}, under suitable comass conditions.

\begin{remark}\label{rem:m=1_volume_bound}
For $m=1$, the assumption that the initial datum in \Cref{thm:main_calibration} has uniformly finite local volume is redundant.
Consider a noncompact connected component and let $\gamma\colon \mathbb R\to\mathbb R^n$ be its arclength parametrization. As in the proof of \Cref{prop:curve_to_unidirectional}, we have $\phi=\d u$ for some $1$-Lipschitz function $u$. For any $s_1,s_2\in \gamma^{-1}(B_1(x_0))$ with $s_1\leq s_2$, we obtain $$s_2-s_1=u(\gamma(s_2))-u(\gamma(s_1))+\int_{s_1}^{s_2} (1-\phi_{\gamma(s)}(\gamma'(s)) ) \d s \leq 2+\E_\phi[\gamma].$$ Hence $\mathcal{L}^1(\gamma^{-1}(B_1(x_0))) \leq \operatorname{diam}(\gamma^{-1}(B_1(x_0))) \leq 2+\E_\phi[\gamma],$ and therefore $$\sup_{x_0\in\mathbb R^n} \mu_\gamma(B_1(x_0)) = \sup_{x_0\in\mathbb R^n} \mathcal{L}^1(\gamma^{-1}(B_1(x_0))) \leq 2+\E_\phi[\gamma] < \infty.$$ 
Thus finite calibration energy implies uniformly finite local volume for every connected component. 
\Cref{thm:main_calibration} may then be applied on each connected component, and the resulting identities can be summed by monotone convergence. 
Accordingly, in dimension $m=1$, the assumption that the initial datum has uniformly finite local volume can be omitted.
\end{remark}

\section{Applications}\label{sec:application}

In this section, we discuss the applications of the energy identity \eqref{eq:energy_identity} that we announced in \Cref{subsec:cone_intro,subsec:intro_application} and compare them with previous works.

\subsection{The case of zero initial energy and cone rigidity}\label{sec:cone}

The zero-energy case yields the following general dynamical rigidity principle for singular calibrated initial traces.

\begin{theorem}\label{thm:calibrated_initial_trace}
    Let $F\colon M^m\times(0,T)\to\R^n$ be a proper oriented mean curvature flow.
    Let $U\subset\R^n$ be an open set such that $F_t(M)\subset U$ for all $t\in(0,T)$, and let $\phi$ be a coflat calibration of degree $m$ on $U$.
    Suppose that $F$ has an oriented initial trace $(\mu_0,\T_0)$ with $\mu_0$ concentrated on $U$ satisfying \eqref{eq:initial_trace_local_volume}.
    % Suppose also that $(\mu_0,\T_0)$ is calibrated by $\phi$ in the sense that
    % \begin{equation}\label{eq:calibrated_initial_trace}
    %     \mu_0=\|\T_0\|,
    %     \qquad
    %     \phi(\vec{T}_0)=1\quad\|\T_0\|\text{-a.e.}
    % \end{equation}
    If $\E_\phi[\mu_0,\T_0]=0$, then the flow $F$ is stationary in the sense that for every $t\in(0,T)$,
    \begin{equation}\label{eq:calibrated_initial_trace_conclusion}
        \phi(\tau)=1\quad\mu_t\text{-a.e.},
        \qquad
        H\equiv0,
        \qquad
        \mu_t=\mu_0,
        \qquad
        \T_t=\T_0.
    \end{equation}
    In particular, the oriented initial trace is necessarily induced by a smooth proper oriented immersion.
\end{theorem}

\begin{proof}
      The energy identity in \Cref{cor:main_calibration_detailed} gives $\E_\phi[F_t]=0$ and $H\equiv0$ for $t\in(0,T)$.
    Thus $F$ is stationary up to tangential motions on $(0,T)$, so the initial-trace convergence identifies the constant pair $(\mu_t,\T_t)$ with $(\mu_0,\T_0)$.
\end{proof}

This implies a strong rigidity property for calibrated cones.
A locally integral cycle $C$ in $\R^n$ is said to be conical if $\lambda C=C$ holds for all $\lambda>0$, where $\lambda C$ denotes the push-forward of $C$ by the dilation $x\mapsto \lambda x$.
Also, $C=\vec{T}_C\|C\|$ is said to be calibrated if $\phi(\vec{T}_C)=1$ holds $\|C\|$-a.e.\ for a coflat calibration $\phi$ on $\R^n$.

\begin{theorem}\label{thm:calibrated_cone_rigidity}
    Let $C$ be a locally integral conical calibrated $m$-cycle in $\R^n$.
    Suppose that a proper oriented mean curvature flow $F\colon M^m\times(0,T)\to\R^n$ has oriented initial trace $(\mu_0,\T_0)=(\|C\|,C)$.
    Then $F$ is stationary. In particular, $C$ is a finite union of planes.
\end{theorem}

\begin{proof}
    Since $C$ is calibrated, it is locally mass-minimizing and hence stationary \cite{MR666108}.
    Therefore, by the monotonicity formula \cite[17.3]{MR756417}, for every $x_0\in\R^n$ and $R\geq1$, we have $\|C\|(B_1(x_0))\leq R^{-m}\|C\|(B_R(x_0))$.
    Taking $R=1+|x_0|$ and using $B_R(x_0)\subset B_{R+|x_0|}(0)$ and the conicality
    $\|C\|(B_r(0))=r^m\|C\|(B_1(0))$, we obtain
    \[
    \|C\|(B_1(x_0))\leq(\tfrac{1+2|x_0|}{1+|x_0|})^m\|C\|(B_1(0))\leq 2^m\|C\|(B_1(0)).
    \]
    Hence \eqref{eq:initial_trace_local_volume} holds, and
    \Cref{thm:calibrated_initial_trace} implies that $F$ is stationary and thus $C$ can be parameterized by a smooth proper immersion.
    Each connected component of such an entirely smooth complete cone passes the origin, thus being a plane.
    Local integrability at the origin yields finiteness.
\end{proof}

Hence, no singular calibrated cone can be resolved by a smooth proper oriented mean curvature flow.

As a particular corollary, we establish a general rigidity theorem for self-expanders that are asymptotic to calibrated cones, or more generally, merely admit a calibrated cone as a varifold-current blow-down.
An immersion $F\colon M^m\to\R^n$ is called a \emph{self-expander} if $F_t\vcentcolon =\sqrt{2t+1}\,F$ defines a mean curvature flow for $t\ge 0$; equivalently,
\begin{equation}\label{eq:self-expander}
    F^\perp = H.
\end{equation}

%With similar arguments as in the proof of \Cref{thm:self_expander}, one obtains a non-energetic rigidity result for self-expanders under weak (measure-current) asymptotic conicality, analogous to \Cref{thm:calibrated_cone_rigidity}.

\begin{corollary}\label{cor:calibrated_cone_rigidity_asymptotic_expander}
    Let $C$ be a locally integral conical calibrated $m$-cycle in $\R^n$.
    Let $F:M^m\to\R^n$ be a proper oriented self-expander and suppose that there exist $\lambda_j\to 0^+$ such that, as $j\to\infty$, $\mu_{\lambda_j F} \rightharpoonup \|C\|$ as Radon measures and $\T_{\lambda_j F} \rightharpoonup C$ as currents.
    Then $F$ must be a finite union of planes.
\end{corollary}

\begin{proof}
    Consider the self-expanding proper oriented flow $F_t=\sqrt{2t}\,F$, which also solves \eqref{eq:MCF} and set $t_j\vcentcolon=\frac12 \lambda_j^2$. Then, by assumption, $\mu_{F_{t_j}}\rightharpoonup \| C\|$ as Radon measures and $\T_{F_{t_j}}\rightharpoonup C$ as currents. \Cref{prop:seq_initial_trace} implies that $F_t$ has the oriented initial trace $(\|C\|,C)$ as $t\to0^+$.
    Hence \Cref{thm:calibrated_cone_rigidity} applies.
\end{proof}

Self-expanders are known to model the asymptotic behavior of the mean curvature flow \cite{MR1025164}, as well as the flow coming out of conical singularities \cite{MR1361726} as discussed in the introduction.
The problem of understanding and classifying self-expanders, first posed by Ecker--Huisken \cite[p.~466]{MR1025164}, remains an important topic.
It is also worth recalling that self-expanders, as well as other typical self-similar solutions, can be regarded as weighted minimal submanifolds \cite{ilmanen2026lectures}.

The theory of asymptotically conical self-expanders, as well as self-shrinkers, is well developed in the literature particularly in codimension one.
In addition to the results discussed in \Cref{subsec:cone_intro}, we also mention the uniqueness results for asymptotically conical self-shrinkers by Wang \cite[Theorem 1.3]{MR3194490} and for self-expanders by Bernstein \cite[Theorem 1.2]{MR4176548}, as well as Khan's extensions to all codimensions \cite[Theorems 1.1 and 1.4]{MR4669035}.
In the self-expander case, one needs to impose a rather strong decay at infinity, and without sufficient decay, uniqueness fails in general; see, e.g., \cite[Theorem 1.1]{MR2982716}, \cite[Corollary 1.4]{MR4492654}, and \cite[Corollary 1.11]{shao2025self}. 
We also refer to further conditional uniqueness results \cite{MR4055164,MR4126890,MR4399282,MR4524828} in codimension one.

% MEMO:
% \begin{itemize}
%     \item Cf.\ Park, Ancient mean curvature flow asymptotic to Simons cone.
%     %\item Comparison with dynamical instability contexts? Stuvard--Tonegawa, Motegi.
% \end{itemize}

\subsection{Uniqueness of self-expanders}\label{subsec:self-expander}

% An immersion $F\colon M^m\to\R^n$ is called a \emph{self-expander} if $F_t\vcentcolon =\sqrt{2t+1}\,F$ defines a mean curvature flow for $t\ge 0$; equivalently,
% \begin{equation}\label{eq:self-expander}
%     F^\perp = H.
% \end{equation}
% Self-expanders are known to model the asymptotic behavior of the mean curvature flow \cite{MR1025164}, as well as the flow coming out of conical singularities \cite{MR1361726}.
% The problem of understanding and classifying self-expanders, first posed by Ecker--Huisken \cite[p.~466]{MR1025164}, remains an important topic.
% It is also worth recalling that self-expanders, as well as other typical self-similar solutions, can be regarded as weighted minimal submanifolds \cite{ilmanen2026lectures}.

An important consequence of the energy identity \eqref{eq:energy_identity} is a uniqueness theorem for self-expanders (cf.\ \eqref{eq:self-expander}) in the regime of finite calibration energy, which is of a different nature from \Cref{cor:calibrated_cone_rigidity_asymptotic_expander}.
% If the calibration $\phi$ has constant coefficients, then it satisfies the same scaling property as volume:
% \begin{align}\label{eq:scaling}
%     \E_\phi[\lambda F] = \lambda^m\E_\phi[F] \qquad \text{for }\lambda>0.
% \end{align}
Recall that any constant-coefficient calibration $\phi$ satisfies the volume-like scaling property \eqref{eq:scaling}.
In fact, this remains true even if merely the calibration $\phi$ is coflat on $\R^n$ and \emph{dilation-invariant} in the sense that 
\begin{equation}
    \text{for $\mathcal H^m$-a.e.\ $x\in\R^n$, we have $\phi_{\lambda x}=\phi_x$ for all $\lambda>0$}.
\end{equation} 

\begin{theorem}\label{thm:self_expander}
    Let $F\colon M^m\to\R^n$ be a connected, proper oriented self-expander such that $\E_\phi[F]<\infty$ for a dilation-invariant coflat calibration $\phi$ of degree $m$ on $\R^n$.
    Then $F$ must be a plane, calibrated by $\phi$ and passing through the origin.
    
    In particular, if $\phi$ admits no calibrated plane, then there is no smooth proper oriented self-expander with finite $\phi$-calibration energy.
\end{theorem}

\begin{proof}
    We verify the assumptions of \Cref{thm:main_calibration_detailed}.
    Since $U=\R^n$, the validity of \eqref{eq:inclusion} is trivial.
    By the scaling property \eqref{eq:scaling} the self-expander $F_t=\sqrt{2t+1}\,F$ satisfies $\E_\phi[F_t]=(2t+1)^{\frac{m}{2}}\E_\phi[F]<\infty$, which in particular ensures \eqref{eq:local_calibration_energy_bound}.
    Then \Cref{thm:main_calibration_detailed} implies that for $t>0$ 
    \[
    (2t+1)^{\frac{m}{2}}\E_\phi[F]=\E_\phi[F_t] \leq \E_\phi[F_0]=\E_\phi[F].
    \]
    Hence $\E_\phi[F]=0$, so $\E_\phi[F_t]=0$ and, by \eqref{eq:initial_trace_energy_identity}, 
    %so the associated current $\T_F$ is calibrated by a coflat $\phi$, thus being locally mass-minimizing.
    %Since $F$ is smooth, this implies 
    we conclude $H_F\equiv0$.
    Then \eqref{eq:self-expander} yields $F^\perp\equiv0$, that is, $F$ is a connected complete smooth cone, thus being a plane through the origin.
    Since $\E_\phi[F]=0$, the plane is calibrated by $\phi$.
\end{proof}

We stress that the assumption of finite calibration energy in \Cref{thm:self_expander} is formulated purely in terms of the behavior of the ends, without any restriction inside, in contrast to standard Bernstein-type uniqueness theorems.
Unlike \Cref{cor:calibrated_cone_rigidity_asymptotic_expander}, the ends are not assumed to converge to any prescribed cone, nor even to admit an asymptotic cone; instead, only finite energy is imposed.
Here we need not assume the local-volume bound thanks to self-similarity.

Our admissible class of $\phi$ includes arbitrary constant-coefficient calibrations.
In this case, finite $\E_\phi$ does not, in general, prescribe a single asymptotic plane, but only requires the tangent planes to approach (in an integral sense) the set $G(\phi)$, which may be a large family of planes: classical examples include the set of complex planes, special Lagrangian planes, (co)associative planes, or Cayley planes \cite{MR666108}.
In fact, since $\phi$ can be singular, the admissible ends are a priori wider even in the codimension-one case.
In particular, Zhang's calibrations \cite{MR3568077,MR5057730} are admissible as they are dilation-invariant and their singularity is isolated at the origin.

In addition, even in the unidirectional case $\phi_E=\langle\cdot,E\rangle$, our result is substantially new.
It may be viewed as a uniqueness theorem with a Neumann-type boundary condition at infinity: one prescribes the asymptotic tangent direction rather than the asymptotic position. 
The Dirichlet counterpart is much easier: indeed, if a self-expander $F$ is asymptotic to a fixed plane $E$ at infinity, then for each $e\in\S^{n-1}$ orthogonal to $E$, the height function $h=\langle F,e\rangle$ satisfies $\Delta h+\frac{1}{2}g(\nabla|F|^2,\nabla h)-h=0$ and $h\to0$ at infinity, so the maximum principle yields $h\equiv0$.

The Neumann problem is much more subtle: the ends may converge to different planes, or even diverge; see \Cref{ex:topology,ex:single_end_multiplicity,ex:graphical} for more details on our admissible class.
A known result in this direction is Ding's uniqueness theorem in codimension one \cite[Theorem 4.2]{MR4055164}: a self-expander must be a hyperplane whenever the tangent cone at infinity is a hyperplane with an integer multiplicity.
Ding's proof is based on barriers and does not directly extend to higher codimension.
To the authors' knowledge, even in the unidirectional case, our result is the first higher-codimension uniqueness theorem for self-expanders formulated in terms of asymptotic tangent direction.

% We now turn to the proof of \Cref{thm:self_expander}, which is an almost direct consequence of \Cref{thm:main_calibration_detailed} and the scaling property.

The above argument for \Cref{thm:self_expander} also restricts possible self-expanding limits in rescaled graphical convergence problems.
We say that a mean curvature flow $F\colon \R^m\times[0,T)\to\R^n$ is \emph{graphical} if it is of the form
\[
F_t=(\cdot,f(\cdot,t))
\]
for some family of functions $f\colon\R^m\times[0,T)\to\R^{n-m}$.
In codimension one, Ecker--Huisken \cite{MR1025164} proved that if $F_0$ is \emph{asymptotically conical} and $f(\cdot,0)$ has bounded gradient, then the flow is graphical and immortal, and its rescalings
\begin{align}\label{eq:Ecker_Huisken_rescaling}
    \tilde{F}_t\vcentcolon =\frac{1}{\sqrt{2t+1}}F_t
\end{align}
converge locally smoothly to graphical self-expanders.
Recently, this was extended by Savas-Halilaj--Smoczyk \cite{MR4810572} to codimension two.

\begin{theorem}\label{thm:rescaled_limit}
    Let $\phi$ be a dilation-invariant coflat calibration of degree $m$ on $\R^n$.
    Let $F\colon \R^m\times[0,\infty)\to\R^n$ be a graphical mean curvature flow such that the initial datum $f(\cdot,0)$ has bounded gradient, and suppose that its rescaling $\tilde{F}_t$ locally smoothly converges to a self-expander $\tilde{F}_\infty$ along a sequence $t_j\to\infty$.
    If $\E_\phi[F_0]<\infty$, then $\lim_{j\to\infty}\E_\phi[\tilde{F}_{t_j}]=0$ and $\tilde{F}_\infty$ is a plane calibrated by $\phi$.
\end{theorem}

\begin{proof}
    Since the bounded gradient assumption for $f(\cdot,0)$ implies that $F_0$ has uniformly finite local volume, and since any graphical flow is proper, \Cref{cor:main_calibration_detailed} applies to $F$.
    Using first the scaling property \eqref{eq:scaling} and then the energy identity \eqref{eq:energy_identity}, we obtain
    \[
        \E_\phi[\tilde{F}_{t_j}]
        =(2t_j+1)^{-\frac{m}{2}}\E_\phi[F_{t_j}]
        \le (2t_j+1)^{-\frac{m}{2}}\E_\phi[F_0]\to 0
        \qquad (t_j\to\infty).
    \]
    Hence Fatou's lemma yields $\E_\phi[\tilde{F}_\infty]
        \le \lim_{j\to\infty}\E_\phi[\tilde{F}_{t_j}]
        =0$.
   Since $\tilde{F}_\infty$ is graphical, in particular proper, and a self-expander, we may argue as in the proof of \Cref{thm:self_expander} to conclude that $\tilde F_\infty$ is a calibrated plane. 
\end{proof}

In particular, this theorem directly applies to the graphical initial data of Ecker--Huisken \cite[Theorem 5.1]{MR1025164} for $n=m+1$ (see also Stavrou's extension \cite{MR1631112}), of Savas-Halilaj--Smoczyk \cite[Theorem 6.4]{MR4810572} for $n=m+2$ under the area-decreasing property, and of Tsai--Tsui--Wang \cite[Theorem 1.2]{MR4927791} for $n=2m$ under the Lagrangian property with two-convex potential.

In fact, our argument is also compatible with the rescaled convergence theory of Cheng--Sesum \cite{MR3900480} for nongraphical codimension-one solutions (called type III), assuming suitable local-volume bounds.

In all these cases, assuming in addition that the initial datum has finite calibration energy for some dilation-invariant coflat calibration $\phi$, \Cref{thm:rescaled_limit} implies that the resulting self-expanders are $\phi$-calibrated planes.
%On the other hand, not all graphical immortal solutions converge to self-expanders; some solutions (called type IIb) instead converge, after rescaling, to translators \cite{MR4113204,MR4442627}.

\subsection{Rigidity of translators}

Let $v \in \mathbf{S}^{n-1} \subset \R^n$.
An immersion $F\colon M^m\to\R^n$ is called a \emph{translator (in the direction $v$)} if $F_t\vcentcolon =F+tv$ defines a mean curvature flow for all $t\in\R$; equivalently,
\begin{equation}\label{eq:translator}
    v^\perp = H.
\end{equation}
Translators arise naturally as models of type II singularities and as canonical eternal solutions: their abundance and rigidity are also extensively investigated \cite{MR4281668}.

A simple way to construct a class of (stationary) translators is using cylinders. 
Let $\tilde F\colon \tilde M^{m-1}\to\R^n$ be an immersion and let $v\in \S^{n-1}$ be normal along $\tilde F$. We define \emph{the cylinder over $\tilde F$ (in the direction $v$)} by the immersion
\begin{align}
    F\colon \tilde M \times \R\to\R^n,\quad  F(\tilde p,s)\vcentcolon = \tilde F(\tilde p)+ s v.
\end{align}
In addition, we call an immersion $F\colon M^m\to\R^n$ \emph{a cylinder (in the direction $v$)} if there exists $\tilde F \colon \tilde  M^{m-1}\to\R^n$, and an orientation-preserving diffeomorphism $\Phi\colon \tilde M\times \R\to M$ such that $F\circ \Phi$ is the cylinder over $\tilde F$ (in the direction $v$).

\begin{remark}\label{remark:cylinder}
    If $F\colon M^m\to\R^n$ is the cylinder over some $\tilde F\colon \tilde M^{m-1}\to\R^n$ with $M$ connected/oriented/without boundary, then also $\tilde M$ is connected/oriented/without boundary. If $F$ is proper, then $\tilde F$ is proper.
    In addition, the tangential Gauss maps of $F$ and $\tilde F$ are related by
    \begin{align}
    \tau(\tilde p,s) = \tilde\tau(\tilde p)\wedge v\quad \text{for all }(\tilde p,s)\in \tilde M\times \R.
    \end{align}
    Hence, if $m\geq2$, $\phi$ has constant coefficients, and $F$ is calibrated by $\phi$, then $\tilde F$ is calibrated by the constant-coefficient calibration $\tilde\phi\vcentcolon=(-1)^{m-1}\iota_v\phi$.
\end{remark}

Since calibrated immersions are minimal, it is easy to construct translators using cylinders. Indeed, if $\tilde F\colon \tilde M^{m-1}\to\R^{n-1}$ is $\tilde\phi$-calibrated for some calibration $\tilde\phi\in \Omega^{m-1}(\R^{n-1})$ and we view $\R^{n-1}\cong\R^{n-1}\times\{0\}\subset\R^n$, the cylinder over $\tilde F$ (in the direction $e_n$) defines a translator (in the direction $e_n$) which is calibrated by $\phi\vcentcolon = \tilde \phi \wedge \d x_n$.

Our energy identity implies the strong rigidity that all translators with a suitable finite calibration energy are $\phi$-calibrated cylinders. 
We say that a coflat calibration $\phi$ is \emph{translation-invariant in the direction $v$} if
\begin{align}\label{eq:coflat_trans_inv}
    \text{for $\mathcal H^m$-a.e.\ $x\in\R^n$, we have $\phi_{x+sv} = \phi_x$ for all $s\in\R$}.
\end{align}
In the following we adopt the convention that a connected $0$-dimensional manifold is a single point; thus, a connected $1$-dimensional cylinder is a straight line. 

\begin{theorem}\label{thm:translator_constant_coeff}
    Let $F\colon M^m\to\R^n$ be a connected, proper oriented translator in the direction $v\in \S^{n-1}$ such that $\E_\phi[F]<\infty$ for a coflat calibration $\phi$ of degree $m$ on $\R^n$ that is translation-invariant in the direction $v$. Then $F$ is a $\phi$-calibrated cylinder in the direction $v$. 
\end{theorem}

\begin{proof}
    Let $F_t\vcentcolon =F+tv$.
    As in the self-expander case, we can apply \Cref{thm:main_calibration_detailed} to deduce \eqref{eq:energy_identity}.
    Using \eqref{eq:coflat_trans_inv}, we find that $\E_{\phi}[F_t]=\E_{\phi}[F]$ and hence $H\equiv0$ by \eqref{eq:energy_identity}.

    % If $m=1$, then $F$ is a straight line parallel to $v$.
    % By translation-invariance, the nonnegative integrand $1-\phi(\tau)$ is a.e.\ constant along the line. Thus, we either have $\E_\phi[F]=0$ or $\E_\phi[F]=\infty$, so necessarily $\E_\phi[F]=0$ and this completes the case $m=1$.
    
    % Assume from now on that $m\ge2$.
    By \eqref{eq:translator} with $H\equiv0$, we obtain $v^\perp\equiv0$. Consequently, there exists a nowhere vanishing vector field $X$ on $M$ such that $F_* X = v$. After translating $F$, we may assume $\tilde M \vcentcolon= \{ p\in M\mid \langle F(p), v\rangle = 0 \}\neq\emptyset$ and, by the submersion theorem, it is an $(m-1)$-dimensional submanifold (a point if $m=1$). It can be shown that the vector field $X$ generates a global flow and yields an orientation-preserving diffeomorphism $\Phi\colon \tilde M\times \R\to M$. Defining $\tilde F=F\vert_{\tilde M}$, by construction, we deduce that $F\circ \Phi$ is the cylinder over $\tilde F$ in the direction $v$.

    Recall from \Cref{remark:cylinder} that the tangential Gauss maps satisfy $\tau\circ\Phi=\tilde\tau\wedge v$ (where $\tilde\tau$ maps to $\Lambda^0\R^n = \{1\}$ when $m=1$). Hence, 
     \[
    \E_\phi[F]
        = \int_{\R}\int_{\tilde M}
        \bigl(1-\phi_{\tilde F+s v}(\tilde\tau\wedge v)\bigr)\d\mathrm{vol}_{\tilde F}\d s
        <\infty.
    \]
    By \eqref{eq:coflat_trans_inv}, the integrand is independent of $s$ and hence we must have 
    \begin{align}
        \phi_{F}(\tau)\circ\Phi=\phi_{\tilde F+\cdot v}(\tilde\tau\wedge v)=\phi_{\tilde F}(\tilde\tau\wedge v)\equiv1 \qquad \text{a.e. in $\tilde M\times \R$},
    \end{align}
    where we again used \eqref{eq:coflat_trans_inv} in the penultimate step.
\end{proof}

In \Cref{thm:translator_constant_coeff}, if in addition $\phi$ has constant coefficients and
\[
    G_v(\phi)\vcentcolon =\{\xi\in G(\phi)\mid \xi\wedge v=0 \}
\]
is finite, then $F$ is a plane.
Indeed, in this case $\tau(M)\subset G_v(\phi)$ is connected, and thus consists of a single point.
In particular, we obtain the following ($v$-independent) uniqueness theorem.

\begin{corollary}\label{cor:translator_plane}
    Let $F\colon M^m\to\R^n$ be a connected, proper oriented translator with $\E_\phi[F]<\infty$ for a constant-coefficient calibration $\phi\in\Omega^m(\R^n)$.
    \begin{enumerate}
        \item If $G(\phi)$ is finite, then $F$ is a plane.
        \item If $m\in\{1,2,n-1\}$, then $F$ is a plane.
    \end{enumerate}
\end{corollary}

\begin{proof}
    The first assertion follows from the preceding paragraph.
    The cases $m\in\{1,n-1\}$ are also immediate consequences of \Cref{thm:translator_constant_coeff}, since $\phi$ is necessarily unidirectional (cf.\ the beginning of \Cref{subsec:constant_coeff}), where $G(\phi)$ is a singleton.

    Suppose $m=2$.
    Then \Cref{thm:translator_constant_coeff} yields that (after applying a diffeomorphism)
    $M^2 = \tilde M^1\times\R$ with $F(\tilde p,s)=\tilde F(\tilde p)+sv$.
    Since the constant-coefficient calibration $\tilde \phi = - \iota_v\phi\in\Omega^1(\R^n)$ calibrating $\tilde F$ must be unidirectional, the immersion $\tilde F$ of $\tilde M$ must also be a straight line.
    Hence $F$ is a plane.
\end{proof}

\begin{remark}
    The splitting conclusion in \Cref{thm:translator_constant_coeff} is optimal in general: there exist nonplanar translators with zero calibration energy.
    Let $m=3$ and $n=5$, and identify $\R^5=\C^2\times\R$ with $v\vcentcolon =e_5$.
    Let $\omega$ be the K\"ahler form on $\C^2$ (cf.\ \Cref{subsec:kaehler}), and define
    \[
    \phi\vcentcolon =\d x_5\wedge\omega\in\Omega^3(\R^5).
    \]
    One readily checks that $\phi$ is a constant-coefficient calibration.
    % Indeed, for any $\xi\in\mathbf{Gr}_3^+(\R^5)$, we may
    % represent $\xi=v_1\wedge v_2\wedge v_3$ with orthonormal vectors $v_1,v_2,v_3$ such that $v_2,v_3\in \C^2$.
    % Since $\omega$ has comass $1$ on $\C^2$, we obtain $\phi(\xi)=\d x_5(v_1)\omega(v_2,v_3)= \langle e_5,v_1\rangle\omega(v_2,v_3) \le 1$.
    % Equality is attained whenever $v_1=e_5$ and $v_2\wedge v_3\in G(\omega)$, so $\phi$ is a
    % calibration.
    Moreover,
    \[
        G_v(\phi)=\{e_5\wedge\eta\mid \eta\in G(\omega)\} \cong \{\text{complex lines in $\C^2$}\}.
    \]
    In particular, if we consider $F\colon \C\times\R\to\C^2\times\R\cong\R^5$ of the form
    \[
    F(z,s)\vcentcolon =(\tilde{F}(z),s), \qquad \tilde{F}(z)\vcentcolon =(z,f(z)), \qquad f\colon\C\to\C\ \text{holomorphic},
    \]
    then, choosing the appropriate product orientation, we have $\tau=e_5\wedge\tilde\tau$ and $\phi(\tau)=\omega(\tilde\tau)=1$.
    Thus $F$ is calibrated by $\phi$, 
    and thus $H\equiv0$.
    Since $e_5$ is tangent everywhere, we also have $e_5^\perp\equiv0$, and hence
    $F$ is a translator in the direction $e_5$.
\end{remark}

\begin{remark}\label{rem:block_example}
    There is a nontrivial family of $\phi$ for which $G(\phi)$ is finite.
    Trivial examples include the unidirectional case in general (co)dimensions, where $G(\phi)$ is a singleton.
    In addition, for $m\geq3$ and $N\geq2$, if we consider
    \[
    \R^{mN}=P_1\oplus\dots\oplus P_N
    \]
    with mutually orthogonal $m$-planes $P_1,\dots,P_N$ in $\R^{mN}$, and take the volume form $\d\mathrm{vol}_j$ on $P_j$, then the $m$-form
    \[
    \phi_N\vcentcolon =\d\mathrm{vol}_1+\dots+\d\mathrm{vol}_N
    \]
    defines a calibration such that $G(\phi_N)$ is a discrete set with exactly $N$ elements:
    \[
    G(\phi_N)=\{P_1,\dots,P_N\}.
    \]
    Indeed, for any oriented orthonormal $m$-frame $v_1,\dots,v_m\in\R^{mN}$, if $\ell_{ij}\in[0,1]$ denotes the length of $v_i$ projected to $P_j$, then Hadamard's inequality yields the estimate $|\!\d\mathrm{vol}_j(v_1,\dots,v_m)|\leq\prod_{i=1}^m\ell_{ij}$
    and hence, using H\"older's inequality and an elementary norm comparison,
    \[
    |\phi_N(v_1,\dots,v_m)|\leq \sum_{j=1}^N\prod_{i=1}^m\ell_{ij} \leq \prod_{i=1}^m \Big(\sum_{j=1}^N\ell_{ij}^m\Big)^{1/m}\leq \prod_{i=1}^m \Big(\sum_{j=1}^N\ell_{ij}^2\Big)^{1/2} = 1,
    \]
    where equality holds if and only if all $v_1,\dots,v_m$ are contained in a single plane $P_j$.
\end{remark}

\begin{remark}
    In the proof of \Cref{cor:translator_plane} (i), we only used the fact that $G(\phi)$ is totally disconnected. However, in the present constant-coefficient setting, the assumption that $G(\phi)$ is finite yields no loss of generality: under the Pl\"ucker embedding, $G(\phi)$ is a real algebraic set, and therefore has finitely many connected components, cf.\ \cite[Theorem 2.4.5]{MR1659509}. In particular, total disconnectedness implies finiteness.
\end{remark}

\Cref{thm:translator_constant_coeff} and \Cref{cor:translator_plane} should be compared to Lynch--Tinaglia's recent uniqueness result for translator surfaces in $\R^3$ under asymptotic planarity \cite{MR4996046}, motivated by Khan's structure theorem \cite{MR4553955}, and to Neves--Tian's rigidity result for Lagrangian translators in $\C^2$ \cite{MR3091260}.
See also \cite{MR4145202,MR3412395,MR3895630,MR4190399,MR4052220} for various rigidity results for codimension-one translators, as well as \cite{MR4947201,MR4784721} and references therein on the abundance side.

In higher codimensions, there are several Bernstein-type theorems for translators \cite{MR3396413,MR3396441,MR4496501,MR4701858,MR4632930,MR3842856}, which typically require global geometric assumptions.
To the authors' knowledge, our result is the first rigidity theorem in arbitrary codimension whose assumptions are formulated purely in terms of the behavior of the ends.
In particular, our assumption does not impose any topological restrictions.

\subsection{No breather theorems in the nonshrinking case}

A breather is a solution which is self-similar at two distinct times: a mean curvature flow $F\colon M^m\times[t_1,t_2]\to\R^n$, where $t_1<t_2$, is a \emph{breather} if there are a constant $\lambda>0$, an isometry $\Psi\colon \R^n\to\R^n$, and a diffeomorphism $\Xi\colon M\to M$ such that
\begin{equation}\label{eq:breather_def}
    F_{t_2} = \lambda \Psi \circ F_{t_1} \circ \Xi.
\end{equation}
If $\lambda=1$, $\lambda>1$, or $\lambda<1$, then the breather is said to be \emph{steady}, \emph{expanding}, or \emph{shrinking}, respectively.

No breather theorems are rigidity results asserting that, under suitable conditions, all breathers must be self-similar solutions (or even more rigid).
Since Perelman's work for closed Ricci flow \cite{perelman2002entropy}, no breather theorems have been established in various contexts.
As for mean curvature flow, in the shrinking case we have fairly general no breather theorems thanks to Huisken's monotonicity formula, both in the compact \cite{MR2528703} and noncompact setting \cite{cheng2021no}.
On the other hand, in the nonshrinking (i.e., steady or expanding) case, a general no breather theorem is still available in the compact setting thanks to the volume-decreasing property \cite{MR2528703}. The noncompact situation is much more subtle; in fact, Topping \cite[Remark 2]{MR4568346} recently discovered a nontrivial expanding breather.
So far, the only known no breather theorem in the nonshrinking noncompact case seems to be Cheng--Zhang's result under convexity \cite{cheng2021no}.

Our energy identity yields new types of no breather theorems for noncompact mean curvature flow under finite calibration energy.
The first result concerns $\phi$-preserving breathers: for a coflat calibration $\phi$ of degree $m$ on $\R^n$, we say that a connected and oriented breather is \emph{$\phi$-preserving} if the maps $\Psi$ and $\Xi$ in \eqref{eq:breather_def} satisfy
\[
\sign(\Xi) (\lambda\Psi)^*\phi=\lambda^m\phi \qquad \text{$\mathcal{H}^m$-a.e.\ in $\R^n$},
\]
where $\sign(\Xi)=+1$ (resp.\ $-1$) if $\Xi$ is orientation-preserving (resp.\ reversing).
In particular, if $\phi$ has constant coefficients, then thanks to the scaling property, the $\phi$-preserving property is equivalent to satisfying $\sign(\Xi)\Psi^*\phi=\phi$.

\begin{theorem}\label{thm:no_breather_1}
    Let $\phi$ be a coflat calibration of degree $m$ on $\R^n$.
    Let $F\colon M^m\times[t_1,t_2]\to\R^n$ be a connected, proper, oriented, nonshrinking, $\phi$-preserving breather.
    Suppose that $F_{t_1}$ has uniformly finite local volume and finite $\phi$-calibration energy.
    Then $F_{t_1}$ is minimal (and hence $F$ is stationary up to tangential motions).
\end{theorem}

\begin{proof}
    Since the breather is $\phi$-preserving and nonshrinking, we have
    \[
    \E_\phi[F_{t_2}]= \E_\phi[\lambda \Psi \circ F_{t_1} \circ \Xi] = \lambda^m\E_{\lambda^{-m}\sign(\Xi)(\lambda\Psi)^*\phi}[F_{t_1}]= \lambda^m\E_{\phi}[F_{t_1}] \geq \E_{\phi}[F_{t_1}].
    \]
    On the other hand, \Cref{cor:main_calibration_detailed} implies that
    \begin{equation}
        \E_\phi[F_{t_2}]\leq \E_\phi[F_{t_2}] +\int_{t_1}^{t_2} \int_M |H|^2 \d\mathrm{vol} \d t = \E_\phi[F_{t_1}].
    \end{equation}
    Hence $|H|=0$ must hold everywhere.
\end{proof}

% \begin{remark}
%     If we replace the definition of $\phi$-preservation by $\sign(\Xi) (\lambda\Psi)^*\phi=\lambda^m\phi$, then the same argument
%     works even for the variable-coefficient case.
% \end{remark}

Note that Topping's example \cite[Remark 2]{MR4568346} is $\phi$-preserving for any dilation-invariant $\phi$ since $\sign(\Xi)=+1$ and $\Psi=\mathrm{Id}$.
Hence such an example is always ruled out by assuming finite calibration energy.

The next result does not assume the $\phi$-preserving property, but instead restricts the class of calibrations.
We say that a subset $S\subset\mathbf{Gr}_m^+(\R^n)$ is \emph{exposable} if there is a constant-coefficient calibration $\psi_S\in\Omega^m(\R^n)$ such that $G(\psi_S)=S$.
In addition, we say that $S\subset\mathbf{Gr}_m^+(\R^n)$ is \emph{strongly exposable} if every nonempty subset $S'\subset S$ is exposable.
Finally, we say that $\phi$ is \emph{nondegenerate} if the following quadratic lower bound holds for some $c_\phi>0$:
\begin{equation}\label{eq:non_degenerate_quadratic}
    1-\phi(\xi)\ge c_\phi\dist(\xi,G(\phi))^2 \qquad \text{for all }\xi \in \mathbf{Gr}_m^+(\R^n),
\end{equation}
where $\dist$ denotes the extrinsic distance in $\Lambda^m\R^n \cong \R^{\binom nm}$. Such a lower bound is not automatic, see \Cref{rem:nonquadratic}; however, constant-coefficient calibrations always satisfy the converse inequality, cf.\ \Cref{prop:constant_coeff_upper_bound}.

\begin{theorem}\label{thm:no_breather_2}
    Let $F\colon M^m\times[t_1,t_2]\to\R^n$ be a connected, proper, oriented, nonshrinking breather.
    Let $\phi\in\Omega^m(\R^n)$ be a nondegenerate constant-coefficient calibration such that $G(\phi)$ is finite and strongly exposable.
    Suppose that $F_{t_1}$ has uniformly finite local volume and finite $\phi$-calibration energy.
    Then $F_{t_1}$ is minimal.  
\end{theorem}

\begin{proof}
    If $\mathrm{vol}(M)\vert_{t_1}<\infty$, then \Cref{rem:energy_decay_phi=0} implies that $M$ is compact, so that the desired rigidity of nonshrinking breathers immediately follows from the classical decay of the volume.
    Hence we may assume that $\mathrm{vol}(M)\vert_{t_1}=\infty$.

    By the nondegeneracy in \eqref{eq:non_degenerate_quadratic} and by
    $\E_\phi[F_{t_1}]<\infty$, we have $\dist(\tau,G(\phi)) \in L^2(M)$ at time $t_1$.
    In particular, for every $\rho>0$,
    \begin{align}\label{eq:26-05_3}
        \mathrm{vol}\bigl(\{p\in M\mid
        \dist(\tau(p,t_1),G(\phi))\ge \rho\}\bigr)<\infty .
    \end{align}

    We now define the set of the elements of $G(\phi)$ which are actually seen with
    infinite volume, that is,
    \[
        S\vcentcolon=
        \bigl\{
        \xi\in G(\phi) \bigm|
        \lim_{\rho\to0^+}
        \mathrm{vol}\bigl(\{p\in M\mid
        \dist(\tau(p,t_1),\xi)<\rho\}\bigr)=\infty
        \bigr\}.
    \]
    Since $G(\phi)$ is finite and the volume at time $t_1$ is infinite, the set
    $S$ is nonempty. By \eqref{eq:26-05_3}, the definition of $S$, and the finiteness of $G(\phi)$, for $\rho>0$ sufficiently small, we have
    \begin{equation}\label{eq:finite_complement_of_S}
        \mathrm{vol}\bigl(\{p\in M\mid
        \dist(\tau(p,t_1),S)\ge\rho\}\bigr)<\infty .
    \end{equation}

    Since $G(\phi)$ is strongly exposable, there exists a constant-coefficient
    calibration $\psi\in\Omega^m(\R^n)$ such that
    \[
        G(\psi)=S.
    \]
    We average $\psi$ over the signed stabilizer of $S$. To this end, let
    \[
        \Gamma_S\vcentcolon=
        \bigl\{
        (\sigma,R)\in\{\pm1\}\times O(n)
        \bigm| \sigma R_\# S=S
        \bigr\},
    \]
    where $R_\#$ denotes the induced action
    $v_1\wedge\dots\wedge v_m\mapsto Rv_1\wedge\dots\wedge Rv_m$.
    The group $\Gamma_S$ is compact and Hausdorff. 
    Define
    \[
        \bar\psi\vcentcolon=
        \int_{\Gamma_S}\sigma R^*\psi\,\d\nu(\sigma,R),
    \]
    where $\nu$ is the unique Haar measure on $\Gamma_S$ normalized so that $\nu(\Gamma_S)=1$.
    
    We check that $\bar\psi$ is a constant-coefficient calibration satisfying
    \begin{equation}\label{eq:selector_invariant}
        G(\bar\psi)=S, \qquad
        \sigma R^*\bar\psi=\bar\psi
        \quad\text{for all }(\sigma,R)\in\Gamma_S.
    \end{equation}
    Indeed, each form $\sigma R^*\psi$ has constant coefficients and comass $1$, so $\bar\psi$ has constant coefficients and comass at most $1$.
    If $\xi\in S$, then $\sigma R_\#\xi\in S$ for every $(\sigma,R)\in\Gamma_S$, and hence $(\sigma R^*\psi)(\xi)=\psi(\sigma R_\#\xi)=1$.
    Thus $\bar\psi(\xi)=1$ for all $\xi\in S$; in particular, $\bar\psi$ has comass $1$ and $S\subset G(\bar\psi)$.
    Conversely, suppose that $\bar\psi(\xi)=1$ for some $\xi\in\mathbf{Gr}_m^+(\R^n)$.
    Then the function $(\sigma,R)\mapsto (\sigma R^*\psi)(\xi)$ is continuous on $\Gamma_S$ and everywhere bounded above by $1$.
    Since its average on $\Gamma_S$ is $1$, the integrand is identically equal to $1$ on $\Gamma_S$.
    Evaluating at the identity element gives $\psi(\xi)=1$, i.e., $\xi\in G(\psi)=S$. 
    Therefore $G(\bar\psi)=S$.
    Finally, the identity $\sigma R^*\bar\psi=\bar\psi$ follows from the invariance of the Haar measure $\nu$.

    We next prove that $F_{t_1}$ has finite $\bar\psi$-calibration energy.
    Choosing $\rho>0$ as in \eqref{eq:finite_complement_of_S}, and smaller if necessary so that the $\rho$-neighborhoods of the elements of $G(\phi)$ are pairwise disjoint, we have
    $\dist(\xi,S)=\dist(\xi,G(\phi))$ whenever $\dist(\xi,S)<\rho$.
    Therefore, using the elementary estimate $1-\bar{\psi}(\xi)\le C_{\bar{\psi}}\dist(\xi,G(\bar{\psi}))^2$ (which will be shown in \Cref{prop:constant_coeff_upper_bound}), we obtain
    \begin{align}
        \E_{\bar\psi}[F_{t_1}]
        &= \int_{[\dist(\tau,S)<\rho]}(1-\bar\psi(\tau))\d\mathrm{vol}\big|_{t_1} + \int_{[\dist(\tau,S)\ge\rho]}(1-\bar\psi(\tau))\d\mathrm{vol}\big|_{t_1} \\
        &\le C_{\bar\psi}
        \int_M\dist(\tau,G(\phi))^2\d\mathrm{vol}\big|_{t_1} + 2\,\mathrm{vol}\bigl(\{p\in M\mid
        \dist(\tau(p,t_1),S)\ge\rho\}\bigr)\\
        &< \infty .
    \end{align}
    Thus \Cref{thm:main_calibration} applies to the calibration $\bar\psi$, and
    in particular $\E_{\bar\psi}[F_{t_2}]<\infty$.

    Write $\Psi(x)=Qx+a$ with $Q\in O(n)$, $a\in\R^n$.
    We claim that
    \begin{equation}\label{eq:S_preserved}
        \sign(\Xi) Q_\#S=S.
    \end{equation}
    Let $\xi_0\in S$. Suppose, by contradiction, that $\sign(\Xi) Q_\#\xi_0\notin S$.
    Then we can choose $r>0$ so small that the $r$-neighborhood $U_r$ of $\sign(\Xi) Q_\#\xi_0$ in $\mathbf{Gr}_m^+(\R^n)$ has positive distance from $S=G(\bar\psi)$; therefore, there is $c>0$ such that
    \begin{equation}
        1-\bar\psi(\xi) \geq c \qquad \text{for all }\xi \in U_r.
    \end{equation}
    On the other hand, since $\xi_0\in S$, reducing $r>0$ if necessary, the set
    \[
        W_r = W_{r,t}\vcentcolon=
        \bigl\{p\in M\mid \dist(\tau(p,t),\xi_0)<r\bigr\}
    \]
    has infinite volume at time $t=t_1$. 
    By the breather relation, we have
    \[
    \mathrm{vol}(\Xi^{-1}(W_r))|_{t_2}=\lambda^m \mathrm{vol}(W_r)|_{t_1}=\infty,
    \]
    and for any $q\in \Xi^{-1}(W_r)|_{t_2}$ the plane $\tau(q,t_2)$ is contained in $U_r$.
    Thus we obtain
    \[
        \E_{\bar\psi}[F_{t_2}]
        \ge
        \int_{\Xi^{-1}(W_r)}
        (1-\bar\psi(\tau))\d\mathrm{vol}\big|_{t_2}
        \ge
        c \, \mathrm{vol}\bigl(\Xi^{-1}(W_r)\bigr)\big|_{t_2}
        =
        \infty,
    \]
    a contradiction. 
    Hence $\sign(\Xi) Q_\#\xi_0\in S$ for every $\xi_0\in S$.
    Since $S$ is finite and the signed action $\xi\mapsto\sign(\Xi) Q_\#\xi$ is injective, equation \eqref{eq:S_preserved} follows.

    By the invariances in \eqref{eq:S_preserved} and \eqref{eq:selector_invariant}, we have $\sign(\Xi)\Psi^*\bar\psi=\bar\psi$, i.e., the breather $F$ is $\bar\psi$-preserving.
    The assertion then follows by applying \Cref{thm:no_breather_1}.
\end{proof}

The assumptions of \Cref{thm:no_breather_2} are satisfied, for instance, in the unidirectional case.
More nontrivial examples include the block examples as in \Cref{rem:block_example}, and also the double-point calibrations in \cite{MR726326}.

\subsection{Convergence of two-dimensional immortal solutions}
   
Finally we study the long-time behavior of solutions with finite calibration energy, assuming that the flow is immortal, i.e., $T=\infty$.
Even under the immortality assumption, the long-time behavior can be quite subtle; for example, there exist codimension-one immortal solutions given by entire convex graphs for which the maximal mean curvature diverges as $t\to\infty$ \cite{MR4113204,MR4442627}.

In contrast to such wild behavior, the results in this section illustrate the regularizing effect of the calibration energy.
Before we state and prove our convergence results, we review the following consequences of the fundamental work of M\"uller--\v{S}ver\'ak \cite{MR1366547}, including the Gauss--Bonnet formula; see \Cref{lem:MS_GB} below.
For a proper immersion $F\colon M\to\R^n$ of a surface $M$, we write
\begin{align}
    \Theta(F)\vcentcolon = \lim_{r\to\infty}\frac{\mu(B_r(0))}{\pi r^2}
\end{align}
provided the limit exists. This $\Theta$ is called the asymptotic area ratio; geometrically, it counts the total number of ends, counted with their multiplicities. Let $K$ denote the Gaussian curvature, and $\chi(M)$ denote the Euler characteristic of $M$.

\begin{lemma}\label{lem:MS_GB}
    Let $F\colon M^2\to\R^n$ be a complete, connected, and oriented immersion such that $\int_M|A|^2\d\mathrm{vol}<\infty$. Then $F$ is proper, $\Theta(F)\in\N_0$, and
    \begin{equation}
        \int_{M} K\d\mathrm{vol} = 2\pi\big( \chi(M)-\Theta(F)\big).
    \end{equation}
    Moreover, if $M$ is noncompact, then $\int_M K\d\mathrm{vol} \leq 0$, where equality holds if and only if $M\cong \R^2$ and $\Theta(F)=1$.
\end{lemma}

\begin{proof}
    Fix $p_0\in M$. 
    By \cite[Corollary 4.2.5]{MR1366547}, we have
    \begin{align}\label{eq:MS_425}
        \lim_{\mathrm{dist}_M(p_0,p)\to\infty} \frac{\mathrm{dist}_M(p_0,p)}{|F(p_0)-F(p)|}=1.
    \end{align}
    The properness of $F$ follows directly. 
    By 
    \cite[Corollary 4.2.5]{MR1366547}, we have the Gauss--Bonnet formula
    \begin{align}\label{eq:Shiohama_GB}
        \int_{M}K\d\mathrm{vol} = 2\pi (\chi(M)-m_*),
    \end{align}    
    where $m_* = \sum_{i=1}^q m_i$, and $m_i$ is the multiplicity of the $i$-th end. 
    Let $B_r^M(p_0)\subset M$ be the geodesic ball of radius $r>0$. 
    The proof of \cite[Corollary 4.2.5]{MR1366547} yields the identity for the intrinsic asymptotic area ratio
    \begin{align}
    \pi m_* = \lim_{r\to\infty} \frac{\mathrm{vol}(B_r^M(p_0))}{r^2}.    
    \end{align}
    Since $F$ is an isometry, for any $r>0$, we have $
        B_r^{M}(p_0)\subset F^{-1} (B_r(F(p_0))).$
    Moreover, given $\varepsilon>0$, \eqref{eq:MS_425} implies
    that $F^{-1}(B_r(F(p_0))) \subset B^{M}_{(1+\varepsilon)r}(p_0)$ for $r\geq r(\varepsilon,p_0)>0$ sufficiently large,
    so that
    \begin{align}
        \lim_{r\to\infty} \frac{\mathrm{vol}(B_r^M(p_0))}{r^2} &\leq \liminf_{r\to\infty} \frac{\mu(B_r(F(p_0)))}{r^2}\\
        &\leq \limsup_{r\to\infty} \frac{\mu(B_r(F(p_0)))}{r^2} \leq  (1+\varepsilon)^2  \lim_{r\to\infty} \frac{\mathrm{vol}(B^M_{(1+\varepsilon)r} (p_0))}{(1+\varepsilon)^2r^2}.
    \end{align}
    By the arbitrariness of $\varepsilon>0$, we conclude that
    \begin{align}
        \lim_{r\to\infty} \frac{\mathrm{vol}(B_r^M(p_0))}{r^2} =\lim_{r\to\infty} \frac{\mu(B_r(F(p_0)))}{r^2} =\pi \Theta(F).
    \end{align}
    Since $M$ is noncompact, we have $\chi(M)\leq 1$ and $m_*\geq 1$. Hence, $\int_M K \d\mathrm{vol}\leq 0$ with equality if and only if $\chi(M)=1$, so $M\cong \R^2$, and $m_*=\Theta(F)=1$.
\end{proof}

Now we state our main theorem in this section, which gives a general convergence theorem in the two-dimensional genus-zero case under finite total curvature and unit asymptotic area ratio.

\begin{theorem}\label{thm:2d-conv}
    Let $M^2\cong\R^2$ and let $F\colon M\times [0,\infty)\to\R^n$ be an  immortal oriented mean curvature flow with purely normal velocity $\partial_t F=H$ such that $\sup_{M\times [0,T]}|A|<\infty$ for all $T<\infty$. 
    Let $\phi$ be a coflat calibration of degree $2$ on $\R^n$.
    Suppose that the initial datum $F_0\colon M\to\R^n$ is complete and satisfies 
    \begin{equation}\label{eq:02-03-26_3}
        \E_\phi[F_0]<\infty, \quad \int_M|A|^2 \d\mathrm{vol} <\infty, \quad \Theta(F_0) = 1.
    \end{equation}
    Then $F$ is asymptotically planar in the sense that 
    \begin{align}\label{eq:02-03-26_2}
        \lim_{t\to\infty}\Vert \nabla^\ell A\Vert_{\infty}=0 \text{ for all } \ell\in\N_0.
    \end{align}
    
    Moreover, if $\phi$ has constant coefficients, then
    \begin{align}\label{eq:26-05_01}
    \lim_{t\to\infty}\Vert 1-\phi(\tau)\Vert_\infty = 0.    
    \end{align}
    If in addition $G(\phi) = \{E\}$, then $
        \lim_{t\to\infty}\Vert \tau-E\Vert_\infty = 0
    $
    and $F_t$ locally smoothly converges, after translation and reparametrization, to the plane $E$ as $t\to\infty$.
\end{theorem}
In \eqref{eq:02-03-26_2} and in the sequel, we also use $\nabla$ for the normal connection along the immersion $F_t$ and the induced metric $g_t=F_t^*\langle\cdot,\cdot\rangle$ to measure the length of tensors.

The proof of \Cref{thm:2d-conv} combines \Cref{cor:main_calibration_detailed} with \Cref{lem:MS_GB} and a small-energy convergence theorem, see \Cref{thm:2D_small_energy} below. 
The key point is that finite total curvature and $\Theta(F_0)=1$ force the total Gaussian curvature to vanish, which converts the dissipation of mean curvature into that of second fundamental form. 
In particular, the method is quite different from codimension-one results based on maximum principles, such as convergence to translators \cite{MR2321890} or to self-expanders \cite{MR4170224}.

\begin{remark}\label{rem:V_by_monotonicity}
    In \Cref{thm:2d-conv}, we need not assume uniformly finite local volume as it automatically holds for a complete, connected, and oriented immersion $F\colon M^2\to\R^n$ with finite total curvature.
    Indeed, \Cref{lem:MS_GB} yields that $F$ is proper and the monotonicity formula for the Willmore energy \cite[(A.6)]{MR2119722} implies that there exists a universal $C>0$ such that for all $x_0\in\R^n$ and $R\ge 1$,
    \begin{align}\label{eq:03-03-26}
       \mu(B_1(x_0)) \le C\Big(\frac{\mu(B_R(x_0))}{R^2}+\int_{B_R(x_0)}|H|^2 \d\mu\Big),
    \end{align}
    cf.\ \eqref{eq:integral_notation} for the notation.
    In particular, since $|H|^2\le 2|A|^2$, sending $R\to\infty$ in \eqref{eq:03-03-26} yields \eqref{eq:intro_M_finite}.
\end{remark}

We now collect some preparatory results for the proof of \Cref{thm:2d-conv}.
Throughout the rest of this section, we use $C\in(0,\infty)$ to denote a universal constant that is allowed to change from line to line.
We first recall the following inequality for the evolution of the second fundamental form and its derivatives.
\begin{lemma}[cf.\ {\cite[Lemma 2.4]{MR4810572}}]\label{lem:dt_nabla_ell_A}
    Let $F\colon M^m\times[0,T]\to\R^n$ be a purely normal mean curvature flow. 
    Then for $\ell\in\N_0$,
    \begin{align}
        (\partial_t - \Delta)|\nabla^\ell A|^2 + 2|\nabla^{\ell+1}A|^2 \leq C(\ell,m)\sum_{i+j+k=\ell}|\nabla^i A||\nabla^j A||\nabla^k A||\nabla^\ell A|.
    \end{align}
\end{lemma}

A crucial ingredient is the following localized evolution of $|A|^2$.
\begin{lemma}\label{lem:2d-conv_1}
    Let $F\colon M^m\times [0,T]\to\R^n$ be a proper, purely normal mean curvature flow.
    Then for $\gamma=\tilde \gamma \circ F$ with $\tilde\gamma\in C_c^\infty(\R^n)$,
    \begin{align}
        &\partial_t \int_M|A|^2\gamma^2\d\mathrm{vol} + \int_M |\nabla A|^2\gamma^2\d\mathrm{vol}\\
        &\leq C(m) \int_M|A|^4 \gamma^2\d\mathrm{vol} + C(m) \int_{M}|A|^2|\mathrm{D}\tilde\gamma\circ F|^2\d\mathrm{vol}.
    \end{align}
\end{lemma}

\begin{proof}
    As in \eqref{eq:09-04_03}, we may differentiate under the integral due to properness. Using  \eqref{eq:time_derivative_volume_form}, \Cref{lem:dt_nabla_ell_A} with $\ell=0$, and integration by parts for the Laplacian term, we find
   \begin{align}
       &\partial_t \int_M|A|^2\gamma^2\d\mathrm{vol} +2\int_M|\nabla A|^2 \gamma^2\d\mathrm{vol} \\
       & \leq C(m)\int_M|A|^4 \gamma^2\d\mathrm{vol} + 4 \int_M |A| |\nabla A| |\nabla \gamma|\gamma\d\mathrm{vol}  + 2\int_M|A|^2 \langle \mathrm{D}\tilde\gamma\circ F, H\rangle \gamma\d\mathrm{vol}.
   \end{align}
   We now use Young's inequality for the second and third term.
   Absorbing the term involving $|\nabla A|^2$, and noting $|\nabla\gamma|=|(\mathrm{D}\tilde\gamma\circ F)^\top| \leq |\mathrm{D}\tilde\gamma\circ F|$ and $|H|^2\leq m|A|^2$, the statement then follows.
\end{proof}

For $m=2$, we now use this to conclude monotonicity of $\int_M|A|^2\d\mathrm{vol}$ if it is sufficiently small.
We will make repeated use of the following consequence of the Michael--Simon Sobolev inequality.

\begin{lemma}\label{lem:Michael-Simon-2D}
    There exists a universal constant $\varepsilon_0\in (0,1)$ with the following property.
    Let $F\colon M^2\to\R^n$ be an immersed surface (without boundary) and let $u\in C_c^\infty(M)$ with  $\int_{[|u|>0]}|H|^2\d\mathrm{vol}\leq \varepsilon_0$.
    Then
    \begin{align}
        \int_M u^2\d\mathrm{vol}\leq c \Big(\int_M|\nabla u|\d\mathrm{vol}\Big)^2,
    \end{align}
    where $c\in (0,\infty)$ is a universal constant.
\end{lemma}

\begin{proof}
    The two-dimensional case of the Michael--Simon Sobolev inequality \cite[Theorem 2.1]{MR344978} yields, for a universal constant $C_{\mathrm{MS}}>0$ (independent of $n$),
    \begin{align}
        \int_M u^2\d\mathrm{vol}&\leq C_{\mathrm{MS}} \Big(\int_M|\nabla u|\d\mathrm{vol} + \int_M |H| |u|\d\mathrm{vol}\Big)^2\\
        &\leq 2C_{\mathrm{MS}}\Big(\int_M|\nabla u|\d\mathrm{vol}\Big)^2 + 2C_{\mathrm{MS}} \int_{[|u|>0]}|H|^2\d\mathrm{vol} \int_M|u|^2\d\mathrm{vol}.
    \end{align}
    For $\varepsilon_0$ sufficiently small (depending on $C_{\mathrm{MS}}>0$), the claim follows.
\end{proof}

\begin{lemma}\label{lem:2d-conv_2}
    Let $F\colon M^2\times [0,T]\to\R^n$ be a proper, purely normal mean curvature flow.
    Let $\tilde\gamma\in C_c^\infty(\R^n)$, $0\leq \tilde\gamma\leq 1$, and $\gamma=\tilde\gamma\circ F$. 
    There exist universal constants $\varepsilon_0\in(0,1)$ and $C\in (0,\infty)$ such that if $\int_{[\gamma >0]} |A|^2\d\mathrm{vol} \leq \varepsilon_0$ at some $t$, then
    \begin{align}
         &\partial_t \int_M|A|^2\gamma^2\d\mathrm{vol} + \frac12\int_M|\nabla A|^2\gamma^2\d\mathrm{vol} \leq C\Vert \mathrm{D}\tilde\gamma\Vert_\infty^2 \int_{[\gamma>0]}|A|^2\d\mathrm{vol}.
    \end{align}
\end{lemma}
\begin{proof}
    We estimate the first and second term on the right hand side of \Cref{lem:2d-conv_1}. For the second term, we have
    \begin{align}
        \int_M|A|^2 | \mathrm{D}\tilde\gamma\circ F|^2\d\mathrm{vol} \leq \Vert \mathrm{D}\tilde\gamma\Vert_\infty^2 \int_{[\gamma>0]}|A|^2\d\mathrm{vol}.
    \end{align}
    For the first term, by \Cref{lem:Michael-Simon-2D} and $|H|^2\leq 2|A|^2$, we have for small $\varepsilon_0$,
    \begin{align}
        \int_M|A|^4\gamma^2\d\mathrm{vol} &\leq c\Big(\int_M( 2|A||\nabla A| \gamma + |A|^2|\nabla\gamma|)\d\mathrm{vol}\Big)^2\\
        &\leq 8c \int_{[\gamma>0]}|A|^2\d\mathrm{vol} \int_M|\nabla A|^2 \gamma^2\d\mathrm{vol} + 2c \Vert \mathrm{D}\tilde\gamma\Vert_\infty^2 \big(\int_{[\gamma>0]} |A|^2\d\mathrm{vol}\big)^2,\label{eq:19-12-1}
    \end{align}
    by H\"older's inequality. Reducing $\varepsilon_0>0$, the first term on the right hand side of \eqref{eq:19-12-1} can be absorbed by the second term on the left hand side of \Cref{lem:2d-conv_1}.
\end{proof}

\begin{lemma}\label{lem:2d-conv_3}
    Let $F\colon M^2\times [0,T]\to\R^n$ be a proper, purely normal mean curvature flow, where the initial datum has uniformly finite local volume and $\sup_{M\times [0,T]}|A|<\infty$. 
    If $\int_M|A|^2\d\mathrm{vol} \vert_{t=0}<\infty$, then also $\int_M|A|^2\d\mathrm{vol} < \infty$ for all $t\in [0,T]$. Moreover, there is a universal constant $\varepsilon_0\in(0,1)$ such that if $\int_M|A|^2\d\mathrm{vol}\vert_{t=0}<\varepsilon_0$, then 
    \begin{align}
         &\int_M|A|^2\d\mathrm{vol} + \frac12 \int_0^t\int_M|\nabla A|^2\d\mathrm{vol}\d t' \leq \varepsilon_0
    \end{align}
    for all $t\in [0,T]$.
\end{lemma}
\begin{proof}    
    Let $x_0\in\R^n$, $R>0$ and $\chi_{B_R(x_0)}\leq \tilde\gamma\leq \chi_{B_{2R}(x_0)}$ with $\Vert \mathrm{D}\tilde\gamma\Vert_\infty \leq \frac{C}{R}$.   
    Writing $\Lambda\vcentcolon = \sup_{M\times [0,T]}|A|$ and integrating \Cref{lem:2d-conv_1} over $[0,t]$ implies (cf.\ \eqref{eq:integral_notation})
     \begin{align}
        &\int_{B_R(x_0)} |A|^2 \d\mu_t \\
        &\leq \int_{B_{2R}(x_0)} |A|^2 \d\mu_0 + C(\Lambda^2 + R^{-2}) \int_0^t \int_{B_{2R}(x_0)}|A|^2\d\mu_{t'} \d t'.\label{eq:02-02-26_2}
    \end{align}
    Note that by \Cref{lem:uniform_bounded_local_volume_new}, $F_t$ has uniformly finite local volume for all $t\in [0,T]$.
    For $t\in [0,T]$ and $R>0$, arguing as in \eqref{eq:R-unif-bounded-local-volume}, the following function is well defined:
    \begin{align}\label{eq:def_kappa}
          \kappa(t,R) \vcentcolon = \sup_{x_0\in\R^n}\int_{B_R(x_0)} |A|^2\d\mu_t \leq \Lambda^2\sup_{t\in[0,T]} \sup_{x_0\in\R^n}\mu_t(B_R(x_0)).
    \end{align}
    In particular, for any $R>0$, $\kappa(\cdot,R)$ is bounded and lower semicontinuous.
    With $\Gamma=\Gamma(n)$ as in the proof of \Cref{thm:main_calibration_detailed}, we obtain
    \begin{align}\label{eq:kappa_ineq}
        \kappa(t,R) \leq \Gamma\kappa(0,R) + C\Gamma (\Lambda^2+R^{-2}) \int_0^t \kappa(t',R)\d t'.
    \end{align}
    Gr\"onwall's inequality implies that for $t\in [0,T]$ we have \begin{align}\label{eq:curvature_concentration_estimate}
        \kappa(t,R) \leq \Gamma \kappa(0,R) e^{C\Gamma(\Lambda^2+R^{-2})t}.
    \end{align}
    In particular, sending $R\to\infty$, we conclude that $\sup_{t\in [0,T]} \Vert A\Vert_{L^2}<\infty$, yielding the first part of the statement.

    We now assume $\int_M|A|^2\d\mathrm{vol}\vert_{t=0}<\varepsilon_0$ with $\varepsilon_0>0$ as in \Cref{lem:2d-conv_2}. We set
    \begin{align}
        T_* \vcentcolon = \sup\{ T' \in [0,T]\mid \int_M|A|^2\d\mathrm{vol} \leq \varepsilon_0 \text{ for all }t\in [0,T']\}.
    \end{align}
    Sending $R\to\infty$ in \eqref{eq:02-02-26_2}, we find  $T_*>0$.
    Integrating \Cref{lem:2d-conv_2} over $[0, T_*]$ and sending $R\to\infty$ yield
    \begin{align}\label{eq:02-02-26_3}
         &\int_M|A|^2\d\mathrm{vol}\big\vert_{t=T_*} + \frac12 \int_0^{T_*}\int_M|\nabla A|^2\d\mathrm{vol}\d t \leq \int_M|A|^2\d\mathrm{vol}\big\vert_{t=0}.
    \end{align}
    If $T_*<T$, then by assumption and \eqref{eq:02-02-26_3} we have $\int_M|A|^2\d\mathrm{vol}\vert_{t=T_*}<\varepsilon_0$.
    Applying \eqref{eq:02-02-26_2} with $R\to\infty$ and $t=0$ replaced by $t=T_*$, this contradicts the maximality of $T_*$. 
    Thus $T_*=T$ and the statement follows from \eqref{eq:02-02-26_3} (as $T$ can be replaced by any $t\in[0,T]$).
\end{proof}

As a consequence, we obtain a small energy result. Note that we will not require finite calibration energy in the following result; however, this only allows us to conclude subconvergence as $t\to\infty$. Here, we work with initial data $F_0\colon M\to\R^n$ that are properly immersed and have \emph{bounded geometry} (of all orders), i.e., such that 
\begin{align}\label{eq:bounded_geometry}
	\Vert \nabla^\ell A\Vert_\infty<\infty\quad \text{ for all }\ell\in \N_0.
\end{align}
In the class of immersions of bounded geometry, it is possible to show short-time existence of a purely normal mean curvature flow by standard arguments. This solution is proper with bounded geometry locally uniformly in time and is unique with this property. We should note that the general form of short-time existence is not explicitly available in the literature; however, it has been proven for the case $n=m+1$ \cite[Theorem 4.2]{MR1117150} or $m=1$ \cite{MR5008349}.

\begin{theorem}\label{thm:2D_small_energy}
    There exists $\varepsilon_0 \in (0,1)$ with the following property.
    Let $M\cong \R^2$ and $F_0\colon M\to\R^n$ be a complete immersion with bounded geometry \eqref{eq:bounded_geometry} such that
    \[
    \int_{M}|A|^2\d\mathrm{vol}< \varepsilon_0.
    \]
    Then, there exists a unique immortal, proper, and purely normal mean curvature flow $F\colon M\times [0,\infty)\to\R^n$ starting from $F_0$ such that $\sup_{M\times[0,T]}|A|<\infty$ for all $T<\infty$. Moreover, $F_t$ subconverges, after conformal reparametrization and translation, locally smoothly to an embedded plane as $t\to\infty$ with
    \begin{align}\label{eq:24-04-26_03}
    \lim_{t\to\infty} \Vert \nabla^\ell A\Vert_\infty =0 \quad\text{ for all }\ell\in\N_0.
    \end{align}
    Moreover, we have the estimates
    \begin{align}\label{eq:05-02-26_2}
        \int_M|A|^2\d\mathrm{vol} +  t\int_M|\nabla A|^2\d\mathrm{vol} + t^2 \int_M|\nabla^2 A|^2\d\mathrm{vol} \leq C(n)\varepsilon_0 \quad \text{ for }t\geq 0.
    \end{align}
\end{theorem}

\begin{proof}
    By \Cref{lem:MS_GB} and \Cref{rem:V_by_monotonicity}, the immersion $F_0$ is proper and has uniformly finite local volume. By short-time existence, there exists a solution on some maximal interval $[0,T_{\mathrm{max}})$ with $T_{\mathrm{max}}\in (0,\infty]$ which is proper and has bounded geometry locally uniformly in time. In particular, for all $T<T_{\mathrm{max}}$, we have
    \begin{align}\label{eq:05-03-26}
        \sup_{M\times [0,T]}|\nabla^\ell A|<\infty \quad \text{for all }\ell\in\N_0. 
    \end{align}
    Applying \Cref{lem:2d-conv_3} on $[0,T]\subset[0,T_{\mathrm{max}})$, we find
    \begin{align}\label{eq:20-12-1}
       \int_M|A|^2\d\mathrm{vol}\big\vert_{t=T}+\frac{1}{2}\int_0^{T} \int_M|\nabla A|^2\d\mathrm{vol}\d t \leq \varepsilon_0.
    \end{align}

    We now prove that for $\ell=1,2$,
    \begin{align}\label{eq:260306_2}
        &T^\ell\int_M |\nabla^\ell A|^2\d\mathrm{vol}\big\vert_{t=T} + \int_0^T t^\ell \int_M|\nabla^{\ell+1} A|^2\d\mathrm{vol}\d t \leq C(n)\varepsilon_0.
    \end{align}
    To this end we introduce a cutoff function: 
    Let $R>0$, $x_0\in\R^n$, and $\tilde\gamma \in C_c^\infty(\R^n)$ such that $\chi_{B_R(x_0)}\leq \tilde
    \gamma\leq \chi_{B_{2R}(x_0)}$, $\|\mathrm{D}\tilde\gamma\|_\infty\leq C/R$, and $\|\mathrm D^2 \tilde\gamma\|_\infty\leq C/R^2$, and define $\gamma\vcentcolon =\tilde\gamma\circ F$.
    
    We first prove \eqref{eq:260306_2} for $\ell=1$.
    Applying \Cref{lem:dt_nabla_ell_A} with $\ell=1$ and arguing as in the proof of \Cref{lem:2d-conv_1}, we have
    \begin{align}
        &\partial_t \int_M|\nabla A|^2\gamma^4\d\mathrm{vol} + 2\int_M|\nabla^2 A|^2 \gamma^4\d\mathrm{vol}
        \\
        &\leq C \Vert \mathrm{D}\tilde\gamma\Vert_\infty \int_M |\nabla^2 A| |\nabla A| \gamma^3\d\mathrm{vol} + C \int_M|A|^2|\nabla A|^2\gamma^4\d\mathrm{vol} \\
        & \quad + 4\int_M|\nabla A|^2 \gamma^3 \langle \mathrm{D}\tilde\gamma\circ F, H\rangle\d\mathrm{vol} \\
        &\leq C \Vert \mathrm{D}\tilde\gamma\Vert_\infty^2 \int_M |\nabla A|^2 \gamma^2\d\mathrm{vol} + C \int_M|A|^2|\nabla A|^2\gamma^4\d\mathrm{vol} + \frac12 \int_M|\nabla^2 A|^2 \gamma^4\d\mathrm{vol}.
    \end{align}
    We estimate the second term using \Cref{lem:Michael-Simon-2D}, yielding
    \begin{align}
        &\int_M|A|^2|\nabla A|^2 \gamma^4\d\mathrm{vol} \ \\
        &\leq c \Big(\int_M (|\nabla A||\nabla A|\gamma^2 + |A| |\nabla^2 A| \gamma^2 + 2|A| |\nabla A| \gamma |\nabla\gamma|\d\mathrm{vol}\Big)^2\\
        &\leq 3c \int_{[\gamma>0]} |\nabla A|^2\d\mathrm{vol} \int_M |\nabla A|^2 \gamma^4\d\mathrm{vol}+ 3c\int_{[\gamma>0]}|A|^2\d\mathrm{vol} \int_M |\nabla^2 A|^2 \gamma^4\d\mathrm{vol} \\
        &\quad + 3\cdot 4 c \Vert \mathrm{D}\tilde\gamma\Vert_\infty^2\int_{[\gamma>0]}|A|^2\d\mathrm{vol} \int_M |\nabla A|^2\gamma^2\d\mathrm{vol}.
    \end{align}
    Using \eqref{eq:20-12-1}, reducing $\varepsilon_0>0$ if necessary, and absorbing, we conclude that
     \begin{align}
        &\partial_t \int_M|\nabla A|^2\gamma^4\d\mathrm{vol} + \int_M|\nabla^2 A|^2 \gamma^4\d\mathrm{vol}
        \\
        &\leq C R^{-2} \int_M  |\nabla A|^2 \gamma^2\d\mathrm{vol} + C \int_{[\gamma>0]} |\nabla A|^2\d\mathrm{vol} \int_M |\nabla A|^2 \gamma^4\d\mathrm{vol}.\label{eq:20-12-2}
    \end{align}
    Now, due to \eqref{eq:20-12-1}, there exists a sequence $t_j\to 0^+$ such that
    \begin{align}\label{eq:05-03-26_03}
        \lim_{j\to\infty} t_j \int_M|\nabla A|^2\d\mathrm{vol}\big\vert_{t=t_j} =0.
    \end{align}
    Multiplying \eqref{eq:20-12-2} by $t$, and integrating on $[t_j,T]$ for $T<T_{\mathrm{max}}$, we find
    \begin{align}
        &T\int_M|\nabla A|^2 \gamma^4\d\mathrm{vol}\big\vert_{t=T} + \int_{t_j}^T t \int_M|\nabla^2 A|^2 \gamma^4\d\mathrm{vol}\d t \\
        &\leq C \int_{t_j}^T \Big(\int_{[\gamma>0]}|\nabla A|^2\d\mathrm{vol}+R^{-2}\Big)t\int_M|\nabla A|^2 \gamma^2\d\mathrm{vol} \d t \\
        &\quad + \underbrace{t_j\int_M|\nabla A|^2\d\mathrm{vol}\big\vert_{t=t_j}}_{=:I_1} + \underbrace{\int_{t_j}^T\int_M|\nabla A|^2\d\mathrm{vol}\d t}_{=:I_2} .\label{eq:05-03-26_02}
    \end{align}As in \eqref{eq:def_kappa}, using \Cref{lem:uniform_bounded_local_volume_new} and \eqref{eq:05-03-26}, we find that for all $R>0$ the function
    \begin{align}
        \kappa_1(T,R)&\vcentcolon= \sup_{x_0\in\R^n} \left( T\int_{B_R(x_0)}|\nabla A|^2 \d\mu_T + \int_{t_j}^T t \int_{B_R(x_0)}|\nabla^2 A|^2\d\mu_t\d t \right)
    \end{align}
    is lower semicontinuous and locally bounded with respect to $T\in [t_j,T_{\mathrm{max}})$. 
    As in  \eqref{eq:02-02-26_2}--\eqref{eq:kappa_ineq}, estimate \eqref{eq:05-03-26_02} and a covering argument  yield that for all $T\in [t_j, T_{\mathrm{max}})$ we have
    \begin{align}
        \kappa_1(T,R) \leq I_1+I_2 + C\Gamma\int_{t_j}^T \Big( \int_M|\nabla A|^2\d\mathrm{vol}+R^{-2} \Big)\kappa_1(t,R) \d t.\label{eq:05-03-26_04}
    \end{align}
    Since $I_1$ is independent of $T$ while $I_2$ is increasing in $T$ and finite-valued (by \eqref{eq:20-12-1}), we can apply Gr\"onwall's inequality to conclude that for all $T\in [t_j, T_{\mathrm{max}})$
    \begin{align}
        \kappa_1(T,R) &\leq \big( I_1+I_2 \big) \exp\bigg( C \Gamma \int_{t_j}^T \Big(\int_M|\nabla A|^2\d\mathrm{vol}+R^{-2} \Big)\d t \bigg)\\
        &= \big( I_1+I_2 \big) \exp \bigg( C \Gamma \Big( I_2+R^{-2}(T-t_j) \Big) \bigg).\label{eq:02-02_1}
    \end{align}
    Noting that $I_2 \leq C\varepsilon_0$ by \eqref{eq:20-12-1}, sending $R\to\infty$ and $j\to\infty$, which yields $I_1\to0$ by \eqref{eq:05-03-26_03}, and using $\varepsilon_0<1$, we obtain \eqref{eq:260306_2} with $\ell=1$ for all $T\in [0,T_{\mathrm{max}})$.
    
    Next, we prove \eqref{eq:260306_2} for $\ell=2$.
    We use \Cref{lem:dt_nabla_ell_A} with $\ell=2$ to find
    \begin{align}
        (\partial_t -\Delta)|\nabla^2 A|^2 + 2 |\nabla^3 A|^2 \leq C |A|^2|\nabla^2 A|^2 + C |A| |\nabla A|^2|\nabla^2A|,
    \end{align}
    so that we have
    \begin{align}
        &\partial_t \int_M|\nabla^2 A|^2 \gamma^6\d\mathrm{vol} + 2 \int_M|\nabla^3 A|^2 \gamma^6\d\mathrm{vol} \\
        &\leq C\int_M|A|^2 |\nabla^2A|^2 \gamma^6\d\mathrm{vol}+ C \int_M|\nabla A|^4 \gamma^6\d\mathrm{vol} \\
        &\quad + C \Vert \mathrm{D}\tilde\gamma\Vert_\infty \int_M|\nabla^3 A||\nabla^2 A|\gamma^5\d\mathrm{vol} +6 \int_M|\nabla^2 A|^2 \gamma^5 \langle \mathrm{D}\tilde\gamma\circ F, H\rangle\d\mathrm{vol} \\
         &\leq C\int_M|A|^2 |\nabla^2A|^2 \gamma^6\d\mathrm{vol}+ C \int_M|\nabla A|^4 \gamma^6\d\mathrm{vol} \\
        &\quad + C \Vert \mathrm{D}\tilde\gamma\Vert_\infty^2 \int_M|\nabla^2 A|^2\gamma^4\d\mathrm{vol}+ \frac{1}{2}\int_M|\nabla^3 A|^2 \gamma^6\d\mathrm{vol}.
    \end{align}
    We now estimate the first term using \Cref{lem:Michael-Simon-2D} by
    \begin{align}
        &\int_M|A|^2 |\nabla^2A|^2\gamma^6\d\mathrm{vol}\\
        &\leq c\Big(\int_M|\nabla A||\nabla^2 A|\gamma^3\d\mathrm{vol} + \int_M|A| |\nabla^3 A| \gamma^3\d\mathrm{vol} + 3\Vert \mathrm{D}\tilde\gamma\Vert_\infty\int_M|A||\nabla^2 A| \gamma^2\d\mathrm{vol} \Big)^2 \\
        &\leq 3c \int_M|\nabla A|^2\d\mathrm{vol}\int_M |\nabla^2 A|^2 \gamma^6\d\mathrm{vol} + 3c \int_M|A|^2\d\mathrm{vol} \int_M|\nabla^3 A|^2 \gamma^6\d\mathrm{vol} \\
        &\quad + 3\cdot 9 c\Vert \mathrm{D}\tilde\gamma\Vert_\infty^2 \int_M|A|^2\d\mathrm{vol} \int_M|\nabla^2 A|^2 \gamma^4\d\mathrm{vol}.
    \end{align}
    For the second term, again using \Cref{lem:Michael-Simon-2D}, we have
    \begin{align}
        &\int_M|\nabla A|^4 \gamma^6\d\mathrm{vol} \\
        &\leq c\Big( 2 \int_M |\nabla^2 A| |\nabla A| \gamma^3\d\mathrm{vol} + 3\Vert \mathrm{D}\tilde\gamma \Vert_\infty \int_M|\nabla A|^2 \gamma^2\d\mathrm{vol}\Big)^2 \\
       &\leq  2 c\int_M|\nabla A|^2\d\mathrm{vol}\int_M|\nabla^2 A|^2 \gamma^6\d\mathrm{vol} + 2\cdot 9 c\Vert \mathrm{D}\tilde\gamma\Vert_\infty^2 \Big(\int_M|\nabla A|^2 \gamma^2\d\mathrm{vol}\Big)^2,
    \end{align}
    and using integration by parts, the last term is further estimated by
    \begin{align}
        & \Big(\int_M|\nabla A|^2 \gamma^2\d\mathrm{vol}\Big)^2 \\
        &\leq \int_M|A|^2\d\mathrm{vol} \int_M|\Delta A|^2 \gamma^4\d\mathrm{vol} + C \Vert \mathrm{D}\tilde\gamma\Vert_\infty^2\int_M|A|^2\d\mathrm{vol}\int_M|\nabla A|^2 \gamma^2\d\mathrm{vol}.
    \end{align}
    As $|\Delta A|^2\leq C|\nabla^2 A|^2$, after absorbing, we thus conclude that
    \begin{align}
        &\partial_t \int_M|\nabla^2 A|^2 \gamma^6\d\mathrm{vol} + \int_M|\nabla^3 A|^2 \gamma^6\d\mathrm{vol}\\
        &\leq C\Big(\int_M|\nabla A|^2\d\mathrm{vol}+ R^{-2}\Big)\int_M|\nabla^2 A|^2 \gamma^4\d\mathrm{vol} + CR^{-4} \int_M|\nabla A|^2 \gamma^2\d\mathrm{vol}.\label{eq:05-03-26_05}
    \end{align}
    As in \eqref{eq:05-03-26_03}, estimate \eqref{eq:260306_2} with $\ell=1$ implies that there exist $t_j\to 0^+$ such that
    \begin{align}\label{eq:260306_03}
        \lim_{t_j\to0^+} t_j^2 \int_M|\nabla^2 A|^2\d\mathrm{vol}\big\vert_{t=t_j}=0.
    \end{align}
    Multiplication of \eqref{eq:05-03-26_05} with $t^2$ and integration on $[t_j, T]$ yield
    \begin{align}
        &T^2 \int_M|\nabla^2 A|^2\gamma^6\d\mathrm{vol}\big\vert_{t=T} + \int_{t_j}^T t^2 \int_M|\nabla^3 A|^2 \gamma^6\d\mathrm{vol}\d t\\
        & \leq C \int_{t_j}^T \Big(\int_M|\nabla A|^2\d\mathrm{vol} + R^{-2}\Big)t^2 \int_M|\nabla^2 A|^2 \gamma^4\d\mathrm{vol}\d t \\
        &\quad + \underbrace{t_j^2 \int_M|\nabla^2 A|^2\d\mathrm{vol}\big\vert_{t=t_j}}_{=:I_3} + \underbrace{\int_{t_j}^T 2t\int|\nabla^2A|^2\d\mathrm{vol}\d t}_{=:I_4} + \underbrace{\frac{C}{R^{4}} \int_{t_j}^T t^2 \int_M|\nabla A|^2\d\mathrm{vol}\d t}_{=:I_5}.
    \end{align}
    We may now argue as in \eqref{eq:05-03-26_02}--\eqref{eq:02-02_1}: we apply Gr\"onwall's inequality to
    \begin{align}
        \kappa_2(T,R)&\vcentcolon= \sup_{x_0\in\R^n} \left( T^2\int_{B_R(x_0)}|\nabla^2 A|^2 \d\mu_T + \int_{t_j}^T t^2 \int_{B_R(x_0)}|\nabla^3 A|^2\d\mu_t\d t \right)
    \end{align}
    to conclude that
    \begin{align}
        \kappa_2(T,R) \leq (I_3+I_4+I_5) \exp\bigg( C\Gamma  \Big(I_2+R^{-2}(T-t_j)\Big)\bigg).
    \end{align}
    Note that $I_4\leq C(n)\varepsilon_0$ by \eqref{eq:260306_2} with $\ell=1$, and that the integral in $I_5$ is finite due to \eqref{eq:20-12-1}, so that $I_5\to0$ as $R\to\infty$.
    Recalling $I_2\leq C\varepsilon_0$, sending first $R\to\infty$ and then $j\to\infty$, which yields $I_3\to0$ by \eqref{eq:260306_03}, we finally obtain \eqref{eq:260306_2} with $\ell=2$. 

    Now, for $R\geq 1$, the $C^2$-norm of $\tilde\gamma$ is bounded by a uniform constant. Hence, by interpolation \cite[(4.9)]{MR1900754}, reducing $\varepsilon_0>0$, we have
    \begin{align}\label{eq:KS_Linfty}
        \Vert A\Vert_{L^\infty([\gamma=1])}^4\leq C \int_{[\gamma>0]}|A|^2\d\mathrm{vol}\Big(\int_{[\gamma>0]}|\nabla^2 A|^2\d\mathrm{vol}+ \int_{[\gamma>0]}|A|^2\d\mathrm{vol}\Big).
    \end{align}
    Sending $R\to\infty$, and using \eqref{eq:20-12-1} and \eqref{eq:260306_2} with $\ell=2$, we conclude that 
    \begin{align}\label{eq:05-02-26_4}
        \sup_{M\times [0,T_{\mathrm{max}})} |A|<\infty.
    \end{align}
    We now apply Ecker--Huisken's interior estimates \cite[Theorem 3.4 and Remark 3.6 (ii)]{MR1117150}. 
    Note that even though this result is only stated for $n=m+1$, essentially the same proof also works in higher codimensions. Thus, for all $\delta>0$, \eqref{eq:05-02-26_4} implies
    \begin{align}\label{eq:05-02-26}
        \sup_{M\times [\delta,T_{\mathrm{max}})}|\nabla^\ell A|<\infty\quad \text{for all }\ell\in\N_0.
    \end{align}        
    If $T_{\mathrm{max}}<\infty$, short-time existence thus allows us to extend the flow past $T_{\mathrm{max}}$. This contradicts maximality, so $T_{\mathrm{max}}=\infty$.
    
    For the subconvergence, we observe $\int_M|A|^2\d\mathrm{vol} <\varepsilon_0<8\pi$.       
    By \cite[Theorem 4.3.1]{MR1366547}, $F_t$ is embedded and admits a conformal reparametrization $\tilde{F}_t\colon \C\to\R^n$ with
    \begin{align}
        e^{-2 c(n,\varepsilon_0)}|z_1-z_2|\leq |\tilde{F}_t(z_1)-\tilde{F}_t(z_2)|\leq e^{c(n, \varepsilon_0)}|z_1-z_2|\qquad \text{for all } z_1,z_2\in \mathbf{C},
    \end{align}
    and the conformal factor $u_t$ satisfies
    \begin{align}\label{eq:conf_factor_MS}
        \Vert u_t \Vert_{\infty}\leq c(n, \varepsilon_0).
    \end{align}
    Comparing coordinate and covariant derivatives as in \cite[Proof of Theorem 1.2]{MR1900754}, and arguing by induction, \eqref{eq:05-02-26} implies that
    \begin{align}
        \sup_{t\geq 1}\Vert \mathrm D^\ell \tilde{F}_{t}\Vert_{\infty}<\infty
    \end{align}
    for all $\ell\in \N$. 
    Fix any sequences $t_j\to\infty$ and $\{z_j\}\subset \C$. 
    After passing to a subsequence (without relabeling), we find that there is some conformal embedding $F_\infty:\C\to\R^n$ with $F_\infty(0)=0$ such that
    \begin{align}\label{eq:260306_1}
        \tilde F_{t_j}(\cdot +z_j) - \tilde F_{t_j}(z_j) \to F_\infty \qquad (j\to\infty)
    \end{align}
    locally smoothly. Sending $T=t_j \to \infty$ in \eqref{eq:260306_2} with $\ell=1$ implies that for $F_\infty$ we have $\nabla A\equiv 0$. With \Cref{lem:Michael-Simon-2D} we find that $F_\infty$ satisfies
    \begin{align}\label{eq:02-03-26}
        \int_{\C}|A|^4 \gamma^2\d\mathrm{vol}\leq c\Vert \mathrm{D}\tilde\gamma\Vert_\infty^2\Big(\int_{\C}|A|^2\d\mathrm{vol}\Big)^2 \to 0
    \end{align}
    as $R\to\infty$, and thus $A\equiv 0$. Hence, $F_\infty$ is an embedded plane.

    For \eqref{eq:24-04-26_03}, suppose there exists $\ell\in \N_0$, $\delta>0$, $z_j\in \C$ and $t_j\to\infty$ such that 
    \begin{align}
        |\nabla^\ell A_{\tilde F_{t_j}}(z_j)|\geq \delta.
    \end{align}
    Taking this specific $z_j\in \C$ in \eqref{eq:260306_1}, the locally smooth convergence after translation implies that $\tilde F_{t_j}(\cdot + z_j)\to F_\infty$ smoothly on $B_1(0)$ where $F_\infty$ is a plane, a contradiction.
    
    Lastly, estimate \eqref{eq:05-02-26_2} follows from \eqref{eq:20-12-1} and \eqref{eq:260306_2}.    
\end{proof}

We are now ready to prove \Cref{thm:2d-conv}.

\begin{proof}[Proof of \Cref{thm:2d-conv}]
   First, by \Cref{lem:MS_GB} and \Cref{rem:V_by_monotonicity}, we find that $F_0$ is proper and has uniformly finite local volume. For $T<\infty$, let $\Lambda \vcentcolon = \sup_{M\times [0,T]}|H|<\infty$. Since $F$ is a purely normal mean curvature flow, we have $F_t^{-1}(B_R(0)) \subset F_0^{-1}(B_{R+\Lambda T}(0))$ for all $R>0, t\in [0,T]$. In particular, this implies that $F\colon M\times[0,\infty)\to\R^n$ is a proper mean curvature flow. The monotonicity of the volume form $\partial_t(\d\mathrm{vol})=-|H|^2\d\mathrm{vol}$ further implies
    \begin{align}
        \mu_t(B_r(0))\leq \mu_0(B_{r+\Lambda T}(0)).
    \end{align}
   We conclude
    \begin{align}
        \lim_{r\to\infty} \frac{\mu_t(B_r(0))}{\pi r^2} \leq \lim_{r\to\infty} \frac{\mu_0(B_{r+\Lambda T}(0))}{\pi r^2} = \lim_{r\to\infty}\frac{\mu_0(B_{r}(0))}{\pi r^2}=1.
    \end{align}
    Using that $\chi(\R^2)=1$, from \Cref{lem:2d-conv_3} and \Cref{lem:MS_GB} we obtain
    \begin{align}\label{eq:19-12-2}
        \int_M K\d\mathrm{vol} = 0 \quad \text{ for all }t\geq 0.
    \end{align}    
    Since $|H|^2 = |A|^2+2K$, \Cref{cor:main_calibration_detailed} and \eqref{eq:19-12-2} imply
    \begin{align}\label{eq:AL2L2}
        \int_0^\infty \int_M|A|^2\d\mathrm{vol}\d t \leq \E_\phi[F_0].
    \end{align}
    Hence, there exists $t_0\in (0,\infty)$ such that $  \int_M|A|^2\d\mathrm{vol}\vert_{t_0} <\varepsilon_0$.  
    Moreover, applying Ecker--Huisken's interior estimates \cite[Theorem 3.4 and Remark 3.6 (ii)]{MR1117150} (extended to the case of any codimension) as in \eqref{eq:05-02-26}, we find that $F_{t_0}$ has bounded geometry, i.e., \eqref{eq:bounded_geometry} holds.  
    Consequently, \Cref{thm:2D_small_energy} applies and \eqref{eq:02-03-26_2} follows.

    Suppose now that $\phi$ has constant coefficients. Note that any plane $F_\infty$ that may be obtained along a subsequence $t_j\to\infty$ and $\{z_j\}\subset \C$ as in \eqref{eq:260306_1} satisfies
    \begin{align}\label{eq:24-04-26_2}
        \E_\phi[F_\infty]\leq \liminf_{j\to\infty}\E_\phi[\tilde F_{t_j}]\leq \E_\phi[F_0]<\infty,
    \end{align}
    where $\tilde F_t$ is as in the proof of \Cref{thm:2D_small_energy}. Since $\tau_{F_\infty}$ is constant, we conclude that $F_\infty$ is $\phi$-calibrated.  
    To prove \eqref{eq:26-05_01}, suppose that there exist $\delta>0$, $t_j\to \infty$, and $z_j\in \C$ such that 
    \begin{align}\label{eq:24-04-26_02}
        1-\phi(\tau_{\tilde{F}_{t_j}})(z_j)\geq \delta.
    \end{align}   
    Taking this specific $z_j\in \C$ in \eqref{eq:260306_1}, the locally smooth convergence after translation implies that $\tau_{\tilde F_{t_j}(\cdot + z_j)}\to \tau_{F_\infty}$ uniformly in $B_1(0)$. Since any such $F_\infty$ must be $\phi$-calibrated, we obtain a contradiction to \eqref{eq:24-04-26_02}.  
    
    Finally, suppose $G(\phi)=\{ E\}$, where we assume, without loss of generality, that $E=e_1\wedge e_2$. Since $\phi$ has constant coefficients, \eqref{eq:26-05_01} implies $\lim_{t\to\infty}\Vert\tau-E\Vert_\infty=0$.    
    Let $\pi_E\colon \R^n \to \C$ be the canonical projection onto $E\cong \C$ and  $L_t \vcentcolon=\pi_E \D \tilde{F}_t (0)$. Writing $L_t z = a_t z+ b_t \bar z$, $z\in\C$, for some $a_t,b_t\in\C$, we infer from \eqref{eq:conf_factor_MS} and $\tau(0,t)\to E$ that $|a_t|\geq \delta>0$ for all $t$ large enough and further that $b_t\to 0$; hence, for $t$ sufficiently large, $L_t$ is an orientation-preserving $\R$-linear isomorphism.
    Therefore, $L_t^{-1} = a_t^{-1}\mathrm{Id}_{\C}+o(1)$ as $t\to\infty$. Replacing $\tilde F_t(z)$ by $\tilde F_t(a_t^{-1} z)$ in \eqref{eq:260306_1}, we conclude that the limit $F_\infty$ satisfies $F_\infty(0)=0$ and $\pi_E\D F_\infty(0) = \mathrm{Id}_{\C}$. Together with the fact that $F_\infty$ is a conformal parametrization of $E$, we obtain $F_\infty(z) = (z, 0)\in \C\times\R^{n-2} \cong \R^n$. 
    By the uniqueness of $F_\infty$ we can upgrade the subconvergence to full convergence. 
\end{proof}

\subsection{A gap theorem for two-dimensional self-shrinkers}

Finally we discuss applications to self-shrinkers.
In view of scaling, our energy identity does not yield as strong rigidity properties as in the nonshrinking case.
However, we still obtain some nontrivial new conditions.

An immersion $F\colon M^m\to\R^n$ is called a \emph{self-shrinker} if $F_t\vcentcolon =\sqrt{1-2t}\,F$ defines a mean curvature flow for $0\leq t <\frac12$; equivalently,
\begin{equation}\label{eq:self-shrinker}
    F^\perp = -H.
\end{equation}
Self-shrinkers arise as blow-up profiles of type I singularities of the mean curvature flow \cite{MR1030675}.

\begin{proposition}\label{prop:self-shrinker_Pohozaev}
    Let $F\colon M^m\to\R^n$ be a proper oriented self-shrinker with $\E_\phi[F]<\infty$ for a dilation-invariant coflat calibration $\phi$ of degree $m$ on $\R^n$.
    Then
    \begin{equation}
        \E_\phi[F] = \frac{1}{m}\int_M|H|^2\d\mathrm{vol}.
    \end{equation}
\end{proposition}

\begin{proof}
    As in the proof of \Cref{thm:self_expander}, \Cref{thm:main_calibration_detailed} also applies to the corresponding flow $\{F_t\}$.
    Substituting the scaling properties $\E_\phi[F_t]=(1-2t)^{\frac{m}{2}}\E_\phi[F]$, cf.\ \eqref{eq:scaling}, and $|H_{F_t}|^2\d\mathrm{vol}_{F_t}=(1-2t)^{\frac{m-2}{2}}|H_F|^2\d\mathrm{vol}_{F}$ into \eqref{eq:energy_identity}, we deduce that, for $0\leq t<\frac{1}{2}$,
    \begin{align}
        \E_\phi[F] &= (1-2t)^{\frac{m}{2}}\E_\phi[F] + \int_0^t\int_M(1-2t')^{\frac{m-2}{2}}|H_F|^2\d\mathrm{vol}\d t' \\&= (1-2t)^{\frac{m}{2}}\E_\phi[F] + \frac{1-(1-2t)^{\frac{m}{2}}}{m}\int_M|H_F|^2\d\mathrm{vol}.
    \end{align}
    Rearranging and dividing by $1-(1-2t)^{\frac{m}{2}}$ yield the desired identity.
\end{proof}

This may be regarded as a Pohozaev-type identity, i.e., an integral constraint due to a scale variation arising from the partial differential equation.

As an application we obtain a gap theorem for two-dimensional, possibly noncompact self-shrinkers.

\begin{theorem}\label{thm:shrinker_2D}
    There is a universal constant $\varepsilon_0>0$ with the following property:
    Let $F\colon M^2\to\R^n$ be a connected, proper oriented self-shrinker with $\sup_M|A|<\infty$ and $\int_M K \d\mathrm{vol} = 0$ (understood as an improper integral if $M$ is noncompact).
    If
    \[
    \E_\phi[F]<\varepsilon_0
    \]
    for a dilation-invariant coflat calibration $\phi$ of degree $2$ on $\R^n$, then $F$ is a plane.
\end{theorem}

\begin{proof}
    It is easy to deduce from \Cref{prop:self-shrinker_Pohozaev} and the classical Willmore inequality that there are no closed self-shrinkers with small $\E_\phi$ (or equivalently small volume).
    Hence we only discuss the case that $M$ is noncompact.
    
    By \Cref{prop:self-shrinker_Pohozaev} and $\int_M K \d\mathrm{vol} = 0$, we have
    \begin{equation}\label{eq:calibration_energy_total_curvature}
        \E_\phi[F] = \frac{1}{2}\int_M|H|^2\d\mathrm{vol} = \frac{1}{2}\int_M|A|^2\d\mathrm{vol}.
    \end{equation}
    In particular, we may assume that $F$ has finite total curvature.
    By \Cref{rem:V_by_monotonicity}, $F$ also has uniformly finite local volume.
    In addition, by \Cref{lem:MS_GB}, we have $M\cong \R^2$.
    
    We now write $\hat{F}_t$ for the reparametrization of $F_t=\sqrt{1-2t}\,F$ with purely normal velocity. Hence $\hat F(\cdot,t)\vcentcolon =\hat F_t$, $t\in [0,\frac12)$, defines a purely normal mean curvature flow.
    As in the previous section, using Ecker--Huisken's interior estimates, we find that $\hat{F}_t$ has bounded geometry \eqref{eq:bounded_geometry} for $t>0$.
    By self-similarity, $\hat{F}_0$ ($=F$) also satisfies \eqref{eq:bounded_geometry}.
    Now, if we take the constant $\varepsilon_0$ sufficiently small, then by \eqref{eq:calibration_energy_total_curvature} the total curvature of $\hat{F}_0$ is also small so that \Cref{thm:2D_small_energy} applies.
    In particular, the flow $\hat{F}$ has bounded second fundamental form on $M\times[0,\frac12)$ (and exists immortally).   
    By the self-shrinking property, this occurs only when $\hat{F}_0$ satisfies $|A|\equiv0$.
    Hence $F$ is a plane.
\end{proof}

Our gap theorem is different in nature from existing gap theorems for self-shrinkers involving pointwise curvature bounds.
Le--Sesum \cite{MR2880211} proved, in codimension one, that a complete embedded self-shrinker with polynomial volume growth and $|A|^2<1$ is a hyperplane, using an identity of Colding--Minicozzi \cite{MR2993752}; Cao--Li \cite{MR3018176} extended this result to arbitrary codimension.
By contrast, our assumption is of integral type in terms of the calibration energy.

\appendix
\section{Finite calibration energy}\label{sec:finite_energy}

In this appendix, we discuss in more detail the class of smooth immersions with finite smooth-calibration energy, clarifying the geometric meaning and scope of this condition in several important settings. 
We first address the general situation, and then turn to constant-coefficient calibrations and some specific cases. Throughout this section, all immersions are understood to be oriented.

\subsection{General calibrations}

We first observe that under some regularity assumptions, finite calibration energy implies pointwise convergence $\phi(\tau)\to1$ at infinity.

\begin{proposition}\label{prop:asymptotic_calibration_bounded_A}
    Let $U\subset\R^n$ be open, let $\phi\in\Omega^m(U)$ be a calibration, and let $F\colon M^m\to U$ be a complete immersion.
    Assume that
    \[
    \sup_M |A|<\infty,
    \qquad
    \E_\phi[F]<\infty,
    \]
    and that $\phi$ has Lipschitz coefficients on $F(M)$.
    Then for every $\varepsilon>0$ there exists a compact set $K_\varepsilon\subset M$ such that
    \[
    1-\phi(\tau) \le \varepsilon
    \qquad\text{on }M\setminus K_\varepsilon.
    \]
\end{proposition}

\begin{proof}
    A computation in normal coordinates gives $|\nabla \tau|\le C(m)|A|$ and hence, by the product rule,
    \[
    |\nabla(\phi(\tau))|\le C(m,n)(\mathrm{Lip}(\phi)+|A|).
    \]
    Therefore $1-\phi(\tau)$ is globally Lipschitz on $M$.

    Suppose that the conclusion fails.
    Then there exist $\varepsilon>0$ and a sequence $\{p_j\}\subset M$ such that $\dist_M(p_j,p_{j'})\ge 1$ for $j\neq j'$ and $1-\phi(\tau)\ge \varepsilon$ at $p_j$ for all $j$.
    By the global Lipschitz bound, there exists $r=r(\varepsilon)>0$ such that, for all $j$,
    \[
    1-\phi(\tau)\ge \frac{\varepsilon}{2}
    \qquad\text{on }B_r^M(p_j),
    \]
    where $B_r^M(p)$ denotes the geodesic ball in $(M,g)$ with radius $r$ centered at $p$.
    Taking smaller $r$, we may assume that the balls $B_r^M(p_j)$ are pairwise disjoint.

    Since $\sup_M|A|<\infty$, the Gauss equation yields a uniform bound for the sectional curvature of $M$, and we also get a lower bound on the injectivity radius (using, for instance, \cite[Theorem 2.1]{MR2379801}). G\"unther's volume comparison theorem %[Lee, Introduction to Riemannian Manifolds, 2nd ed. Theorem 11.14]
    then yields that for $r$ sufficiently small, $\inf_{p\in M}\operatorname{vol}(B_r^M(p))>0$, and therefore
    \[
    \E_\phi[F]
    \ge \sum_{j=1}^\infty \int_{B_r^M(p_j)} (1-\phi(\tau)) \d\mathrm{vol}
    \ge \frac{\varepsilon}{2}\sum_{j=1}^\infty \operatorname{vol}(B_r^M(p_j))
    =\infty,
    \]
    contradicting $\E_\phi[F]<\infty$.
\end{proof}

\subsection{Constant-coefficient calibrations}\label{subsec:constant_coeff}

Now we turn to the constant-coefficient case.
By Hodge duality, constant-coefficient calibrations in $\Omega^{m}(\R^n)$ correspond to those in $\Omega^{n-m}(\R^n)$.
If $m=1$ or $m=n-1$, every constant-coefficient calibration is unidirectional; in particular, $G(\phi)$ is a singleton.
In the case $m=2$ and dually, $m=n-2$, the possibilities of $G(\phi)$ are already classified: for $m=2$, $G(\phi)$ consists of, up to rotation, all complex lines in $\C^{r}\subset\R^n$ for some $1\leq r \leq \lfloor\frac{n}{2}\rfloor$, see Harvey--Lawson \cite{MR666108}.
More generally, the possible sets $G(\phi)$ for constant-coefficient calibrations $\phi$ on $\R^n$ are classified for $n\le 7$ in \cite{MR666108,MR726326,MR818350}.
The case $n=8$ is only partial, though already rather well understood \cite{MR936802}.
In higher dimensions $n\geq9$, the situation is much richer and far from classification; see, for example, the spinorial constructions in \cite{MR1208563} and a recent exotic example \cite{zust2025calibration}.

We observe that in the constant-coefficient case the convergence at infinity can be stated in a more geometrical way: the tangent direction approaches $G(\phi)$.

\begin{corollary}\label{cor:constant_coeff_asymptotic_planes}
    Let $\phi\in\Omega^m(\R^n)$ be a constant-coefficient calibration, and let $F\colon M^m\to\R^n$ be a complete immersion.
    If
    \[
    \sup_M |A|<\infty,
    \qquad
    \E_\phi[F]<\infty,
    \]
    then for every $\varepsilon>0$ there exists a compact set $K_\varepsilon\subset M$ such that
    \[
    \dist(\tau(p),G(\phi))<\varepsilon
    \qquad\text{for all }p\in M\setminus K_\varepsilon.
    \]
    Recall that $\dist$ is the extrinsic distance in $\Lambda^m\R^n \cong \R^{\binom nm}$.
\end{corollary}

\begin{proof}
    By \Cref{prop:asymptotic_calibration_bounded_A}, we have $1-\phi(\tau(p))\to 0$ as $p\to\infty$ in $M$.
    If the conclusion were false, then there would exist $\varepsilon>0$ and points $p_j\to\infty$ in $M$ such that $\dist(\tau(p_j),G(\phi))\geq\varepsilon$ for all $j$.
    Since $\mathbf{Gr}_m^+(\R^n)$ is compact, after passing to a subsequence we may assume that $\tau(p_j)\to\xi\in\mathbf{Gr}_m^+(\R^n)$.
    Then $\dist(\xi,G(\phi))\ge \varepsilon$, while continuity of $\phi$ and $1-\phi(\tau(p_j))\to 0$ imply that $\phi(\xi)=1$, i.e., $\xi\in G(\phi)$, a contradiction.
\end{proof}

\begin{remark}\label{rem:variable_coeff_counterexample}
    The assertion of \Cref{cor:constant_coeff_asymptotic_planes} does not hold in the variable-coefficient case.
    Even if $G(\phi_x)=\{E\}$ for every $x$, finite calibration energy need not force $\tau\to E$ at infinity.
    Indeed, let $x=(x_1,x_2,x_3,x_4)\in\R^4$, $E\vcentcolon =e_1\wedge e_2$, $E'\vcentcolon =e_3\wedge e_4$,
    and define
    \[
    \phi\vcentcolon =\d x_1\wedge \d x_2 + a(x_3,x_4)\d x_3\wedge \d x_4,
    \quad
    a(x')\vcentcolon =1-e^{-|x'|^2},
    \quad
    x'=(x_3,x_4).
    \]
    Then $\d\phi=0$ and $\phi$ has comass $1$, and also $G(\phi_x)=\{E\}$ for all $x\in\R^4$, but
    \[
    \E_\phi[F_{E'}]
    =\int_{E'} (1-a(x'))\d\mu
    =\int_{\R^2} e^{-|x'|^2}\d x'<\infty,
    \]
    where $F_{E'}$ denotes a parametrization of the plane $E'$.
    Thus pointwise uniqueness of the calibrated plane does not by itself imply asymptotic convergence to that plane.
\end{remark}

So far we derived necessary conditions for finite calibration energy.
The following is a simple sufficient condition.

\begin{proposition}\label{prop:constant_coeff_upper_bound}
    Let $\phi\in\Omega^m(\R^n)$ be a constant-coefficient calibration.
    Then there exists $C_\phi>0$ such that, for all $\xi\in \mathbf{Gr}_m^+(\R^n)$,
    \[
    0\le 1-\phi(\xi)\le C_\phi\dist(\xi,G(\phi))^2.
    \]
    In particular, if an immersion $F\colon M^m\to\R^n$ satisfies $\dist(\tau,G(\phi))\in L^2(M)$, then $\E_\phi[F]<\infty$.
\end{proposition}

\begin{proof}
    Set $f_\phi(\xi)\vcentcolon =1-\phi(\xi)$ for $\xi\in \mathbf{Gr}_m^+(\R^n)$.
    Since $\phi$ has constant coefficients, $f_\phi$ is a smooth function on the compact manifold $\mathbf{Gr}_m^+(\R^n)$.
    Moreover $f_\phi\ge 0$ and $G(\phi)=f_\phi^{-1}(0)$.
    Hence every point of $G(\phi)$ is a minimum point of $f_\phi$, so $\d f_\phi=0$ on $G(\phi)$.
    Since $\mathbf{Gr}_m^+(\R^n)$ is compact, by Taylor's theorem there is $C_\phi>0$ such that $f_\phi(\xi)\leq C_\phi\dist_\mathrm{int}(\xi,G(\phi))^2$ holds for any $\xi\in\mathbf{Gr}_m^+(\R^n)$, where $\dist_\mathrm{int}$ denotes the intrinsic (geodesic) distance in $\mathbf{Gr}_m^+(\R^n)$.
    By compactness, $\dist_\mathrm{int}$ is equivalent to the extrinsic distance on $\mathbf{Gr}_m^+(\R^n)$, and the proof is complete by redefining $C_\phi$.
\end{proof}

\begin{remark}\label{rem:nonquadratic}
    The condition $\E_\phi[F]<\infty$ is equivalent to $\dist(\tau,G(\phi))\in L^2(M)$ if the nondegeneracy \eqref{eq:non_degenerate_quadratic} holds.
    % \begin{equation}\label{eq:non_degenerate_quadratic}
    %     1-\phi(\xi)\ge c_\phi\dist(\xi,G(\phi))^2 \qquad \text{for all }\xi \in \mathbf{Gr}_m^+(\R^n)
    % \end{equation}
    % for some $c_\phi>0$. In this case, we say that $\phi$ is \emph{nondegenerate}. 
    However, this may not hold in general.
    An explicit example is given by
    \[
    \phi\vcentcolon =e^{123}+\frac12 e^{156}+\frac12 e^{426}+\frac12 e^{453} \in \Omega^3(\R^6),
    \]
    where $e^{ijk}$ denotes the dual of $e_i\wedge e_j\wedge e_k$. This is a calibration, cf.\ \cite{MR726326}, with $G(\phi)=\{E\}$ for $E\vcentcolon =e_1\wedge e_2\wedge e_3$. Explicit computations along
    \[
    \xi_t\vcentcolon =\frac{(e_1+t e_4)\wedge (e_2+t e_5)\wedge (e_3+t e_6)}{(1+t^2)^{3/2}} \in \mathbf{Gr}_3^+(\R^6)
    \]
    show that $1-\phi(\xi_t) = \frac{3}{8}t^4 + O(t^6)$, whereas $|\xi_t-E|^2 =3t^2 - \frac{15}{4}t^4 +O(t^6)$ as $t\to 0$.
\end{remark}

There are some simple cases where the lower bound \eqref{eq:non_degenerate_quadratic} also holds.
The simplest example is the unidirectional case (see the proof of \Cref{prop:finite_unidirectional_calibration} below).
In particular, if $m=1$ (or $m=n-1$), then every constant-coefficient calibration is unidirectional and hence finite calibration energy is equivalent to $\dist(\tau, G(\phi))\in L^2(M)$.
%More nontrivially, in the case $m=2$ (or $m=n-2$), the equivalence also holds. Counterexamples emerge for $m=3$ and $n=6$.

\subsection{Unidirectional case}

Here we focus on the unidirectional case $\phi_E=\langle\cdot,E\rangle$ for a fixed plane $E\in\mathbf{Gr}_m^+(\R^n)$.
In this case we have a simple characterization of finite calibration energy.

\begin{proposition}\label{prop:finite_unidirectional_calibration}
    Let $F\colon M^m\to\R^n$ be an immersion.
    Let $\phi_E\in\Omega^m(\R^n)$ be a unidirectional calibration in the direction $E\in\mathbf{Gr}_m^+(\R^n)$.
    Then $\E_{\phi_E}[F]<\infty$ holds if and only if $|\tau-E|\in L^2(M)$, or equivalently, $\dist(\tau,G(\phi_E)) \in L^2(M)$.
\end{proposition}

\begin{proof}
    This follows since $G(\phi_E)=\{E\}$ and $1-\phi_E(\xi)=1-\langle \xi,E\rangle
    =\frac12|\xi-E|^2
    =\frac12\dist(\xi,G(\phi_E))^2$.
\end{proof}

In the case of entire graphs with bounded gradient, we further have the following characterization.

\begin{proposition}\label{prop:graphical_unidirectional}
    Let $F\colon \R^m\to\R^n$ be an immersion of the form $F=(\cdot,f)$ for $f\colon\R^m\to\R^{n-m}$ with $\|\D f\|_\infty<\infty$.
    Let $E=e_1\wedge\dots\wedge e_m$.
    Then $\E_{\phi_E}[F]<\infty$ if and only if $|\D f|\in L^2(\R^m)$.
\end{proposition}

\begin{proof}
    % We first note that, since $\langle\tau,E\rangle=\frac{1}{\sqrt{\det(g)}}$,
    % \begin{align}
    %     \E_{\phi_E}[F]
    %     = \int_{\R^m}\bigl(\sqrt{\det(g)}-1\bigr)\d x
    %     = \int_{\R^m}\frac{\det(g)-1}{\sqrt{\det(g)}+1}\d x. \label{eq:0407-01}
    % \end{align}
    % Let $L\vcentcolon =\|\D f\|_\infty$ and $S\vcentcolon =(\D f)^T \D f$. 
    % Then $g=I+S$, and $S$ is symmetric and nonnegative.
    % Let $\lambda_1,\dots,\lambda_m\ge0$ denote the eigenvalues of $S$ at some fixed $x\in\R^m$.
    % Then
    % \[
    %     \sum_{i=1}^m \lambda_i = \tr S = |\D f|^2,
    %     \qquad
    %     0\le \lambda_i \le L^2.
    % \]
    % Moreover, we also have
    % \begin{align}
    %     1\leq \det(g)=\det(I+S)=\prod_{i=1}^m(1+\lambda_i) \leq (1+L^2)^m.
    % \end{align}
    % Combining the above formulae and estimates with \eqref{eq:0407-01}, it is now sufficient to show that there exist constants $c,C>0$, depending only on $m$ and $L$, such that
    % \begin{align}
    %     c\sum_{i=1}^m \lambda_i \le \prod_{i=1}^m(1+\lambda_i)-1 \le C\sum_{i=1}^m \lambda_i.
    % \end{align}
    % This follows since the telescoping identity yields
    % \[
    %     \prod_{i=1}^m(1+\lambda_i)-1
    %     = \sum_{i=1}^m\biggl(\prod_{j=1}^i(1+\lambda_j)-\prod_{j=1}^{i-1}(1+\lambda_j)\biggr)
    %     = \sum_{i=1}^m \lambda_i\prod_{j=1}^{i-1}(1+\lambda_j),
    % \]
    % and since $1\leq \prod_{j=1}^{i-1}(1+\lambda_j) \leq (1+L^2)^{m-1}$.
    Since $\langle\tau,E\rangle=\frac{1}{\sqrt{\det(g)}}$ and $g=I+(\D f)^T\D f$,
    we have
    \[
    \E_{\phi_E}[F] = \int_{\R^m}\bigl(\sqrt{\det(g)}-1\bigr)\d x = \int_{\R^m}\Bigl(\sqrt{\det(I+(\D f)^T\D f)}-1\Bigr)\d x.
    \]
    The assertion follows from the fact that the integrand and $|\D f|^2$ are uniformly comparable provided $\|\D f\|_\infty<\infty$.
\end{proof}

We now exhibit several examples of finite direction energy $\E_{\phi_E}<\infty$, as well as of uniformly finite local volume.

\begin{example}[Topological complexity]\label{ex:topology}
    For any oriented surface $M^2$ of finite genus and finitely many ends, there exists a proper immersion $F\colon M\to\R^3$ with uniformly finite local volume such that $\mathcal{E}_{\phi_E}[F]+\sup_M|A|<\infty$.

    If we do not require bounded second fundamental form, we may even construct examples of infinitely many distinct planar ends that are discretely distributed and connected by a suitable family of tubes getting thinner and thinner.
\end{example}

\begin{example}[Single end with arbitrary multiplicity]\label{ex:single_end_multiplicity}
    Even if the topology is simple, the ends may have nontrivial overlaps.
    For $p\in\N$, let $F\colon \R^2\to\R^4$ be defined by $z \mapsto (z^p,\eta(|z|)z)$, where $z\in\C\cong\R^2$ and $\eta\in C^\infty_c([0,2))$ with $\eta\equiv1$ on $[0,1]$, cf.\ \cite[p.~248]{MR1366547}.
    Then the single end has arbitrary multiplicity $p\geq1$, while both $|A|$ and $|\tau-E|$ for $E\vcentcolon =e_1\wedge e_2$ have compact support; in particular, $\sup_{\R^2}|A|<\infty$, $F$ has uniformly finite local volume, and $\E_{\phi_E}[F]<\infty$.
    In addition, by using the compact-support property we can compute the total Gaussian curvature $\int_M K \d\mathrm{vol} = 2\pi(1-p)$.
    Moreover, if we define $F_R$ by replacing $\eta$ with $\eta_R\vcentcolon =\eta(\cdot/R)$ for $R>0$, then the total Gaussian curvature remains unchanged while the total curvature satisfies $\int_M|A|^2 \d\mathrm{vol} \to 4\pi(p-1)$ as $R\to\infty$ (but $\E_{\phi_E}[F_R]\to\infty$).
\end{example}

\begin{example}[Unbounded graphical ends]\label{ex:graphical}
    Even if the ends of $F\colon M^m\to\R^n$ with uniformly finite local volume and $\E_{\phi_E}[F]+\sup_M|A|<\infty$ are graphical, they may not be asymptotic to planes.
    As already discussed in \cite{MR5008349}, if $m=1$, admissible ends include diverging and oscillating graphs such as $f(x)=|x|^\alpha$ and $f(x)=|x|^\alpha\sin\log|x|$ with $\alpha\in(0,\frac{1}{2})$.
    In dimension $m=2$, radially diverging ends such as $f(x)=\log\log|x|$ and their oscillating variants are still admissible (cf.\ \Cref{prop:graphical_unidirectional}).
\end{example}

\subsection{K\"ahler case}\label{subsec:kaehler}

We now discuss some classical non-unidirectional examples.
The first model case of a multidirectional example concerns the K\"ahler calibration; see \cite{MR666108} for details.

Let $m=2k$ and $n=2k+2\ell$ with $k,\ell\geq1$, and identify $\R^n=\C^{k+\ell}$. 
The \emph{K\"ahler calibration} $\phi_{K}\in \Omega^{2k}(\R^n)$ is defined by
\[
\phi_{K}\vcentcolon =\frac{1}{k!}\omega^k,
\]
where $\omega\vcentcolon =\sum_{j=1}^{k+\ell}\d x_j\wedge \d y_j=\frac{\sqrt{-1}}{2}\sum_{j=1}^{k+\ell}\d z_j\wedge \d\bar{z}_j$, $z_j=x_j+\sqrt{-1}\,y_j$, denotes the standard K\"ahler $2$-form on $\C^{k+\ell}$.
Then the set $G(\phi_K)$ consists of canonically oriented complex $k$-planes, i.e.,
%(In real terms, $V\in \mathbf{Gr}_{2k}^+(\R^{2(k+\ell)})$ is a complex $k$-plane if $\dim{V}=2k$ and $J(V)=V$, where $J(x,y)=(-y,x)$ for $x,y\in\R^{k+\ell}$.)
\[
G(\phi_K)=\{\text{complex $k$-planes in $\C^{k+\ell}$}\}.
%\cong U(k+\ell)/(U(k)\times U(\ell)),
\] 
% where $U(k)$ stands for the unitary group of degree $k$,
% so the real dimension is $\dim G(\phi_K)=(k+\ell)^2-k^2-\ell^2=2k\ell$.
%"Husemoller, Fibre Bundles (3rd ed), Chapter 8, Section 2, Theorem 2.2" for a reference
Hence, finite K\"ahler calibration energy admits multidirectional ends.
Indeed, $\phi_K$-calibrated submanifolds are exactly complex $k$-dimensional submanifolds, demonstrating the admissibility of a broad class of curved ends.
In particular, if an immersion $F\colon \C^k\to\C^{k+\ell}$ is given as an entire graph, $F(z)=(z,f(z))$ for some $f\colon\C^k\to\C^\ell$, then $F$ is K\"ahler calibrated if and only if $f$ is holomorphic.

Now we characterize the class of immersions with finite K\"ahler calibration energy under some assumptions.
From now on we will write the ambient complex dimension $N\vcentcolon=k+\ell$ to shorten notation.

We first focus on the case of dimension $k=1$.
In this case, every oriented immersion $F\colon M^2\to\C^{N}$ (with $N\geq2$) induces the canonical complex structure $j\in\Gamma(\operatorname{Aut}(TM))$ such that $g(ju,v)=\d\mathrm{vol}(u,v)$, which is the $\pi/2$-rotation on each tangent plane, so that $(M,j)$ is a Riemann surface.
Let $J$ be the standard complex structure on $\C^N$; in particular, $\omega(u,v)=\langle Ju,v\rangle$.
Now we define the nonlinear Cauchy--Riemann operator
\[
\bar{\partial} F \vcentcolon =\frac12(\d F+J\circ \d F\circ j).
\]
Then $\bar{\partial} F = 0$ holds on $M$ if and only if $F\colon (M,j)\to(\C^N,J)$ is pseudo-holomorphic (i.e., $\d F\circ j=J\circ \d F$).

Using the canonical Riemann surface structure, we can directly characterize the K\"ahler calibration energy $\E_\omega$ in terms of the following \emph{$\bar{\partial}$-energy}
\begin{equation}\label{eq:dbar_energy}
    \bar{\mathcal{D}}[F] \vcentcolon= \int_M |\bar{\partial} F|^2 \d\mathrm{vol}.
\end{equation}

\begin{proposition}\label{prop:kahler_L2_1D}
    Let $F\colon M^2\to\C^{N}$ be an immersion.
    Then the pointwise identity $|\bar{\partial} F|^2=1-\omega(\tau)$ holds on $M$.
    In particular, 
    \[
    \E_\omega[F] = \bar{\mathcal{D}}[F].
    \]
\end{proposition}

The pointwise identity is classical; see, for example, \cite[Lemma 2.2.1]{MR2954391}.
We include a short proof in our terminology for the reader's convenience.

\begin{proof}[Proof of \Cref{prop:kahler_L2_1D}]
    Fix $p\in M$.
    Let $(e,je)$ be a positively oriented orthonormal basis on $T_pM$.
    Since $\d F$ is isometric, $\d F(e)$ and $\d F(je)$ are also orthonormal in $\C^{N}$, and
    \[
    \tau=\d F(e)\wedge \d F(je), \qquad \omega(\tau)=\omega(\d F(e),\d F(je))=\langle J\d F(e),\d F(je)\rangle.
    \]
    We then compute
    \begin{align}
        |\bar{\partial} F|^2 &= |\bar{\partial} F(e)|^2+|\bar{\partial} F(je)|^2 =\frac14|\!\d F(e) + J\d F(je)|^2 + \frac14|\!\d F(je) - J\d F(e)|^2\\
        &= 1 - \langle J\d F(e),\d F(je)\rangle = 1- \omega(\tau),
    \end{align}
    which also directly yields the energy representation.
\end{proof}

\begin{remark}\label{rem:interpretation_kaehler}
    Consequently, our energy identity \eqref{eq:energy_identity} implies that the $\bar{\partial}$-energy in \eqref{eq:dbar_energy} is a strict Lyapunov functional along any proper oriented mean curvature flow of surfaces in $\C^N$.
    Note also that, if $\theta_\omega:M\to[0,\pi]$ denotes the K\"ahler angle (introduced in \cite{MR692106}), i.e., $\omega(\tau)=\cos\theta_\omega$, then we may alternatively represent
    \[
    \E_{\omega}[F] = \int_M (1-\cos\theta_\omega) \d\mathrm{vol}.
    \]
\end{remark}

Now we turn to the case of general dimension $k\geq1$.
In this case, in order to induce a suitable almost complex structure to the domain manifold, we work in the symplectic framework.

Let $F\colon M^{2k}\to(\C^N,\omega)$ (with $N\geq k+1$) be an \emph{oriented symplectic immersion}, meaning that the pullback $F^*\omega$ is symplectic and induces the given orientation on $M$.
Since $F^*\omega$ is clearly closed, it is symplectic if and only if it is nondegenerate, i.e., $(F^*\omega)^k\neq0$ pointwise on $M$.
The orientation condition then means that $(F^*\omega)^k>0$ on $M$.
Since $\frac{1}{k!}(F^*\omega)^k=F^*\phi_K=\phi_K(\tau)\d\mathrm{vol}$, we find that $F$ is an oriented symplectic immersion if and only if
\[
\phi_K(\tau)>0 \qquad \text{on $M$}.
\]

Let $J_\omega\in\Gamma(\operatorname{Aut}(TM))$ be the almost complex structure on $M$ associated with the pair $(F^*\omega,g)$ in the sense of \cite[Proposition 12.6]{MR1853077}, where we recall that $g=F^*\langle\cdot,\cdot\rangle$.
More precisely, if $S\in\Gamma(\operatorname{Aut}(TM))$ is defined by $g(Su,v)=F^*\omega(u,v)$, then $J_\omega\vcentcolon =(-S^2)^{-1/2}S$, cf.\ \cite[Proposition 12.3]{MR1853077}.
% Note that in general $g_\omega(u,v)\vcentcolon =F^*\omega(u,J_\omega v)\neq g(u,v)$.
We then similarly define the nonlinear Cauchy--Riemann operator
\[
\bar{\partial}F\vcentcolon =\frac12(\d F+J\circ \d F\circ J_\omega).
\]
Here we use the same notation $\bar{\partial}$ because if $k=1$ then $J_\omega=j$.

One can show that $F$ is calibrated by $\phi_K$ if and only if $\bar{\partial}F=0$ holds on $M$, or equivalently, $F \colon (M,J_\omega)\to(\C^N,J)$ is pseudo-holomorphic.
Here we indicate that a quantitative version of this fact provides a simple characterization of finite K\"ahler calibration energy, again in terms of the $\bar{\partial}$-energy $\bar{\mathcal{D}}$ (defined as in \eqref{eq:dbar_energy}).

\begin{proposition}\label{prop:kahler_L2}
    Let $F\colon M^{2k}\to\C^N$ be an oriented symplectic immersion.
    Then the pointwise estimates
    \[
    1-\phi_K(\tau)\le |\bar{\partial}F|^2\le k(1-\phi_K(\tau))
    \]
    hold on $M$.
    In particular, $\E_{\phi_K}[F]<\infty$ holds if and only if
    \[
    \bar{\mathcal{D}}[F]=\int_M|\bar{\partial}F|^2\d\mathrm{vol} < \infty.
    \]
\end{proposition}

\begin{proof}
Fix $p\in M$.
Since $S\in \operatorname{Aut}(T_pM)$ is real skew-adjoint and invertible, its normal form is a block diagonal matrix of the form
\[
S = \bigoplus_{\alpha=1}^k
\begin{pmatrix}
0 & -\lambda_\alpha\\
\lambda_\alpha & 0
\end{pmatrix}
,\qquad \lambda_\alpha>0.
\]
In other words, there is an orthonormal basis $(e_1,f_1,\dots,e_k,f_k)$ of $T_pM$ such that $Se_\alpha=\lambda_\alpha f_\alpha$ and $Sf_\alpha=-\lambda_\alpha e_\alpha$.
Therefore, $F^*\omega=\sum_{\alpha=1}^k \lambda_\alpha e_\alpha^\flat\wedge f_\alpha^\flat$, where $\flat$ denotes the musical isomorphism of tangent vectors on $M$.
Also, since $J_\omega=(-S^2)^{-1/2}S$, we have $f_\alpha=J_\omega e_\alpha$.
Moreover, we estimate $\lambda_\alpha
=F^*\omega(e_\alpha,f_\alpha)\leq1$, using that $F$ is an isometry and $\omega$ has comass $1$.
Consequently, we obtain the representation
\[
F^*\omega=\sum_{\alpha=1}^k \lambda_\alpha e_\alpha^\flat \wedge (J_\omega e_\alpha)^\flat, \qquad 0 < \lambda_\alpha \leq 1.
\]

Now, since $\phi_K(\tau)\d\mathrm{vol}=\frac{1}{k!}(F^*\omega)^k$, we can represent $\phi_K(\tau)$ by
\begin{equation}\label{eq:0412_01}
    \phi_K(\tau)=\prod_{\alpha=1}^k \lambda_\alpha.
\end{equation}

On the other hand, since $\d F$ is isometric, the vectors $\d F(e_\alpha)$ and $\d F(J_\omega e_\alpha)$ are orthonormal. Moreover, we have $\langle J \d F(e_\alpha),\d F(J_\omega e_\alpha)\rangle = \lambda_\alpha$, and thus
\begin{equation}
    |\!\d F(e_\alpha)+J\d F(J_\omega e_\alpha)|^2+|\!\d F(J_\omega e_\alpha)-J\d F(e_\alpha)|^2=4(1-\lambda_\alpha).
\end{equation}
Summing over $\alpha$, we obtain
\begin{equation}\label{eq:0412_02}
    |\bar\partial F|^2=\sum_{\alpha=1}^k(1-\lambda_\alpha).
\end{equation}

Since $0<\lambda_\alpha\le 1$, we have the elementary inequalities (that may be readily proven by induction on $k$)
\[
1-\prod_{\alpha=1}^k\lambda_\alpha
\le
\sum_{\alpha=1}^k(1-\lambda_\alpha)
\le
k\Bigl(1-\prod_{\alpha=1}^k\lambda_\alpha\Bigr),
\]
which combined with \eqref{eq:0412_01} and \eqref{eq:0412_02} yield the desired estimates.
\end{proof}

\subsection{Special Lagrangian case}

The second model case is the special Lagrangian calibration, again classical in \cite{MR666108}.

Let $n=2m$ and identify $\R^n=\C^m$.
Let $\Omega\vcentcolon =\d z_1\wedge\cdots\wedge \d z_m$.
Recall that $F\colon M^m \to \C^m$ is called a \emph{special Lagrangian of phase $e^{\sqrt{-1}\theta}$} if it is \emph{Lagrangian}, i.e., the K\"ahler $2$-form $\omega$ satisfies
$F^*\omega=0$, and has constant \emph{phase} $e^{\sqrt{-1}\theta}$, i.e., $F^*\Omega=e^{\sqrt{-1}\theta}\d\mathrm{vol}$ for $\theta\in\R$. 
By Harvey--Lawson \cite{MR666108}, this is equivalent to $F$ being calibrated by the \emph{special Lagrangian calibration}
\[
\phi_\theta\vcentcolon =\Re\bigl(e^{-\sqrt{-1}\theta}\Omega\bigr).
\]
We denote the oriented Lagrangian Grassmannian by
\[
\mathrm{Lag}^+(\C^m)
\vcentcolon=
\{\xi\in \mathbf{Gr}_m^+(\R^{2m})
 \mid \omega|_{P_\xi}=0\},
 %\cong U(m)/SO(m),
\]
where $P_\xi\subset\R^{2m}=\C^m$ denotes the oriented $m$-plane represented by $\xi$.
Then
\begin{align}
    G(\phi_\theta) &=\{\xi\in \mathrm{Lag}^+(\C^m)\mid
        \Omega(\xi)=e^{\sqrt{-1}\theta}\}\\
    &=\{Ae_1\wedge\cdots\wedge Ae_m\mid A\in U(m),\ \textstyle\det_{\C}A=e^{\sqrt{-1}\theta}\}.
    % \cong SU(m)/SO(m).
\end{align}
% In particular, $G(\phi_0)=\{A\R^m\mid A\in SU(m)\}$ and $G(\phi_\theta)=e^{i\theta/m}G(\phi_0)$.
% $\dim G(\phi_\theta)=\frac{1}{2}(m^2+m-2)$.

Now we characterize the finite calibration energy regime in the special Lagrangian case.
Here both the Lagrangian-deviation $F^*\omega$ and the phase-deviation $\Omega(\tau)-e^{\sqrt{-1}\theta}$ need to be controlled.
Note that $F^*\Omega=\Omega(\tau)\d\mathrm{vol}$.

\begin{proposition}\label{prop:sLag_L2}
    Let $m\geq 2$, $\theta\in\R$, and $F\colon M^m\to\C^m$ be an immersion.
    Then the pointwise estimates
    \[
    \frac{1}{2\lfloor \frac{m}{2}\rfloor}|F^*\omega|^2 + \frac12\big| \Omega(\tau)-e^{\sqrt{-1}\theta} \big|^2
    \le  1-\phi_\theta(\tau) 
    \le \frac12|F^*\omega|^2 + \frac12\big| \Omega(\tau)-e^{\sqrt{-1}\theta} \big|^2
    \]
    hold on $M$.
    In particular, $\E_{\phi_\theta}[F]<\infty$ holds if and only if
    \[
    |F^*\omega|\in L^2(M)
    \quad\text{and}\quad
    \Omega(\tau)-e^{\sqrt{-1}\theta}\in L^2(M).
    \]
\end{proposition}

\begin{proof}
    Fix $p\in M$.
    Let $r\vcentcolon =\lfloor \frac{m}{2}\rfloor$.
    For simplicity, we first assume that $m$ is even.
    Then $r=\frac{m}{2}$.
    Let $S\in \operatorname{End}(T_pM)$ be defined by $g(Su,v)=F^*\omega(u,v)$.
    As $m$ is even, a similar argument to the K\"ahler case implies that there exist an orthonormal basis $(e_1,f_1,\dots,e_r,f_r)$ of $T_pM$ and numbers $0\leq \lambda_1,\dots,\lambda_r \leq 1$ such that
    \begin{equation}\label{eq:0412_03}
        F^*\omega=\sum_{\alpha=1}^r \lambda_\alpha e_\alpha^\flat \wedge f_\alpha^\flat,
        \qquad
        |F^*\omega|^2=\sum_{\alpha=1}^r \lambda_\alpha^2, \qquad 0 \leq \lambda_\alpha \leq 1.
    \end{equation}
    (Here $\lambda_\alpha=0$ is allowed since $S$ may not be invertible.)
    
    Let $R\in \C^{m \times m}$ be the complex matrix whose columns are
    \begin{equation}\label{eq:even_R_def}
        \d F(e_1),\d F(f_1),\dots,\d F(e_r),\d F(f_r).
    \end{equation}
    Then $|\Omega(\tau)|=|\det_{\C}R|$.
    Also, the Gram matrix $R^*R$ is the matrix of the pullback Hermitian form
    $g+\sqrt{-1}F^*\omega$ in the basis $(e_1,f_1,\dots,e_r,f_r)$.
    Hence,
    \begin{equation}\label{eq:even_R^*R_diagonal}
        R^*R = \bigoplus_{\alpha=1}^r
        \begin{pmatrix}
        1 & \sqrt{-1}\lambda_\alpha\\
        -\sqrt{-1}\lambda_\alpha & 1
        \end{pmatrix},
    \end{equation}
    and therefore
    \begin{equation}\label{eq:0412_04}
        \textstyle
        |\Omega(\tau)|^2
        =
        |\det_\C R|^2
        =
        \det_\C (R^*R)
        = \displaystyle
        \prod_{\alpha=1}^r(1-\lambda_\alpha^2).
    \end{equation}
    
    Next we show that, even if $m$ is odd (where $r=\frac{m-1}{2}$), the representations in \eqref{eq:0412_03} and \eqref{eq:0412_04} remain true.
    In this case, the normal form of $S$ has the same diagonal block structure with one additional zero block.
    Hence, we can take an orthonormal basis $(e_1,f_1,\dots,e_r,f_r,e')$ of $T_pM$ such that $Se'=0$ and equation \eqref{eq:0412_03} holds.
    Then, defining $R\in\C^{m \times m}$ by \eqref{eq:even_R_def} with an additional column $\d F(e')$, we obtain the same diagonal block matrix \eqref{eq:even_R^*R_diagonal} with an additional $1\times1$ block $(1)$ added.
    Thus we also reach equation \eqref{eq:0412_04}.
    
    Now we observe that the following elementary inequalities hold:
    \begin{equation}\label{eq:0412_05}
        \frac1r\sum_{\alpha=1}^r\lambda_\alpha^2
        \le
        1-\prod_{\alpha=1}^r(1-\lambda_\alpha^2)
        \le
        \sum_{\alpha=1}^r\lambda_\alpha^2.
    \end{equation}
    Indeed, using $0\leq \lambda_\alpha \leq 1$, the first inequality follows by
    \[
    \frac1r\sum_{\alpha=1}^r\lambda_\alpha^2\leq \max\lambda_\alpha^2 = 1 - (1-\max\lambda_\alpha^2) \leq 1-\prod_{\alpha=1}^r(1-\lambda_\alpha^2),
    \]
    while, using a telescoping identity argument, the second follows by
    \[
    1-\prod_{\alpha=1}^r(1-\lambda_\alpha^2)=\sum_{\alpha=1}^r\lambda_\alpha^2 \Big( \prod_{\beta=1}^{\alpha-1}(1-\lambda_\beta^2) \Big) \leq \sum_{\alpha=1}^r\lambda_\alpha^2.
    \]
    
    Estimates \eqref{eq:0412_05} combined with \eqref{eq:0412_03} and \eqref{eq:0412_04} ensure that
    \[
    \frac1r|F^*\omega|^2
    \le
    1-|\Omega(\tau)|^2
    \le
    |F^*\omega|^2.
    \]
    On the other hand, since $\phi_\theta(\tau)=\Re(e^{-\sqrt{-1}\theta}\Omega(\tau))$, we have
    \[
    1-\phi_\theta(\tau)
    =
    1-\Re(e^{-\sqrt{-1}\theta}\Omega(\tau))
    =
    \frac12|e^{\sqrt{-1}\theta}-\Omega(\tau)|^2+\frac12(1-|\Omega(\tau)|^2).
    \]
    Combining these relations, we obtain the desired estimates.
\end{proof}

\begin{remark}\label{rem:interpretation_lagrangian}
    In particular, if $F$ is Lagrangian ($F^*\omega=0$), then by \Cref{prop:sLag_L2} the special Lagrangian calibration energy $\E_{\phi_\theta}$ takes the form of a phase energy
    \begin{equation}\label{eq:phase_energy_Lagrangian}
        \mathcal{L}_{\theta}[F]=\frac{1}{2}\int_M \big| \Omega(\tau)-e^{\sqrt{-1}\theta} \big|^2 \d\mathrm{vol}.
    \end{equation}
    It is known that the Lagrangian property is preserved along the mean curvature flow \cite{smoczyk1996canonical}, even in the noncompact setting as in \Cref{thm:main_calibration} under suitable assumptions \cite{MR2299559,MR3881969}.
    Such a flow is called \emph{Lagrangian mean curvature flow} and our energy identity \eqref{eq:energy_identity} directly implies that \eqref{eq:phase_energy_Lagrangian} is a strict Lyapunov functional in this case.
\end{remark}

\begin{remark}
    If $m\in\{2,3\}$, then the estimates in \Cref{prop:sLag_L2} become identities. 
    Thus, without assuming that $F$ is Lagrangian, we have the exact energy formula
    \[
    \E_{\phi_\theta}[F] = \frac12 \int_M|F^*\omega|^2\d\mathrm{vol} + \frac12 \int_M\big| \Omega(\tau)-e^{\sqrt{-1}\theta} \big|^2\d\mathrm{vol}.
    \]
    Recall that in the case $m=2$, the calibration $\phi_\theta$ can be reduced to the K\"ahler calibration $\omega$ after pulling back by a rotation in $\R^4$. However, for $m\geq 3$, such a reduction is not possible.
\end{remark}

% The holomorphic volume form $\Omega=\Omega_I$ of the standard complex structure $I$ is written as
% \[
% \Omega_I = dz_1 \wedge dz_2 = (dx_1 \wedge dx_2 - dy_1 \wedge dy_2) + \sqrt{-1}( dx_1 \wedge dy_2 + dy_1 \wedge dx_2 ).
% \]
% Using the standard hyperk\"ahler structure $(I,J,K)$ of $\C^2$, we then have
% \[
% \Omega_I = \omega_J+\sqrt{-1}\omega_K,
% \]
% where $\omega_J$ and $\omega_K$ are the K\"ahler forms corresponding to $J$ and $K$, respectively.
% Hence
% \[
% \phi_\theta = \Re(e^{-\sqrt{-1}\theta}\Omega_I)
% =
% \cos\theta \omega_J+\sin\theta \omega_K,
% \]
% which is the K\"ahler form of the rotated complex structure $J_\theta:=\cos\theta J + \sin\theta K$.

It is an interesting open problem to classify special Lagrangian submanifolds. Even for the particular setting of entire graphs, such a complete classification is not yet available; this is in stark contrast to the K\"ahler case, where K\"ahler calibrated entire graphs are precisely given by holomorphic functions.
See, e.g., \cite{MR1686614,MR1930884,MR2199179} for rigidity and \cite{MR3508325,MR4277327,tsai2025ansatz} for construction of special Lagrangian entire graphs.

% Finally, we briefly review special Lagrangian entire graphs: in contrast to the K\"ahler case, the complete classification of such graphs is not known.
% Let $F\colon \R^m\to\R^{2m}$ be a gradient graph, i.e, for a smooth function $u\colon\R^m\to\R$,
% \[
% F(x)=(x,\D u(x)).
% \]
% Then $F$ is known to be Lagrangian.
% Letting $\lambda_1(x),\dots,\lambda_m(x)$ be the eigenvalues of the Hessian $\D^2u(x)$, we define the phase function
% \[
% \Theta_u(x)\vcentcolon =\sum_{j=1}^m \arctan\lambda_j(x).
% \]
% Then $F$ is a special Lagrangian of phase $e^{\sqrt{-1}\theta}$ if and only if $\Theta_u=\theta$ (mod $2\pi$) holds on $\R^m$, where $\theta\in(-\frac{m}{2}\pi,\frac{m}{2}\pi)$.
% The special Lagrangian equation $\Theta_u=\theta$ has trivial quadratic solutions, where $F$ are planes.
% If $m=2$, then any entire solution is harmonic ($\theta=0$) or quadratic ($\theta\neq0$) \cite{MR1686614}.
% For $m\geq3$, there are more nontrivial entire solutions, which are necessarily non-convex \cite{MR1930884}.
% More precisely, in the supercritical regime $|\theta|>\frac{m-2}{2}\pi$ the solutions $u$ are again quadratic  \cite{MR2199179}, but the critical $|\theta|=\frac{m-2}{2}\pi$ and subcritical regime $|\theta|<\frac{m-2}{2}\pi$ admit a substantially new class of solutions \cite{MR3508325,MR4277327,tsai2025ansatz}, which may not be polynomial nor harmonic, so far not classified.

\bibliography{MCF}

\end{document}